\documentclass[a4paper, 12pt, twoside, onecolumn ]{article}
\usepackage[utf8]{inputenc}
\usepackage{amsmath, amssymb, amsthm,mathrsfs}
\usepackage{enumerate} 
\usepackage{graphicx}
\usepackage[top=1in, bottom=1.3in, left=1.3in, right=1in]{geometry}
\renewcommand{\baselinestretch}{1.2}
\usepackage{lmodern}
\usepackage[all]{xy}
\usepackage{tikz-cd} 
 \usepackage{float}
 \usepackage{hyperref}
 \usepackage{accents}
 \usepackage{verbatim}
 \usepackage{import}
 \usepackage{multirow}
 \usepackage{multicol}
\date{}

\newtheorem{defn}{\bf Definition}[section] 
\newtheorem{exam}[defn]{\bf Example}

\newtheorem{lem}[defn]{\bf Lemma}
\newtheorem{thm}[defn]{\bf Theorem} 

\newtheorem{rmk}{\bf Remark}

\numberwithin{equation}{section}
\usepackage[pagewise]{lineno}

\newcommand{\footremember}[2]{%
    \footnote{#2}
    \newcounter{#1}
    \setcounter{#1}{\value{footnote}}%
}
\newcommand{\footrecall}[1]{%
    \footnotemark[\value{#1}]%
} 
\title{Finite Element Analysis of Time-Fractional Integro-differential Equation of Kirchhoff type for Non-homogeneous Materials}
\author{%
  Lalit Kumar \footremember{alley}{Department of Mathematics, Indian Institute of Technology Bombay, Mumbai-400076, India. lalitccc528@gmail.com}%
  \and Sivaji Ganesh Sista\footremember{trailer}{Department of Mathematics, Indian Institute of Technology Bombay, Mumbai-400076, India. siva@math.iitb.ac.in}%
  \and Konijeti Sreenadh\footrecall{alley} \footnote{Department of Mathematics, Indian Institute of Technology Delhi, New Delhi-110016, India. sreenadh@math.iitd.ac.in}%
  }
\date{}
\begin{document}
\maketitle

\begin{abstract}
	\noindent In this paper, we study a time-fractional initial-boundary value problem of Kirchhoff type involving memory term for non-homogeneous materials \eqref{6-10-22-1}. The energy argument is applied  to derive the a priori bounds on the solution of  the problem \eqref{6-10-22-1}. Consequently, we prove the  existence and  uniqueness of the  weak solution to  the considered problem. We keep the time variable continuous and discretize the space domain using a conforming FEM to obtain the semi discrete formulation of the problem under consideration. The semi discrete error analysis is carried out by modifying the standard Ritz-Volterra projection operator. To obtain the  numerical solution of the problem \eqref{6-10-22-1} efficiently, we develop a new linearized L1 Galerkin FEM.  This numerical scheme is shown to have a convergence rate of $O(h+k^{2-\alpha})$, where $\alpha~ (0<\alpha<1)$ is the fractional derivative exponent,  $h$ and $k$ are the  discretization parameters in the  space and time directions respectively. Further, this convergence rate is improved in the time direction by proposing a novel linearized  L2-1$_{\sigma}$  Galerkin FEM. We prove that this  numerical  scheme has an  accuracy rate of $O(h+k^{2})$.  Finally, a  numerical experiment is  conducted  to validate our theoretical claims.
	\end{abstract}

\section{Introduction}
Let  $\Omega$ be a convex and bounded subset of $\mathbb{R}^{d}~(d\geq 1)$ with smooth boundary $\partial \Omega$ and $[0,T]$ is a fixed finite time interval. We consider the following integro-differential equation of Kirchhoff type involving fractional time derivative of order $\alpha ~(0<\alpha <1)$ for non-homogeneous materials 
\begin{align*}\tag{$D_{\alpha}$}\label{6-10-22-1}
^{C}D^{\alpha}_{t}u- \nabla\cdot\left(M\left(x,t,\|\nabla u\|^{2}\right)\nabla u\right)=f(x,t)+\int_{0}^{t}b(x,t,s)u(s)~ds \quad \text{in}~ \Omega \times (0,T],
\end{align*}
with initial and boundary conditions
\begin{align*}
u(x,0)&=u_{0}(x) \quad \text{in} ~\Omega,\\
u(x,t)&=0 \quad \text{on} ~\partial \Omega \times [0,T],
\end{align*}
where  $u:=u(x,t) :\Omega \times [0,T]\rightarrow \mathbb{R} $ is the unknown function, $M:\bar{\Omega}\times [0,T]\times (0,\infty)\rightarrow (0,\infty),$ initial data $ u_{0}$, source term $f$ are known functions and $b(x,t,s)$ is a memory operator to be defined in Section \ref{12-5-4}. The notation $^{C}D_{t}^{\alpha}u$ in the problem  \eqref{6-10-22-1} is the fractional time derivative of order $\alpha$ in the Caputo sense, which is defined in \cite{podlubny1998fractional} as
\begin{equation}\label{LS1}
^{C}D_{t}^{\alpha}u =\frac{1}{\Gamma(1-\alpha)} \int_{0}^{t}\frac{1}{(t-s)^{\alpha}}\frac{\partial u}{\partial s}(s)~ds,
\end{equation}
where $\Gamma(\cdot)$ denotes the gamma function.
\par For the case $\alpha=1$, authors in  \cite{kumar2020finite} proposed a linearized backward Euler-Galerkin FEM and a linearized Crank-Nicolson-Galerkin FEM with an accuracy rate of $O(h+k)$ and $O(h+k^{2})$, respectively.
\par There are various notions of  fractional derivatives other than the Caputo derivative, which include Riemann-Liouville, Gr\"{u}nwald-Letnikov, Weyl, Marchaud, and Riesz fractional derivatives \cite{miller1993introduction, podlubny1998fractional}. Among these,  Caputo fractional derivative and Riemann-Liouville fractional derivative are the  most commonly used in the literature \cite{podlubny1998fractional}. These two fractional derivatives are related to each other by the following relation for absolutely continuous function $u$, see \cite{podlubny1998fractional}
\begin{equation}\label{1.2}
^{C}D_{t}^{\alpha}u(t)=~^{R}D_{t}^{\alpha}\left(u(t)-u(0)\right):=\partial^{\alpha}_{t}u(t),
\end{equation} 
where  $^{R}D_{t}^{\alpha}u$ is the Riemann-Liouville fractional derivative defined by 
\begin{equation}
^{R}D_{t}^{\alpha}u(t) =\frac{1}{\Gamma(1-\alpha)}\frac{d}{dt}\int_{0}^{t}\frac{1}{(t-s)^{\alpha}}u(s)~ds .
\end{equation}
It is worth noting that  fractional derivatives other than the
 Caputo fractional derivative  require initial condition containing the limiting value of fractional derivative at $t=0$ \cite{podlubny1998fractional}, which has no physical interpretation. An advantage of choosing Caputo fractional derivative in the problem \eqref{6-10-22-1} is that it allows the initial and boundary conditions in the same way as those for integer order differential equations. 
 \par There are many physical and biological  processes in which the mean-squared displacement of the particle motion grows only sublinearly with time $t$, instead of linear growth. For instance, acoustic wave propagation in viscoelastic materials \cite{mainardi2010fractional}, cancer invasion system \cite{manimaran2019numerical}, anomalous diffusion transport \cite{metzler2000random},  which cannot be  described  accurately by classical  models having integer order derivatives.  Therefore the study of fractional differential  equations has  evolved immensely in  recent years.
 \par Mathematical problems involving fractional time derivatives  have been studied by many researchers, for instance, see   \cite{giga2017well, huang2005time, schneider1989fractional}. 
Analytical solutions of fractional differential equations are expressed in terms of Mittag-Leffler function, Fox $H$-functions, Green functions, and hypergeometric functions. Such special functions are more complex to compute, which restrict the applications of fractional calculus in applied sciences. This motivates the researchers to develop numerical algorithms for solving fractional differential equations.
\par There are two predominant discretization techniques in time for fractional differential equations, namely Gr\"{u}nwald-Letnikov approximation \cite{dimitrov2013numerical} and L1 type approximation \cite{alikhanov2015new,lin2007finite}. The second category, viz. the L1 type approximation scheme is based on piecewise interpolation of the integrand in the definition \eqref{LS1} of the Caputo fractional derivative. Lin and Xu in  \cite{lin2007finite} developed the L1 scheme based on piecewise linear interpolation for Caputo fractional derivative and Legendre spectral method in space for the following time-fractional PDE in one space dimension
\begin{equation}\label{1.3}
\begin{aligned}
^{C}D^{\alpha}_{t}u-\frac{\partial^{2}u}{\partial x^{2}}&=f(x,t)\quad x \in (0,1),~ t \in (0,T],\\
u(x,0)&=g(x) \quad x \in (0,1),\\
u(0,t)=u(1,t)&=0 \quad 0 \leq t \leq T,
\end{aligned}
\end{equation} 
and achieved the convergence estimates of $O(h^{2}+k^{2-\alpha})$ for  solutions belonging to $C^{2}\left([0,T];H^{2}(\Omega)\cap H^{1}_{0}(\Omega)\right)$. Recently, Alikhanov  \cite{alikhanov2015new} proposed a modification of the L1 type scheme in the time direction and difference scheme in the space direction for some linear extension  the problem \eqref{1.3}. In his work, the  author proved that the convergence rate is of $O(h^{2}+k^{2})$ for solutions belonging to $C^{3}\left([0,T];C^{4}(\Omega)\right)$.
\par  On a similar note, there has been considerable attention devoted to the nonlocal diffusion problems where  diffusion coefficient depends on the entire domain rather than pointwise. Lions \cite{lions1978some} studied the following problem 
\begin{equation*}
    \frac{\partial^{2}u}{\partial t^{2}}-M\left(x,\int_{\Omega}|\nabla u|^{2}dx\right)\Delta u=f(x,t) \quad \text{in}\quad \Omega \times [0,T],
\end{equation*}
which models transversal oscillations of an elastic string or membrane by considering the change in length during vibrations. Also, the nonlocal terms appear in various physical and biological systems. For instance, the temperature in a thin region during friction welding \cite{kavallaris2007behaviour}, Kirchhoff equations with magnetic field \cite{mingqi2018critical}, Ohmic heating with variable thermal conductivity \cite{tzanetis2001nonlocal}, and many more. We cite \cite{correa2004existence, goel2019kirchhoff, mishra2016existence} for some contemporary works related to the  existence, uniqueness, and regularity of the problems involving Kirchhoff type diffusion coefficient.
\par The models discussed above behave accurately only for a perfectly homogeneous medium, but in real-life situations, a large number of heterogeneities are present, which cause some memory effect or feedback term. These phenomena cannot be described by classical PDEs, which motivate us to study  time-fractional PDEs for non-homogeneous materials. We note  that this class of equations has not been analyzed in the literature yet, and this is the first attempt to establish new results for the problem \eqref{6-10-22-1}.
 \par  The salient features of the considered problem  is its doubly nonlocal nature due to the presence of Kirchhoff term and fractional derivative or  memory term. The appearance of the Kirchhoff term creates major difficulties in the theoretical as well as in numerical analysis of the problem under consideration. We cannot apply Laplace/Fourier transformation in the problem \eqref{6-10-22-1}, therefore explicit representation of its solution in terms of Fourier expansion is not possible. To resemble this issue, we use Galerkin method to show the well-posedness of the weak formulation of the problem \eqref{6-10-22-1}.
 \par The fully discrete formulation of the considered problem produces a system of nonlinear algebraic equations. In general, numerical schemes based on the Newton method are adopted to solve this system \cite{gudi2012finite}. The Kirchhoff term leads to the highly non-sparse Jacobian of this system \cite{gudi2012finite}. As a result of which we require high computational cost as well as  huge computer storage for solving this system. We reduce these costs by developing a new linearization technique for the nonlinearity. 
 \par On the other hand, the memory term incorporates the history of the phenomena under investigation by virtue of  which we need to store  the value of approximate solution at all previous time steps. This process demands large computer memory. We overcome this difficulty by discretizing the memory term using modified trapezoidal or modified Simpson's rule \cite{pani1992numerical}.
 \par To prove the well-posedness of the weak formulation of the problem \eqref{6-10-22-1}, we reduce the weak formulation onto a finite dimensional subspace of $H^{1}_{0}(\Omega)$. The theory of fractional differential equations \cite{diethelm2002analysis} ensures the existence of Galerkin sequence of weak solutions. The a priori bounds on these Galerkin sequences are attained by employing the energy argument. We make use of these a priori bounds in  Aubin-Lions type compactness lemma \cite{li2016analysis} to  prove that the Galerkin sequence converges to the weak solution of the problem \eqref{6-10-22-1}.
 \par To determine the semi discrete error estimates, we alter the definition of standard Ritz-Volterra projection operator \cite{cannon1988non} so that it reduces the complications caused by  the Kirchhoff term. This modified Ritz-Volterra projection operator follows the best approximation properties same as that of standard one. These best approximation properties play a key role in deriving the semi discrete error estimates.  
 \par To obtain the numerical solution,  we construct two fully discrete formulations for the problem \eqref{6-10-22-1} by discretizing the space domain using a conforming FEM \cite{thomee2007galerkin} and the time direction by uniform mesh.  First, we develop a new linearized L1 Galerkin FEM. This method comprises of L1 type approximation \cite{lin2007finite} of the Caputo fractional derivative, linearization technique  for the Kirchhoff type nonlinearity, and modified Simpson's rule \cite{pani1992numerical} for approximation of the memory term. We acquire the a priori bounds on the solution of this numerical scheme and show that this numerical scheme is accurate of $O(h+k^{2-\alpha})$.
 \par Further, we increase the accuracy of this scheme in the time direction by  replacing the L1 scheme with the L2-1$_{\sigma}$ scheme \cite{alikhanov2015new} for the approximation of the Caputo fractional derivative. As a  consequence, we propose a new linearized L2-1$_\sigma$ Galerkin FEM which has a convergence rate of $O(h+k^{2})$. These numerical results are supported by conducting a numerical experiment in MATLAB software.
\par \textbf{Turning to the layout of this paper:} In Section \ref{12-5-4}, we provide  some notations, assumptions, and  preliminaries results that will be used throughout this work. In Section \ref{11-10-22-52}, we state  main contributions of this article. Section \ref{11-10-22-54} contains  the proof of well-posedness of the weak formulation of  the problem \eqref{6-10-22-1}. In Section \ref{11-10-22-55}, we define semi discrete formulation of the considered problem  and derive a priori bounds as well as  error estimates on semi discrete solutions. In Section \ref{11-10-22-56}, we develop  a new linearized  L1 Galerkin FEM. We  derive a priori bounds on numerical solutions of the developed numerical scheme and prove its accuracy rate of $O(h+k^{2-\alpha})$. In Section \ref{11-10-22-57}, we achieve improved convergence rate of $O(h+k^{2})$ by proposing a new linearized  L2-1$_\sigma$ Galerkin FEM. Section \ref{11-10-22-58} includes a numerical experiment that confirms the sharpness of theoretical results. Finally, we conclude this work in Section \ref{11-10-22-59}.

\section{Preliminaries}\label{12-5-4}
\noindent  Let $L^{1}(\Omega)$ be the set of all equivalence classes of the integrable functions on $\Omega$ with the norm 
\begin{equation}
    \|g\|_{L^{1}(\Omega)}=\int_{\Omega}|g(x)|~dx~~\text{for}~~g \in L^{1}(\Omega).
\end{equation}
Let $L^{2}(\Omega)$ be the set of all equivalence classes of the square integrable functions on $\Omega$ with the norm 
\begin{equation}\label{12-5-2}
    \|g\|^{2}=\int_{\Omega}|g(x)|^{2}~dx~~\text{for}~~g \in L^{2}(\Omega).
\end{equation}
The norm  defined in \eqref{12-5-2} is induced by the inner product $(\cdot,\cdot)$ as follows 
\begin{equation}
    (g,h)=\int_{\Omega}g(x)h(x)~dx~~\text{for}~~g,h \in L^{2}(\Omega).
\end{equation}
The Sobolev space $W^{1,1}(\Omega)$ is the collection of all functions in $L^{1}(\Omega)$ such that its distributional derivative of order one is also in $L^{1}(\Omega)$, i.e.,
\begin{equation}
    W^{1,1}(\Omega)=\left\{ g \in L^{1}(\Omega) ;~  D g \in L^{1}(\Omega)\right\}.
\end{equation}
The norm on the space $W^{1,1}(\Omega)$ is given by 
\begin{equation}
    \|g\|_{W^{1,1}(\Omega)}=\|g\|_{L^{1}(\Omega)}+\|D g\|_{L^{1}(\Omega)}~~\text{for}~~g \in W^{1,1}(\Omega).
\end{equation}
The sobolev space $H^{m}(\Omega),\left(m \in \{1,2\}\right)$ is  the set of all functions in $L^{2}(\Omega)$ such that its distributional derivatives upto order $m$ are also in $L^{2}(\Omega)$, i.e.,
\begin{equation}
H^{m}(\Omega)=\left\{g \in L^{2}(\Omega) ;~ D^{\beta}g \in L^{2}(\Omega),~ |\beta|\leq m \right\},
\end{equation}
where $\beta$ is multiindex. The norm on the space $H^{m}(\Omega)$  is induced by the following inner product $(\cdot,\cdot)_{m}$ as follows 
\begin{equation}
    (g,h)_{m}=\sum_{|\beta|\leq m}(D^{\beta}g,D^{\beta}h)~~\text{for}~~g,h \in H^{m}(\Omega).
\end{equation}
We denote $H^{m}_{0}(\Omega),\left(m \in \{1,2\}\right)$ be the closure of $C^{\infty}_{C}(\Omega)$ in $H^{m}(\Omega)$. The space $H^{m}_{0}(\Omega)$ can be characterised by the functions in $H^{m}(\Omega)$ having zero trace on the boundary $\partial \Omega$ \cite[Section 2.7]{kesavan1989topics}. The dual space  of the $H^{m}_{0}(\Omega)$ is denoted by $H^{-m}(\Omega)$.
\par For any  Hilbert space $X$, we denote $L^{2}(0,T;X)$ be the set of all measurable functions $g:[0,T]\rightarrow X$ such that
\begin{equation}
 \int_{0}^{T}\|g(s)\|^{2}_{X}~ds < \infty.   
\end{equation} 
The norm on the space $L^{2}(0,T;X)$  is given by 
\begin{equation}
    \|g\|^{2}_{L^{2}(0,T;X)}=\int_{0}^{T}\|g(s)\|_{X}^{2}~ds~~\text{for}~~g \in L^{2}(0,T;X).
\end{equation}
We also define a weighted $L^{2}_{\alpha}(0,T;X)$  space consisting of all measurable functions $g:[0,T]\rightarrow X$ such that
\begin{equation}
    \sup_{t \in (0,T)}\left(\frac{1}{\Gamma(\alpha)}\int_{0}^{t}(t-s)^{\alpha-1}\|g(s)\|^{2}_{X}~ds\right) < \infty.
\end{equation}
The norm on the space $L^{2}_{\alpha}(0,T;X)$  is given by  \cite[(4.5)]{li2018some}
\begin{equation}
    \|g\|^{2}_{L^{2}_{\alpha}(0,T;X)}=\sup_{t \in (0,T)}\left(\frac{1}{\Gamma(\alpha)}\int_{0}^{t}(t-s)^{\alpha-1}\|g(s)\|^{2}_{X}~ds\right)~~\text{for}~~g \in L^{2}_{\alpha}(0,T;X).
\end{equation}
One can observe that $L^{2}_{\alpha}(0,T;X)  \subset  L^{2}(0,T;X)$. The set of all measurable functions $g:[0,T]\rightarrow X$ such that
\begin{equation}
  \text{ess}\sup_{t\in (0,T) }\|g(t)\|_{X} < \infty  
\end{equation}
is denoted by $L^{\infty}(0,T;X)$. The norm on this space is given by 
\begin{equation}
    \|g\|_{L^{\infty}(0,T;X)}=\text{ess}\sup_{t\in (0,T) }\|g(t)\|_{X}~~\text{for}~~g \in L^{\infty}(0,T;X).
\end{equation}
\noindent For any two quantities $a$ and $b$, the notation $a\lesssim b$ means that there exists a generic positive constant $C$ such that $a\leq Cb$, where $C$ depends on data but independent of discretization parameters and may vary at different occurrences.
\par Throughout the paper, we assume the following hypotheses on data:\\
\noindent \text{(H1)}\quad  Initial data $u_{0}\in H^{2}(\Omega)\cap H^{1}_{0}(\Omega)$ and source term  $f \in L^{\infty}\left(0,T;L^{2}(\Omega)\right)$.\\
\noindent \text{(H2)}\quad Diffusion coefficient $M: \bar{\Omega} \times [0,T]\times (0,\infty)\rightarrow (0,\infty)$ is a Lipschitz continuous function such that  there exists a constant $m_{0}$ which satisfies 
\begin{equation*}\label{1.1}
M(x,t,s)\geq m_{0} >0 ~\text{for all }~ (x,t,s) \in \bar{\Omega} \times [0,T]\times (0,\infty) ~\text{and} \left(m_{0}-4L_{M}K^{2}\right)>0,
\end{equation*}
where $K=\left(\|\nabla u_{0}\|+\|f\|_{L^{\infty}\left(0,T;L^{2}(\Omega)\right)}\right)$ and $L_{M}$ is a Lipschitz constant.\\
\noindent \text{(H3)}\quad  Memory operator $b(x,t,s)$ is a second order partial differential operator of the form  
\begin{equation*}
b(x,t,s)u(s):=-\nabla\cdot(b_{2}(x,t,s)\nabla u(s))+\nabla \cdot(b_{1}(x,t,s)u(s))+b_{0}(x,t,s)u(s),
\end{equation*}
with $b_{2}:\bar{\Omega}\times[0,T]\times [0,T] \rightarrow \mathbb{R}^{d\times d}$ is a symmetric and positive definite  matrix with entries $[b_{2}^{ij}(x,t,s)]$, $b_{1}:\bar{\Omega} \times[0,T]\times [0,T] \rightarrow \mathbb{R}^{d}$ is a vector with entries $[b_{1}^{j}(x,t,s)]$ and  $b_{0}: \bar{\Omega}\times [0,T]\times [0,T] \rightarrow \mathbb{R}$ is a scalar function. We assume that $b_{2}^{ij}, b_{1}^{j}, b_{0}$  are smooth functions  in all variables $(x,t,s)\in \bar{\Omega}\times [0,T]\times [0,T]$ for $i,j=1,2,\dots,d$.
\par We define a function $B(t,s,u(s),v)$ for all $t,s$ in $[0,T]$ and for all $u(s),v$ in $H^{1}_{0}(\Omega)$ as follows 
\begin{equation}
    B(t,s,u(s),v):=(b_{2}(x,t,s)\nabla u(s),\nabla v)+(\nabla\cdot(b_{1}(x,t,s)u(s)),v)+(b_{0}(x,t,s)u(s),v).
\end{equation}
Using (H3) and Poincar\'{e} inequality we can prove that there exists a positive constant $B_{0}$ such that for all $t,s$ in $[0,T]$ and for all $u,v$ in $H^{1}_{0}(\Omega)$ we have 
\begin{equation}\label{1.1a}
    |B(t,s,u(s),v)|\leq B_{0}\|\nabla u(s)\|~\|\nabla v\|.
\end{equation}
We denote 
\begin{equation}\label{2.1}
k(t):=\frac{t^{-\alpha}}{\Gamma(1-\alpha)},
\end{equation}
and $\ast$ indicates the convolution of two  integrable functions $g$ and $h$ on $[0,T]$ as 
\begin{equation}\label{354}
    (g\ast h)(t)=\int_{0}^{t}g(t-s)h(s)~ds \quad \text{for all}~~t~~\text{in}~~[0,T].
\end{equation}
\begin{rmk}\label{rmk1}
	Note that $l(t)$ defined by  $l(t):=\frac{t^{\alpha-1}}{\Gamma(\alpha)}$ satisfies $k\ast l=1$.
\end{rmk}
\begin{lem}\cite[Lemma 18.4.1]{gripenberg1990volterra} \label{12-5-5} Let $H$ be a real Hilbert space and $T>0$. Then for any $\tilde{k} \in W^{1,1}(0,T)$ and  $v \in L^{2}(0,T;H)$ we have
\begin{equation}
\begin{aligned}
    \left(\frac{d}{dt}\left(\tilde{k}\ast v\right)(t),v(t)\right)_{H}&=\frac{1}{2}\frac{d}{dt}\left(\tilde{k}\ast \|v\|^{2}_{H}\right)(t)+\frac{1}{2}\tilde{k}(t)\|v\|_{H}^{2}\\
    &+\frac{1}{2}\int_{0}^{t}\left[-\tilde{k}'(s)\right]\|v(t)-v(t-s)\|_{H}^{2}~ds~~\text{a.e.}~~t \in (0,T).
    \end{aligned}
\end{equation}
\end{lem}

\begin{lem}\label{lem2.5}  \cite[Theorem 8]{almeida2017gronwall}
	Let $u,v$ be two nonnegative integrable functions on $[a,b]$ and $g$ a continuous function in $[a,b]$. Assume that  $v$ is nondecreasing in $[a,b]$ and $g$ is nonnegative and  nondecreasing in $[a,b]$. If
	\begin{equation*}
	u(t)\leq v(t)+g(t)\int_{a}^{t}(t-s)^{\alpha-1}u(s)~ds~~\text{for}~\alpha \in (0,1)~\text{and}~\forall~t \in  [a,b],
	\end{equation*}
	then 
	\begin{equation*}
	u(t)\leq v(t)E_{\alpha}\left[g(t)\Gamma(\alpha)(t-a)^{\alpha}\right]~~\text{for}~\alpha \in (0,1)~\text{and}~\forall~t \in  [a,b],
	\end{equation*}
	where $E_{\alpha}(\cdot)$ is the one parameter Mittag-Leffler function \cite[Section 1.2]{podlubny1998fractional}.
\end{lem}

\begin{lem}\label{13-5-1}\cite{diethelm2002analysis} Consider the following initial value problem 
\begin{equation}\label{13-5-2}
\begin{aligned}
    \partial^{\alpha}_{t}y(t)&=g(t,y(t)),~ t \in (0,T],~\alpha \in (0,1),\\
    y(0)&=y_{0}.
    \end{aligned}
\end{equation}
Let $y_{0}\in \mathbb{R}, K^{\ast}>0, t^{\ast}>0$. Define $D=\left\{(t,y(t));~t\in [0,t^{\ast}],~|y-y_{0}|\leq K^{\ast}\right\}$. Let function $g:D\rightarrow \mathbb{R}$ be a continuous. Define $M^{\ast}=\sup_{(t,y(t))\in D}|g(t,y(t)|.$  Then there exists a continuous function $y\in C[0,T^{\ast}]$ which solves the problem \eqref{13-5-2}, where 
\begin{equation}
T^{\ast}=\begin{cases} t^{\ast};&~~~~M^{\ast}=0,\\
\min\{t^{\ast}, \left(\frac{K^{\ast}\Gamma(1+\alpha)}{M^{\ast}}\right)^{\frac{1}{\alpha}}\};&~~~~\text{else}.
\end{cases}
\end{equation}
\end{lem}

\begin{lem}\label{13-5-3} \cite[Lemma 4.1]{li2018some}
For $T>0$ and $\alpha \in (0,1)$. Let $X,Y,$ and $Z$ be the Banach spaces such that $X$ is compactly embedded in $Y$ and $Y$ is continuously embedded in $Z$. Suppose that $W\subset L^{1}_{\text{loc}}(0,T;X)$ satisfies the following 
\begin{enumerate}
    \item There exist a constant $C_{1}>0$ such that for all $u\in W$
    \begin{equation}
       \sup_{t \in (0,T)}\left(\frac{1}{\Gamma(\alpha)}\int_{0}^{t}(t-s)^{\alpha-1}\|u(s)\|^{2}_{X}~ds\right) \leq  C_{1}. 
    \end{equation}
    \item There exists a constant $C_{2}>0$ such that for all $u\in W$
    \begin{equation}
        \|\partial^{\alpha}_{t}u\|_{L^{2}(0,T;Z)}\leq C_{2}.
    \end{equation}
\end{enumerate}
Then $W$ is relatively compact in $L^{2}(0,T;Y)$.
\end{lem}

\begin{lem} \cite{zacher2009weak} \label{L3.1}Let $k$ be the kernel defined in \eqref{2.1} then there exists a sequence of kernels $k_{n}$ in $W^{1,1}(0,T)$ such that $k_{n}$ is nonnegative and nonincreasing in $(0,\infty)$. Also 
\begin{equation}
    k_{n} \rightarrow k ~\text{in} ~L^{1}(0,T)~ \text{as}~ n \rightarrow \infty,
\end{equation}
and 
\begin{equation}
    \frac{d}{dt}\left(k_{n}\ast u\right)\rightarrow \frac{d}{dt}(k\ast u)~\text{in} ~L^{2}(0,T;L^{2}(\Omega))~ \text{as}~ n \rightarrow \infty.
\end{equation}
\end{lem} 

\section{Main results}\label{11-10-22-52}
\noindent  The connection  between Caputo fractional derivative and Riemann-Liouville fractional derivative  reduces the problem  \eqref{6-10-22-1} into  
\begin{equation}\tag{$R_{\alpha}$}\label{6-10-22-2}
\begin{aligned}
\partial^{\alpha}_{t}u- \nabla\cdot\left(M\left(x,t,\|\nabla u\|^{2}\right)\nabla u\right)&=f(x,t)+\int_{0}^{t}b(x,t,s)u(s)~ds~\text{in}~\Omega \times (0,T],
\end{aligned}
\end{equation}
with initial and boundary conditions
\begin{equation*}
    \begin{aligned}
    u(x,0)&=u_{0}(x) \quad \text{in} ~\Omega,\\
    u(x,t)&=0 \quad \text{on} ~\partial \Omega \times [0,T].\\
    \end{aligned}
\end{equation*}
\noindent The notion of weak formulation for the problem \eqref{6-10-22-2} is described as follows:\\
find 
$u$ in $L^{\infty}\left(0,T;L^{2}(\Omega)\right) \cap L^{2}\left(0,T;H^{1}_{0}(\Omega)\right)$ and $\partial^{\alpha}_{t}u$ in $L^{2}\left(0,T;{L^{2}(\Omega)}\right) $ such that the following equations hold for all $v$ in $H^{1}_{0}(\Omega)$ and $a.e. ~t$ in $(0,T]$
\begin{equation}\tag{$W_{\alpha}$}\label{6-10-22-3}
\begin{aligned}
\left(\partial^{\alpha}_{t}u,v\right)+ \left(M\left(x,t,\|\nabla u\|^{2}\right)\nabla  u,\nabla v\right)
&=\left(f,v\right)+\int_{0}^{t}B(t,s,u(s),v)~ds,~\text{in}~\Omega \times (0,T],\\
u(x,0)&=u_{0}(x) \quad \text{in} ~\Omega.
\end{aligned}
\end{equation}
\begin{thm}\label{t2.5}
	\textbf{(Well-posedness of the weak formulation \eqref{6-10-22-3})} Under the hypotheses (H1), (H2), and (H3)  the problem \eqref{6-10-22-3} admits a unique  solution  that satisfies the following a priori bounds
	\begin{equation}\label{t2.5-1}
	\|u\|_{L^{\infty}\left(0,T;L^{2}(\Omega)\right)}+	\|u\|_{L^{2}_{\alpha}\left(0,T;H^{1}_{0}(\Omega)\right)}\lesssim \left(\|\nabla u_{0}\|+\|f\|_{L^{\infty}(0,T;L^{2}(\Omega))}\right),
	\end{equation}
	\begin{equation}\label{t2.5-2}
	\|u\|_{L^{\infty}(0,T;H^{1}_{0}(\Omega))}+	\|u\|_{L^{2}_{\alpha}\left(0,T;H^{2}(\Omega)\right)} \lesssim \left(\|\nabla u_{0}\|+\|f\|_{L^{\infty}(0,T;L^{2}(\Omega))}\right).
	\end{equation} 
\end{thm}
\noindent For the semi discrete formulation of the problem  \eqref{6-10-22-1}, we discretize the domain in the space variable by a conforming FEM \cite{thomee2007galerkin} and keep the time direction continuous. Let $\mathbb{T}_{h}$ be a shape regular (non overlapping), quasi-uniform triangulation of the domain $\Omega$ and $h$ be the discretization parameter in  the space direction. We define a  finite dimensional subspace $X_{h}$ of $H^{1}_{0}(\Omega)$ as 
 $$X_{h}:=\{v_{h}\in C(\bar \Omega)~:~v_{h}|_{\tau}~ \text{is a linear polynomial for all}~ \tau \in \mathbb{T}_{h} ~\text{and} ~v_{h}=0 ~\text{on}~\partial \Omega\}.$$
 \par The semi discrete formulation for the problem \eqref{6-10-22-1} is to seek $u_{h}$ in $X_{h}$ such that the following equations hold for all $v_{h}$ in $X_{h}$ and $a.e.$ $t$ in $(0,T]$ 
 \begin{equation}\tag{$S_{\alpha}$}\label{6-10-22-4}
 \begin{aligned}
 \left(\partial^{\alpha}_{t}u_{h},v_{h}\right) &+\left(M\left(x,t,\|\nabla u_{h}\|^{2}\right)\nabla u_{h},\nabla v_{h}\right)\\
 &=(f,v_{h})+\int_{0}^{t}B(t,s,u_{h}(s),v_{h})~ds~\text{in}~\mathbb{T}_{h}\times (0,T],\\
 u_{h}(x,0)&=u_{h}^{0} ~\text{in}~\mathbb{T}_{h},
 \end{aligned}
 \end{equation}
 where initial condition $u_{h}^{0}$ is in $X_{h}$ which  will be chosen later in the proof of Theorem \ref{t2.6}.
\begin{thm}\label{t2.6}
	\textbf{(Error estimate for the semi discrete  formulation  \eqref{6-10-22-4})} Suppose that hypotheses (H1), (H2), and (H3) hold. Then  we have the following error estimate  for the solution  $u_{h}$ 
of the semi discrete scheme \eqref{6-10-22-4}
\begin{equation}
	\|u-u_{h}\|_{L^{\infty}\left(0,T;L^{2}(\Omega)\right)}+\|u-u_{h}\|_{L^{2}_{\alpha}\left(0,T;H^{1}_{0}(\Omega)\right)} \lesssim  h,
	\end{equation}
	provided that $u(t)$ is in $H^{2}(\Omega)\cap H^{1}_{0}(\Omega)$ for a.e. $t$ in $[0,T]$.
\end{thm}
\noindent Further,  we move to the fully discrete formulation of the problem \eqref{6-10-22-1} for that we divide the interval $[0,T]$  into sub intervals of uniform step size $k$ and $t_{n}=nk$ for $n=0,1,2,3,\dots,N$ with $t_{N}=T$. We approximate  the Caputo fractional derivative by L1 scheme, Kirchhoff type nonlinearity by linearization, and memory term by modified Simpson's rule as follows\\
\noindent  \textbf{L1 approximation scheme \cite{li2016analysis}:} In this scheme, Caputo fractional derivative is approximated at the point $t_{n}$ using linear interpolation or backward Euler difference formula  as follows
\begin{equation}\label{2.9}
\begin{aligned}
^{C}D^{\alpha}_{t_{n}}u~&=\frac{1}{\Gamma(1-\alpha)}\int_{0}^{t_{n}}\frac{1}{(t_{n}-s)^{\alpha}}\frac{\partial u}{\partial s}ds\\
&=\frac{1}{\Gamma(1-\alpha)}\sum_{j=1}^{n}\frac{u^{j}-u^{j-1}}{k}\int_{t_{j-1}}^{t_{j}}\frac{1}{(t_{n}-s)^{\alpha}}~ds+\mathbb{Q}^{n}\\
&=\frac{k^{-\alpha}}{\Gamma(2-\alpha)}\sum_{j=1}^{n}a_{n-j}\left(u^{j}-u^{j-1}\right)+\mathbb{Q}^{n}\\
&=\mathbb{D}^{\alpha}_{t}u^{n}+\mathbb{Q}^{n}
\end{aligned}
\end{equation}
where $a_{i}=(i+1)^{1-\alpha}-i^{1-\alpha},~i\geq0$, $u^{j}=u(x,t_{j})$, and $\mathbb{Q}^{n}$ is the truncation error. \\
\textbf{Linearization:} For nonlinear term we use the following linearized approximation of $u$ at $t_{n}$ given by 
\begin{equation}\label{7-10-22-1}
\begin{aligned}
    u^{n}&\approx 2u^{n-1}-u^{n-2},~\text{for}~n~\geq~2\\
    &:=\bar{u}^{n-1}.
    \end{aligned}
\end{equation}
\noindent \textbf{Modified Simpson's rule \cite{pani1992numerical}:} Let $m_{1}=[k^{-1/2}]$, where $[\cdot]$ denotes the greatest integer  function. Set $k_{1}=m_{1}k$ and $\bar{t}_{j}=jk_{1}$. Let $j_{n}$ be the largest even integer such that $\bar{t}_{j_{n}} <t_{n}$ and introduce quadrature points 
 \[
\bar{t}_{j}^{n}=\begin{cases}
 &jk_{1},\quad 0\leq j\leq j_{n},\\
 &\bar{t}_{j}^{n}+(j-j_{n})k,\quad j_{n}\leq j\leq J_{n},
 \end{cases}
 \]
 where $\bar{t}_{J_{n}}^{n}=t_{n-1}$.
 Then quadrature rule for any function $g$ is as follows
\begin{equation}\label{S1.14}
\begin{aligned}
\int_{0}^{t_{n}}g(s)~ds&=\sum_{j=0}^{n-1}w_{nj}g(t_{j})+q^{n}(g)\\
&=\frac{k_{1}}{3}\sum_{j=1}^{j_{n}/2}\left[g(\bar{t}_{2j}^{n})+4g(\bar{t}_{2j-1}^{n})+g(\bar{t}_{2j-2}^{n})\right]\\
&+\frac{k}{2}\sum_{j=j_{n}+1}^{J_{n}}\left[g(\bar{t}_{j}^{n})+g(\bar{t}_{j-1}^{n})\right]+kg(\bar{t}_{J_{n}}^{n})+q^{n}(g),
\end{aligned}
\end{equation}
where $w_{nj}$ are called quadrature weights and $q^{n}(g)$ is the quadrature error associated with the function $g$ at $t_{n}$.
\par On the basis of approximations  \eqref{2.9}, \eqref{7-10-22-1}, and  \eqref{S1.14} we develop the following linearized L1 Galerkin FEM.\\
\noindent \textbf{ Linearized L1 Galerkin FEM:} Find $u_{h}^{n} ~(n=1,2,3,\dots,N)$ in $X_{h}$ with $\bar{u}_{h}^{n-1}=2u_{h}^{n-1}-u_{h}^{n-2}$ such that the following equations hold for all $v_{h}$ in $X_{h}$ \\
For $n\geq 2,$
\begin{equation}\tag{$E_{\alpha}$}\label{7-10-22-2}
\left(\mathbb{D}^{\alpha}_{t}u_{h}^{n},v_{h}\right)+\left(M\left(x,t_{n},\|\nabla \bar{u}_{h}^{n-1}\|^{2}\right)\nabla u_{h}^{n},\nabla v_{h}\right)=\left(f^{n},v_{h}\right)+\sum_{j=1}^{n-1}w_{nj}B\left(t_{n},t_{j},u_{h}^{j},v_{h}\right).
\end{equation}
For $n=1$,
\begin{equation*}
    \begin{aligned}
\left(\mathbb{D}^{\alpha}_{t}u_{h}^{1},v_{h}\right)+\left(M\left(x,t_{1},\|\nabla u_{h}^{1}\|^{2}\right)\nabla u_{h}^{1},\nabla v_{h}\right)=\left(f^{1},v_{h}\right)+kB\left(t_{1},t_{0},u_{h}^{0},v_{h}\right),
\end{aligned}
\end{equation*}
with  initial condition $u_{h}^{0}$ that is to be chosen later in the proof of Theorem \ref{t2.6}.\\ 
\noindent To access the convergence rate of the developed numerical scheme \eqref{7-10-22-2}, we need the following discrete kernel corresponding to the kernel $(a_{j})$
\begin{lem}\label{L1.2} \cite{li2016analysis}
	Let $p_{n}$ be a sequence defined by 
	\begin{equation*}
	p_{0}=1,~~	p_{n}=\sum_{j=1}^{n}(a_{j-1}-a_{j})p_{n-j}~~\text{for}~~n\geq 1.
	\end{equation*}
	Then $p_{n}$ satisfies
		 \begin{equation}
0<p_{n}<1,
		\end{equation} 
		\begin{equation}\label{2.81}
		\sum_{j=k}^{n}p_{n-j}a_{j-k}=1,\quad 1\leq k \leq n,
		\end{equation}
		\begin{equation}\label{2.91}
		\Gamma(2-\alpha)\sum_{j=1}^{n}p_{n-j}\leq \frac{n^{\alpha}}{\Gamma(1+\alpha)}.
		\end{equation}	
\end{lem}
\begin{thm}\label{J1}
	\textbf{(Convergence estimate for the numerical scheme \eqref{7-10-22-2})} Under the hypotheses (H1), (H2), and (H3) the fully discrete solution  $u_{h}^{n}~(1\leq n \leq N)$ of the scheme \eqref{7-10-22-2} converges to the solution $u$ of the problem \eqref{6-10-22-1} with the following rate of accuracy
	\begin{equation}\label{11-10-22-49}
	\max_{1\leq n \leq N}\|u(t_{n})-u_{h}^{n}\|+\left(k^{\alpha}\sum_{n=1}^{N}p_{N-n}\|\nabla u(t_{n}) -\nabla u_{h}^{n}\|^{2}\right)^{1/2}\lesssim (h+k^{2-\alpha}).
	\end{equation}
\end{thm}
\noindent  At this point one can see that convergence rate is of $O(k^{2-\alpha})$ in the temporal direction. To improve this convergence rate a new linearized fractional Crank-Nicolson-Galerkin FEM is proposed. In this scheme we replace the L1 approximation of the Caputo fractional derivative with L2-1$_{\sigma}$~$(\sigma=\frac{\alpha}{2})$ scheme \cite{alikhanov2015new}  at $t_{n-\sigma}$~$(t_{n-\sigma}=(1-\sigma)t_{n}+\sigma t_{n-1})$, linearization technique for nonlinearity at $t_{n-\sigma}$, and modified Simpson's rule for the memory term at $t_{n-\sigma}$.\\
\noindent \textbf{L2-1$_{\sigma}$ approximation scheme \cite{alikhanov2015new}:} In this scheme, Caputo fractional derivative is approximated at the point $t_{n-\sigma}$  as  follows

\begin{equation}\label{1.20}
^{C}D^{\alpha}_{t_{n-\sigma}}=\tilde{\mathbb{D}}^{\alpha}_{t_{n-\sigma}}u+\tilde{\mathbb{Q}}^{n-\sigma},
\end{equation}
where 
\begin{equation} \tilde{\mathbb{D}}^{\alpha}_{t_{n-\sigma}}u=\frac{k^{-\alpha}}{\Gamma(2-\alpha)}\sum_{j=1}^{n}\tilde{c}_{n-j}^{(n)}\left(u^{j}-u^{j-1}\right), 
\end{equation}
with weights $\tilde{c}_{n-j}^{(n)}$ satisfying $\tilde{c}_{0}^{(1)}=\tilde{a}_{0}$ for $n=1$ and for $n\geq 2$
\begin{equation}
\tilde{c}_{j}^{(n)}=\begin{cases}
\tilde{a}_{0}+\tilde{b}_{1}, & j=0,\\
\tilde{a}_{j}+\tilde{b}_{j+1}-\tilde{b}_{j}, & 1\leq j \leq n-2,\\
\tilde{a}_{j}-\tilde{b}_{j},& j=n-1,
\end{cases}
\end{equation}
where 
\begin{equation*}
\tilde{a}_{0}=(1-\sigma)^{1-\alpha}~\text{and}~
\tilde{a}_{l}=\left(l+1-\sigma\right)^{1-\alpha}-\left(1-\sigma\right)^{1-\alpha} \quad l\geq 1,
\end{equation*}
\begin{equation*}
\tilde{b}_{l}=\frac{1}{(2-\alpha)}\left[\left(l+1-\sigma\right)^{2-\alpha}-\left(l-\sigma\right)^{2-\alpha}\right]-\frac{1}{2}\left[\left(l+1-\sigma\right)^{1-\alpha}+\left(l-\sigma\right)^{1-\alpha}\right]~l\geq 1,
\end{equation*}
with  $\tilde{\mathbb{Q}}^{n-\sigma}$ is the  truncation error.\\
\noindent \textbf{Linearization:} Linearized approximation of  the nonlinearity and diffusion at $t_{n-\sigma}$ is given below.  For Kirchhoff term 
\begin{equation}\label{7-10-22-4}
\begin{aligned}
   u^{n-\sigma}&\approx (2-\sigma)u^{n-1}-(1-\sigma)u^{n-2},~\text{for}~n~\geq~2\\
   &:=\bar{u}^{n-1,\sigma},
   \end{aligned}
\end{equation}
and for diffusion term 
\begin{equation}\label{7-10-22-5}
\begin{aligned}
   u^{n-\sigma}&\approx (1-\sigma)u^{n}+(\sigma)u^{n-1},~\text{for}~n~\geq~1\\
   &:=\hat{u}^{n,\sigma}.
   \end{aligned}
\end{equation}
\noindent \textbf{Modified Simpson's rule:} With a small modification in  \eqref{S1.14} we obtain the following approximation of memory term on $[0,t_{n-\sigma}]$
\begin{equation}\label{S1.14AD}
\begin{aligned}
\int_{0}^{t_{n-\sigma}}g(s)~ds&=\sum_{j=0}^{n-1}\tilde{w}_{nj}g(t_{j})+\tilde{q}^{n-\sigma}(g)\\
&=\frac{k_{1}}{3}\sum_{j=1}^{j_{n}/2}\left[g(\bar{t}_{2j}^{n})+4g(\bar{t}_{2j-1}^{n})+g(\bar{t}_{2j-2}^{n})\right]\\
&+\frac{k}{2}\sum_{j=j_{n}+1}^{J_{n}}\left[g(\bar{t}_{j}^{n})+g(\bar{t}_{j-1}^{n})\right]+\left(1-\sigma\right)kg(\bar{t}_{J_{n}}^{n})+\tilde{q}^{n-\sigma}(g),
\end{aligned}
\end{equation}
where $\tilde{q}^{n-\sigma}(g)$ is the quadrature error associated with the function $g$ at $t_{n-\sigma}$.\\
\noindent By combining all approximations \eqref{1.20}-\eqref{S1.14AD}, we construct the following linearized  L2-1$_{\sigma}$ Galerkin FEM.\\
\noindent \textbf{Linearized  L2-1$_\sigma$ Galerkin  FEM:} Find $u_{h}^{n}~(n=1,2,3,\dots,N)$ in $X_{h}$ with $\bar{u}_{h}^{n-1,\sigma}=\left(2-\sigma\right)u_{h}^{n-1}-\left(1-\sigma\right)u_{h}^{n-2}$ and $\hat{u}_{h}^{n,\sigma}=\left(1-\sigma\right)u_{h}^{n}+\left(\sigma\right)u_{h}^{n-1}$ such that the following equations hold for all $v_{h}$ in $X_{h}$\\
For $n\geq 2$,
\begin{equation}\tag{$F_{\alpha}$}\label{7-10-22-6}
\begin{aligned}
\left(\tilde{\mathbb{D}}^{\alpha}_{t_{n-\sigma}}u_{h}^{n},v_{h}\right)+\left(M\left(x,t_{n-\sigma},\|\nabla \bar{u}_{h}^{n-1,\sigma}\|^{2}\right)\nabla \hat{u}_{h}^{n,\sigma},\nabla v_{h}\right)&=\sum_{j=1}^{n-1}\tilde{w}_{nj}B\left(t_{n-\sigma},t_{j},u_{h}^{j},v_{h}\right)\\
&+\left(f^{n-\sigma},v_{h}\right).
\end{aligned}
\end{equation}
For $n=1$,
\begin{equation*}
\begin{aligned}
\left(\tilde{\mathbb{D}}^{\alpha}_{t_{1-\sigma}}u_{h}^{1},v_{h}\right)+\left(M\left(x,t_{1-\sigma},\|\nabla \hat{u}_{h}^{1,\sigma}\|^{2}\right)\nabla \hat{u}_{h}^{1,\sigma},\nabla v_{h}\right)&=\left(1-\sigma\right)kB\left(t_{1-\sigma},t_{0},u_{h}^{0},v_{h}\right)\\
&+\left(f^{1-\sigma},v_{h}\right),
\end{aligned}
\end{equation*}
\noindent Similar to the Lemma \ref{L1.2} we have the following discrete kernel corresponding to the kernel $(\tilde{c}^{n}_{j})$.
\begin{lem}\label{L1.21}  \cite{liao2019discrete}
Define
	\begin{equation*}
	\tilde{p}_{0}^{(n)}=\frac{1}{\tilde{c}_{0}^{(n)}},~~	\tilde{p}_{j}^{(n)}=\frac{1}{\tilde{c}_{0}^{(n-j)}}\sum_{k=0}^{j-1}\left(\tilde{c}_{j-k-1}^{(n-k)}-\tilde{c}_{j-k}^{(n-k)}\right)\tilde{p}_{k}^{(n)}~~\text{for}~~1\leq j\leq n-1.
	\end{equation*}
	Then $\tilde{p}_{j}^{(n)}$ satisfies
		 \begin{equation}
0< \tilde{p}_{n-j}^{(n)} <1,
		\end{equation} 
		\begin{equation}\label{2.812}
		\sum_{j=k}^{n}\tilde{p}_{n-j}^{(n)}\tilde{c}_{j-k}^{(j)}=1,\quad 1\leq k \leq n \leq N,
		\end{equation}
		\begin{equation}\label{2.914}
		\Gamma(2-\alpha)\sum_{j=1}^{n}\tilde{p}_{n-j}^{(n)}\leq \frac{n^{\alpha}}{\Gamma(1+\alpha)}, \quad 1\leq n \leq N.
		\end{equation}	
\end{lem}
\begin{thm}\label{J2}
	\textbf{(Convergence estimate for the numerical scheme \eqref{7-10-22-6})} Suppose that hypotheses (H1), (H2), and (H3) hold. Then the fully discrete solution  $u_{h}^{n}~(1\leq n \leq N)$ of the scheme \eqref{7-10-22-6} satisfies the following convergence estimate
	\begin{equation}\label{11-10-22-50}
	\max_{1\leq n \leq N}\|u(t_{n})-u_{h}^{n}\|+\left(k^{\alpha}\sum_{n=1}^{N}\tilde{p}_{N-n}^{(N)}\|\nabla u(t_{n})-\nabla u_{h}^{n}\|^{2}\right)^{1/2}\lesssim (h+k^{2}).
	\end{equation}
\end{thm}

\section{Well-posedness of the weak formulation \eqref{6-10-22-3}}\label{11-10-22-54}
In this section, we prove the well-posedness of the weak formulation \eqref{6-10-22-3} using the Galerkin method. For this, we study the  variational formulation \eqref{6-10-22-3} and apply the energy argument to derive a priori bounds on every Galerkin sequence. As a consequence of  compactness Lemma \ref{13-5-3}, these a priori bounds establish the convergence of the Galerkin sequence to the weak solution of the problem \eqref{6-10-22-1}.
\begin{subsection}{Proof of the Theorem \ref{t2.5}}
\begin{proof} Let $\{\lambda_{i}\}_{i=1}^{m}$ be the eigenvalues and $\{\phi_{i}\}_{i=1}^{m}$ be the corresponding  eigenfunctions  of the Dirichlet problem for the standard Laplacian operator in $H^{1}_{0}(\Omega)$. We consider the finite dimensional subspace $\mathbb{V}_{m}$ of $H^{1}_{0}(\Omega)$ such that $\mathbb{V}_{m}= \text{span} \{\phi_{1}, \phi_{2}, \dots , \phi_{m}\}.$  We assume that 
	\begin{equation}\label{3.1}
	u_{m}(\cdot,t)=\sum_{j=1}^{m}\alpha_{mj}(t)\phi_{j}\quad \text{and }\quad  u_{m}(\cdot,0)=\sum_{j=1}^{m}(u_{0},\phi_{j})\phi_{j},
	\end{equation}
	then $u_{m}(\cdot,0)$ converges to $u_{0}$. Consider the problem \eqref{6-10-22-3} onto finite dimensional subspace   $\mathbb{V}_{m}$, $i.e.,$ find  $u_{m} \in \mathbb{V}_{m}$ which  satisfying the following equation for all $v_{m} \in \mathbb{V}_{m}~ \text{and}~ a.e. ~t \in (0,T]$
	\begin{equation}\label{3.2}
	\begin{aligned}
	\left(\partial^{\alpha}_{t}u_{m},v_{m}\right)+\left(M(x,t,\|\nabla u_{m}\|^2)\nabla u_{m}, \nabla v_{m}\right)=(f,v_{m})+ \int_{0}^{t}B(t,s,u_{m}(s),v_{m})~ds.
	\end{aligned}
	\end{equation}
	Put the values of $u_{m}$ and $u_{m}(0)$ in \eqref{3.2}, we obtain a system of fractional order differential equations. Then by the theory of fractional order differential equations Lemma \ref{13-5-1}, the system \eqref{3.2} has a continuous solution $u_{m}(t)$ on some interval $[0,t_{n}), 0<t_{n}< T,$ with vanishing trace of  $k\ast(u_{m}-u_{m}(0))$ at $t=0$ \cite{zacher2009weak}. These local solutions $u_{m}(t)$ are extended to the whole interval by using the following  a priori bounds.\\
	\textbf{(A priori bounds)}
Take $v_{m}=u_{m}(t)$ in \eqref{3.2} to  get 
\begin{equation}\label{8-10-22-3}
	\begin{aligned}
	\left(\frac{d}{dt}\left(k\ast u_{m}\right),u_{m}\right)&+\left(M\left(x,t,\|\nabla u_{m}\|^2\right)\nabla u_{m}, \nabla u_{m}\right)
	\\
	&=\left(\frac{d}{dt}\left(k\ast u_{m}(0)\right),u_{m}\right)+(f,u_{m})+ \int_{0}^{t}B(t,s,u_{m}(s),u_{m}(t))~ds.
	\end{aligned}
	\end{equation}
	Let $k_{n}, n\in \mathbb{N}$ be the sequence of kernels  defined in Lemma \ref{L3.1}, then equation \eqref{8-10-22-3} is rewritten as  
		\begin{equation}\label{3.4Ad}
	\begin{aligned}
	&\left(\frac{d}{dt}(k_{n}\ast u_{m}),u_{m}\right)+\left(M\left(x,t,\|\nabla u_{m}\|^2\right)\nabla u_{m}, \nabla u_{m}\right)\\
	&=h_{mn}(t)+k(t)(u_{m}(0),u_{m})+(f,u_{m})+\int_{0}^{t}B(t,s,u_{m}(s),u_{m}(t))~ds,
	\end{aligned}
	\end{equation}
	with
	\begin{equation}
	    h_{mn}(t)=	\left(\frac{d}{dt}\left(k_{n}\ast u_{m}\right)-	\frac{d}{dt}\left(k\ast u_{m}\right),u_{m}\right).
	\end{equation}
	Use Lemma \ref{12-5-5}, positivity of diffusion coefficient (H2), continuity of $B(t,s,\cdot,\cdot)$ \eqref{1.1a}, Cauchy-Schwarz and Young's inequalities to obtain
	\begin{equation}\label{13-5-5}
	\begin{aligned}
	    &\frac{d}{dt}\left(k_{n}\ast \|u_{m}\|^{2}\right)(t)+k_{n}(t)\|u_{m}\|^{2}+ \|\nabla u_{m}\|^{2}\\
	    &\lesssim h_{mn}(t)+k(t)\|u_{m}(0)\|^{2}+k(t)\|u_{m}\|^{2}+\|f\|^{2}+\|u_{m}\|^{2}	+\int_{0}^{t}\|\nabla u_{m}(s)\|^{2}~ds.
	    \end{aligned}
	\end{equation}
	By convolving the equation \eqref{13-5-5} with the kernel $l(t)=\frac{t^{\alpha-1}}{\Gamma(\alpha)}$ and letting $n \rightarrow \infty$ in \eqref{13-5-5} then equation \eqref{13-5-5} reduces to 
	\begin{equation}\label{13-5-6}
	\begin{aligned}
	\|u_{m}(t)\|^{2}+\left(l\ast \|\nabla u_{m} \|^{2}\right)(t)&\lesssim \|u_{m}(0)\|^{2}+\left(l\ast\|f\|^{2}\right)(t)\\
	&+l\ast\left(\|u_{m}\|^{2}+\int_{0}^{t}\|\nabla u_{m}(s)\|^{2}~ds\right).
	\end{aligned}
	\end{equation}
	In \eqref{13-5-6} we have used the fact that 
	\begin{equation}
	    l\ast \frac{d}{dt}\left(k_{n}\ast \|u_{m}\|^{2}\right)(t)=\frac{d}{dt}\left(k_{n}\ast l \ast \|u_{m}\|^{2}\right)(t)\rightarrow \frac{d}{dt}\left(k\ast l\ast \|u_{m}\|^{2}\right)(t)=\|u_{m}\|^{2},
	\end{equation}
	with 
	\begin{equation}
	    (l\ast k_{n})(t)\rightarrow (l \ast k)(t)= 1 ~~\text{and}~~(l\ast h_{mn})(t)\rightarrow 0~~\text{as}~~n \rightarrow \infty ~~\text{in}~~L^{1}(0,T).
	\end{equation}
	Denote $\tilde{u}_{m}(t):=\|u_{m}(t)\|^{2}+\left(l\ast \|\nabla u_{m} \|^{2}\right)(t)$ and $\tilde{v}_{m}(t):=\|u_{m}(0)\|^{2}+\left(l\ast\|f\|^{2}\right)(t)$. Then equation \eqref{13-5-6} is converted into  
	\begin{equation}\label{13-5-7}
	\begin{aligned}
	\tilde{u}_{m}(t)&\lesssim \tilde{v}_{m}(t)+l\ast\left(\|u_{m}\|^{2}+\int_{0}^{t}\|\nabla u_{m}(s)\|^{2}~ds\right)\\
	&\lesssim \tilde{v}_{m}(t)+ l\ast\left(\|u_{m}\|^{2}+\int_{0}^{t}(t-s)^{\alpha-1}(t-s)^{1-\alpha}\|\nabla u_{m}(s)\|^{2}~ds\right)\\
	&\lesssim \tilde{v}_{m}(t)+\int_{0}^{t}(t-s)^{\alpha-1}\tilde{u}_{m}(s)~ds.
	\end{aligned}
	\end{equation}
	As a consequence of  Lemma \ref{lem2.5} and Poincar\'{e} inequality, we deduce 
	\begin{equation}\label{8-10-22-2}
	\begin{aligned}
   \|u_{m}\|_{L^{\infty}(0,T;L^{2}(\Omega))}+\|u_{m}\|_{L^{2}_{\alpha}(0,T;H^{1}_{0}(\Omega))}&\lesssim \| u_{m}(0)\|+\|f\|_{L^{\infty}(0,T;L^{2}(\Omega))}\\
   &\lesssim \| u_{0}\|+\|f\|_{L^{\infty}(0,T;L^{2}(\Omega))}\\
   &\lesssim \| \nabla u_{0}\|+\|f\|_{L^{\infty}(0,T;L^{2}(\Omega))}.
   \end{aligned}
\end{equation}
Further, we define two new discrete Laplacian operators $\Delta_{m}^{M}, ~\Delta_{m}^{b_{2}}: \mathbb{V}_{m} \rightarrow \mathbb{V}_{m}$ such that 
	\begin{equation}\label{8-10-22-5}
	    (-\Delta_{m}^{M}~u_{m},v_{m}):=(M(x,t,\|\nabla u_{m}\|^{2})\nabla u_{m},\nabla v_{m})~\text{for all}~u_{m},v_{m}~\in~\mathbb{V}_{m},~t\in (0,T],
	\end{equation}
	and 
	\begin{equation}\label{8-10-22-6}
	    (-\Delta_{m}^{b_{2}}~u_{m},v_{m}):=(b_{2}(x,t,s)\nabla u_{m},\nabla v_{m})~\text{for all}~u_{m},v_{m}~\in~\mathbb{V}_{m},~t,s~\in (0,T].
	\end{equation}
Since diffusion coefficient is positive  and $b_{2}$ is a symmetric positive definite matrix therefore  $\Delta_{m}^{M}$ and $ \Delta_{m}^{b_{2}}$ are well defined. We make use of these definitions \eqref{8-10-22-5} and  \eqref{8-10-22-6}  to convert the equation \eqref{3.2} into 
	\begin{equation}\label{8-10-22-7}
\begin{aligned}
&\left(\partial^{\alpha}_{t}u_{m},v_{m}\right)+ \left(-\Delta_{m}^{M}~ u_{m},v_{m}\right)\\
&=\left(f,v_{m}\right)+\int_{0}^{t}(-\Delta_{m}^{b_{2}}~u_{m}(s),v_{m})~ds+\int_{0}^{t}(\nabla \cdot (b_{1}(x,t,s)u_{m}(s)),v_{m})~ds\\
&+\int_{0}^{t}(b_{0}(x,t,s)u_{m}(s),v_{m})~ds.
\end{aligned}
\end{equation}
Put $v_{m}=-\Delta_{m}^{M}~u_{m}(t)$ in \eqref{8-10-22-7} and apply Cauchy-Schwarz inequality together with Young's inequality to obtain 
	\begin{equation}\label{8-10-22-8}
\begin{aligned}
\left(\partial^{\alpha}_{t}\nabla u_{m},M\left(x,t,\|\nabla u_{m}\|^2\right)\nabla u_{m}\right)&+ \|\Delta_{m}^{M}~ u_{m}\|^{2}\\
&\lesssim \int_{0}^{t}\|\Delta_{m}^{b_{2}}~u_{m}(s)\|^{2}~ds+\int_{0}^{t}\|\nabla u_{m}(s)\|^{2}~ds\\
&+\int_{0}^{t}\|u_{m}(s)\|^{2}~ds+\|f\|^{2}.
\end{aligned}
\end{equation}
Estimate $|(b_{2}(x,t,s)\nabla u_{m}(s), \nabla v_{m})|\lesssim \|\nabla u_{m}(s)\|~\|\nabla v_{m}\|$ implies 
\begin{equation}\label{29-10-22-1}
\begin{aligned}
    \|\Delta_{m}^{b_{2}}~u_{m}(s)\|=\sup_{v_{m}\in \mathbb{V}_{m}}\frac{|(b_{2}(x,t,s)\nabla u_{m}(s), \nabla v_{m})|}{\|\nabla v_{m}\|}&\lesssim  \|\nabla u_{m}(s)\|.
    \end{aligned}
\end{equation}
Hypothesis (H2) and estimate \eqref{29-10-22-1} yield
\begin{equation}\label{8-10-22-9}
\begin{aligned}
\left(\frac{d}{dt}\left[k\ast \nabla u_{m}\right](t),\nabla u_{m}\right)+ \|\Delta_{m}^{M}~ u_{m}\|^{2}
&\lesssim \int_{0}^{t}\left(\|\nabla u_{m}(s)\|^{2}+\|u_{m}(s)\|^{2}\right)~ds
+\|f\|^{2}\\
&+k(t)\left(\nabla u_{m}(0),\nabla u_{m}\right).
\end{aligned}
\end{equation}
Following the similar lines  of the proof of estimate \eqref{8-10-22-2}, we reach at 
\begin{equation}\label{8-10-22-10}
\begin{aligned}
\|u\|^{2}_{L^{\infty}(0,T;H^{1}_{0}(\Omega))}+ \left(l\ast\|\Delta_{m}^{M}~ u_{m}\|^{2}\right)(t)
&\lesssim \|\nabla u_{m}(0)\|^{2}+\|f\|^{2}_{L^{\infty}(0,T;L^{2}(\Omega))}\\
&\lesssim \|\nabla u_{0}\|^{2}+\|f\|^{2}_{L^{\infty}(0,T;L^{2}(\Omega))}.
\end{aligned}
\end{equation}
Finally, substitute $v_{m}=\partial^{\alpha}_{t}u_{m}$ in \eqref{8-10-22-7} to have 
	\begin{equation}\label{8-10-22-12}
\begin{aligned}
\|\partial^{\alpha}_{t}u_{m}\|^{2}+\left(-\Delta_{m}^{M}~ u_{m},\partial^{\alpha}_{t}u_{m}\right)
&\lesssim \|f\|^{2}+\int_{0}^{t}(\|\Delta_{m}^{b_{2}}~u_{m}(s)\|^{2}+\|\nabla u_{m}(s)\|^{2})~ds.\\
&+ \int_{0}^{t}\|u_{m}(s)\|^{2}~ds.
\end{aligned}
\end{equation}
Proceeding  further as estimate \eqref{8-10-22-10} is proved  to conclude 
\begin{equation}\label{14-5-3}
\begin{aligned}
\|\partial^{\alpha}_{t}u_{m}\|_{L^{2}(0,T;L^{2}(\Omega))}+\|u_{m}\|_{L^{\infty}(0,T;H^{1}_{0}(\Omega))}&\lesssim \|\nabla u_{m}(0)\|+\|f\|_{L^{\infty}(0,T;L^{2}(\Omega))}\\
&\lesssim \|\nabla u_{0}\|+\|f\|_{L^{\infty}(0,T;L^{2}(\Omega))}.
\end{aligned}
\end{equation}
	Thus, estimates \eqref{8-10-22-2} and \eqref{14-5-3} provide a subsequence of $(u_{m})$ again denoted by $(u_{m})$ such that $u_{m}\rightharpoonup u$ in $L^{2}\left(0,T;H^{1}_{0}(\Omega)\right)$ and  $\partial^{\alpha}_{t}u_{m} \rightharpoonup \partial^{\alpha}_{t}u$ in $L^{2}\left(0,T;L^{2}(\Omega)\right)$. In the  light of estimates \eqref{8-10-22-10} and \eqref{14-5-3}, we apply  compactness Lemma \ref{13-5-3} to conclude $u_{m} \rightarrow u$ in $L^{2}\left(0,T;H^{1}_{0}(\Omega)\right)$. Now  using the fact that $M(x,t,\|\nabla u_{m}\|^{2})$ and $B(t,s,u_{m}(s),v_{m})$ are continuous and an application of Lebesgue dominated convergence theorem, we pass the limit inside \eqref{3.2} 
	which establishes the existence of weak solutions of the problem \eqref{6-10-22-1}. \\
	\noindent \textbf{(Initial Condition)} The weak solution $u$ satisfies the following equation for all $v$ in $H^{1}_{0}(\Omega)$
\begin{equation}\label{N2}
\begin{aligned}
\left(\frac{d}{dt}\left[k\ast (u-u_{0})\right](t),v\right)+ \left(M\left(x,t,\|\nabla u\|^{2}\right)\nabla  u,\nabla v\right)&=\left(f,v\right)+\int_{0}^{t}B(t,s,u(s),v)ds.\\
\end{aligned}
\end{equation}
Let $\phi$ in $C^{1}\left([0,T];H^{1}_{0}(\Omega)\right)$ with $\phi(T)=0$, multiply \eqref{N2} with $\phi$ and integrate by parts to get 
\begin{equation}\label{N24}
\begin{aligned}
-\int_{0}^{T}\left((k\ast (u-u_{0}))(t),v\right)\phi'(t)dt&+ \int_{0}^{T}\left(M\left(x,t,\|\nabla u\|^{2}\right)\nabla  u,\nabla v\right)\phi(t)dt\\
&=\int_{0}^{T}\left(f,v\right)\phi(t)dt+\int_{0}^{T}\int_{0}^{t}B(t,s,u(s),v)\phi(t)dsdt\\
&+((k\ast(u-u_{0}))(0),\phi(0)).
\end{aligned}
\end{equation}
Since $C^{1}\left([0,T];H^{1}_{0}(\Omega)\right)$ is dense in $L^{2}\left(0,T;H^{1}_{0}(\Omega)\right)$, thus using \eqref{3.2} and $(k\ast(u_{m}-u_{m}(0)))$ has vanishing trace at $t=0$, we have 
\begin{equation}\label{N248}
\begin{aligned}
-\int_{0}^{T}\left((k\ast (u_{m}-u_{m}(0)))(t),v\right)&\phi'(t)dt+ \int_{0}^{T}\left(M\left(x,t,\|\nabla u_{m}\|^{2}\right)\nabla  u_{m},\nabla v\right)\phi(t)dt\\
&=\int_{0}^{T}\left(f,v\right)\phi(t)dt+\int_{0}^{T}\int_{0}^{t}B(t,s,u_{m}(s),v)\phi(t)dsdt.
\end{aligned}
\end{equation}
Let $m$ tend  to infinity in \eqref{N248} and comparing with  \eqref{N24} to get 
$((k\ast(u-u_{0}))(0),\phi(0))=0$. Since $\phi(0)$ is arbitrary, so we have $(k\ast(u-u_{0}))(0)=0$ which implies $u=u_{0}$ at $t=0$ \cite{zacher2009weak}.\\
\noindent \textbf{(Uniqueness)}	Suppose that $u_{1},u_{2}$ are solutions of the weak formulation \eqref{6-10-22-3}, then $z=u_{1}-u_{2}$ satisfies the following equation for all $ v \in H_{0}^{1}(\Omega)$ and $a.e. ~ t \in (0,T]$ 
	\begin{equation}\label{3.13}
	\begin{aligned}
	\left(\frac{d}{dt}(k\ast z)(t),v\right)&+\left(M\big(x,t,\|\nabla u_{1}\|^2\big)\nabla z, \nabla v\right)\\
	&=\left(\left[M\big(x,t,\|\nabla u_{2}\|^2\big)-M\big(x,t,\|\nabla u_{1}\|^2\big)\right]\nabla u_{2}, \nabla v\right)\\
	&+ \int_{0}^{t}B(t,s,z(s),v)~ds.
	\end{aligned}
	\end{equation}
	Put $v=z(t)$ in \eqref{3.13} and using (H2), (H3), and a priori bound \eqref{t2.5-2} on $u_{1},u_{2}$ along with Cauchy-Schwarz and Young's inequality to obtain
	\begin{equation}
	\begin{aligned}
	\left(\frac{d}{dt}(k\ast z)(t),z(t)\right)+(m_{0}-4L_{M}K^{2})\|\nabla z\|^{2} \lesssim  
	 \int_{0}^{t}\|\nabla z(s)\|^{2}ds. 
	\end{aligned}
	\end{equation}
 Following the similar lines as in the proof of  estimate \eqref{8-10-22-2} and using (H2)
	we conclude  $\|z\|_{L^{2}\left(0,T;H^{1}_{0}(\Omega)\right)}   =\|z\|_{L^{\infty}\left(0,T;L^{2}(\Omega)\right)}=0$. Thus uniqueness follows. 
\end{proof}
\end{subsection}

\section{Semi discrete formulation and error estimate}\label{11-10-22-55}
In this section, we  discuss the well-posedness of the semi discrete formulation \eqref{6-10-22-4} and  derive error estimate for the semi discrete solution by modifying Ritz-Volterra  projection operator.
\begin{thm}\label{8-10-22-13}
    Suppose that hypotheses (H1), (H2), and (H3) hold. Then there exists a unique solution to the problem \eqref{6-10-22-4} which satisfies the following a priori bounds 
    \begin{equation}\label{8-10-22-14}
	\|u_{h}\|_{L^{\infty}\left(0,T;L^{2}(\Omega)\right)}+	\|u_{h}\|_{L^{2}_{\alpha}\left(0,T;H^{1}_{0}(\Omega)\right)}\lesssim \left(\|\nabla u_{0}\|+\|f\|_{L^{\infty}(0,T;L^{2}(\Omega))}\right),
	\end{equation}
	\begin{equation}\label{8-10-22-15}
\|\partial^{\alpha}_{t}u_{h}\|_{L^{2}(0,T;L^{2}(\Omega))}+\|u_{h}\|_{L^{\infty}(0,T;H^{1}_{0}(\Omega))}\lesssim \left(\|\nabla u_{0}\|+\|f\|_{L^{\infty}(0,T;L^{2}(\Omega))}\right).
	\end{equation} 
\end{thm}
\begin{proof}
    This theorem is proved analogously to the proof of Theorem \ref{t2.5}.
\end{proof}
\noindent For the semi discrete  error estimate, we define a new Ritz-Volterra type projection operator $W:[0,T]\rightarrow X_{h}$ by 
\begin{equation}\label{2.5}
\left(M\left(x,t,\|\nabla u\|^{2}\right)\nabla (u-W),\nabla v_{h}\right):=\int_{0}^{t}B(t,s,u(s)-W(s),v_{h})~ds\quad \text{for all}~ v_{h}~\text{in}~ X_{h}.
\end{equation} 
This modified Ritz-Volterra  projection  operator $W$  is well defined by the positivity of the Kirchhoff term $M$ \cite{cannon1990priori}. This projection operator satisfies the following stability and best approximation results.  
 \begin{lem}\cite{kumar2020finite}\label{lem2.8}
	Consider $W$ is the modified Ritz-Volterra projection operator  defined in \eqref{2.5}, then $\|\nabla W\|$ is bounded for every $t$ in $[0,T]$, i.e., 
	\begin{equation*}
	\|\nabla W\|\lesssim\|\nabla u\|.
	\end{equation*}
\end{lem}
\noindent To derive the best approximation properties of the modified Ritz-Volterra projection operator, we assume some regularity assumptions on the solution $u$ of the problem  \eqref{6-10-22-1} such that
\begin{equation}\label{8-10-22-16}
    \|u(t)\|_{2}\lesssim C~\text{and}~ \|u_{t}(t)\|_{2}\lesssim C~\forall~t~\in~ [0,T].
\end{equation}
\begin{thm} \label{thm2.8}
	Suppose that the solution $u$ of the problem \eqref{6-10-22-1} satisfies \eqref{8-10-22-16}. Then modified Ritz-Volterra projection operator has the following best approximation properties
	\begin{equation}\label{8-10-22-17}
	\begin{aligned}
	\|\rho(t)\|+h\|\nabla \rho(t)\| &\lesssim h^{2}~\forall~t~\in~[0,T],\\
	\|\rho_{t}(t)\|+h\|\nabla \rho_{t}(t)\| &\lesssim  h^{2}~\forall~t~\in~[0,T],
	\end{aligned}
	\end{equation}
	where $\rho:=u-W$. 
\end{thm} 
\begin{proof} 
    For the proof of this theorem we refer the readers to \cite{cannon1988non,cannon1990priori}.
\end{proof}
\par Now error estimate for the semi discrete formulation \eqref{6-10-22-4} is attained  as stated in Theorem \ref{t2.6}.
\begin{subsection}{Proof of the Theorem \ref{t2.6}}
\begin{proof} Denote $(W-u_{h}):=\theta$ such that $u-u_{h}=\rho+\theta.$ Then put $u_{h}=W-\theta$ in the problem  \eqref{6-10-22-4} to  have
	\begin{equation*}
	\begin{aligned}
	\left(\partial^{\alpha}_{t}(W-\theta),v_{h}\right) &+\left(M\left(x,t,\|\nabla u_{h}\|^{2}\right)\nabla (W-\theta),\nabla v_{h}\right)\\
	&=(f,v_{h}) +\int_{0}^{t}B(t,s,W(s)-\theta(s),v_{h})~ds.
	\end{aligned}
	\end{equation*}
	Weak formulation \eqref{6-10-22-3} and the definition \eqref{2.5} of the modified Ritz-Volterra projection operator  $W$ yield 
	\begin{equation}\label{4.4}
	\begin{aligned}
	&\left(\partial^{\alpha}_{t}\theta,v_{h}\right) +\left(M(x,t,\|\nabla u_{h}\|^{2}\right)\nabla \theta,\nabla v_{h})\\
	&=-\left(\partial^{\alpha}_{t}\rho,v_{h}\right)+\int_{0}^{t}B(t,s,\theta(s),v_{h})ds+\left(\left(M(x,t,\|\nabla u_{h}\|^{2}-M(x,t,\|\nabla u\|^{2}\right)\nabla W,\nabla v_{h}\right).
	\end{aligned}
	\end{equation}
	Set $v_{h}=\theta(t)$ in \eqref{4.4} and employ (H2), (H3) to obtain
	\begin{equation}\label{4.4-1}
	\begin{aligned}
	\left(\partial^{\alpha}_{t}\theta,\theta\right) +m_{0}\|\nabla \theta\|^{2}
	&=\|\partial^{\alpha}_{t}\rho\|\|\theta(t)\|+\|\nabla \theta(t)\|\int_{0}^{t}\|\nabla\theta(s)\|ds\\
	&+L_{M}(\|\nabla u_{h}\|+\|\nabla u\|)(\|\nabla \rho\|+\|\nabla \theta\|)\|\nabla W\|\|\nabla \theta\|.
	\end{aligned}
	\end{equation}
	By utilizing proved  a priori bounds  on $\|\nabla u\|, \|\nabla u_{h}\|$ and $\|\nabla W\|$ together with Cauchy-Schwarz and Young's inequality, we obtain
	\begin{equation}
	\begin{aligned}
\left(\partial^{\alpha}_{t}\theta,\theta\right)+(m_{0}-4L_{M}K^{2})\|\nabla \theta\|^{2} &\lesssim \|\partial^{\alpha}_{t}\rho\|^{2}+\|\theta\|^{2}+\int_{0}^{t}\|\nabla \theta(s)\|^{2}ds+\|\nabla \rho\|^{2}.
\end{aligned}
	\end{equation}
	Use (H2) and  apply similar  arguments as we prove estimate \eqref{8-10-22-2} to deduce
	\begin{equation*}
	\begin{aligned}
\|\theta\|^{2}_{L^{\infty}\left(0,T;L^{2}(\Omega)\right)}&+\left(l\ast\|\nabla \theta\|^{2}\right)(t)\lesssim \left[l\ast \left(\|\nabla \rho\|^{2}+\|\partial^{\alpha}_{t}\rho\|^{2}\right)\right](t)+\|\nabla \theta(0)\|^{2}.
\end{aligned}
	\end{equation*}
	Consider 
	\begin{equation}\label{11-10-22-51}
	\begin{aligned}
	  \|\partial^{\alpha}_{t}\rho\|=  \|~^{C}D^{\alpha}_{t}\rho\|&=\left\|\frac{1}{\Gamma(1-\alpha)}\int_{0}^{t}(t-s)^{-\alpha}\frac{\partial \rho}{\partial s}(s)~ds\right\|\\
	    &\lesssim \frac{1}{\Gamma(1-\alpha)}\int_{0}^{t}(t-s)^{-\alpha}\left\|\frac{\partial \rho}{\partial s}(s)\right\|~ds\\
	    &\lesssim \frac{1}{\Gamma(1-\alpha)}\int_{0}^{t}(t-s)^{-\alpha} h^2 ~ds \lesssim h^{2}.
	    \end{aligned}
	\end{equation}
We choose $u_{h}^{0}=W(0)$  such that $\theta(0)=0$. Then apply approximation properties of modified Ritz-Volterra projection operator \eqref{8-10-22-17} and  \eqref{11-10-22-51} to  conclude
\begin{equation*}
	\begin{aligned}
\|\theta\|^{2}_{L^{\infty}\left(0,T;L^{2}(\Omega)\right)}+\left(l\ast\|\nabla \theta\|^{2}\right)(t)\lesssim h^{2}+ h^{4}\lesssim  h^{2}.
\end{aligned}
	\end{equation*}
	Finally, triangle inequality and estimate \eqref{8-10-22-17} finish the proof.
\end{proof}
\end{subsection}
 
\begin{section}{Linearized  L1 Galerkin FEM}\label{11-10-22-56}
	In this section, we prove the well-posedness of the numerical scheme \eqref{7-10-22-2} and  carry out its convergence analysis. The following two lemmas provide  a priori bounds on the solution of the problem \eqref{7-10-22-2}.
	\begin{lem}\label{5.3}
		Under the hypothesis (H1), (H2), and (H3) the solution  $u_{h}^{n}~ (n\geq 1)$ of the scheme \eqref{7-10-22-2}  satisfy the following a priori bound
		\begin{equation}\label{S1.24}
		\begin{aligned}
		\max_{1\leq m \leq N}\|u_{h}^{m}\|^{2}+k^{\alpha}\sum_{n=1}^{N}p_{N-n}\|\nabla u_{h}^{n}\|^{2}
		\lesssim \|\nabla u_{0}\|^{2}+\max_{1\leq n \leq N}\|f^{n}\|^{2}.
		\end{aligned}
		\end{equation}
			\end{lem}
		\begin{proof}
			Put $v_{h}=u_{h}^{1}$ for $n=1$ in the formulation \eqref{7-10-22-2} to get 
			\begin{equation*}
    \begin{aligned}
\left(\mathbb{D}^{\alpha}_{t}u_{h}^{1},u_{h}^{1}\right)+\left(M\left(x,t_{1},\|\nabla u_{h}^{1}\|^{2}\right)\nabla u_{h}^{1},\nabla u_{h}^{1}\right)=\left(f^{1},u_{h}^{1}\right)+kB\left(t_{1},t_{0},u_{h}^{0},u_{h}^{1}\right).
\end{aligned}
\end{equation*}
Employing (H2) and (H3) to obtain 
			\begin{equation*}
			\begin{aligned}
			(1-k^{\alpha}\Gamma(2-\alpha))\|u_{h}^{1}\|^{2}+k^{\alpha}\|\nabla u_{h}^{1}\|^{2}&\lesssim k^{\alpha}
			\left(\|f^{1}\|^{2}+k^{2}\|\nabla u_{h}^{0}\|^{2}\right)+\|u_{h}^{0}\|^{2}.
			\end{aligned}
			\end{equation*}
			For sufficiently small $k$ such that $k^{\alpha}< \frac{1}{\Gamma(2-\alpha)}$, we conclude
			\begin{equation*}
			\begin{aligned}
			\|u_{h}^{1}\|^{2}+k^{\alpha}\|\nabla u_{h}^{1}\|^{2}&\lesssim 
			\|f^{1}\|^{2}+\|\nabla u_{h}^{0}\|^{2}\lesssim 
			\|f^{1}\|^{2}+\|\nabla u_{0}\|^{2}.
			\end{aligned}
			\end{equation*}
			Further, set $v_{h}=u_{h}^{n}$ for $n\geq 2$ in the scheme \eqref{7-10-22-2} to have
			\begin{equation}
			    	\left(\mathbb{D}^{\alpha}_{t}u_{h}^{n},u_{h}^{n}\right)+\left(M\left(x,t_{n},\|\nabla \bar{u}_{h}^{n-1}\|^{2}\right)\nabla u_{h}^{n},\nabla u_{h}^{n}\right)=\left(f^{n},u_{h}^{n}\right)+\sum_{j=1}^{n-1}w_{nj}B\left(t_{n},t_{j},u_{h}^{j},u_{h}^{n}\right).
			\end{equation}
		Applying the identity $\left(\mathbb{D}^{\alpha}_{t}u_{h}^{n},u_{h}^{n}\right)\geq \frac{1}{2}\mathbb{D}^{\alpha}_{t}\|u_{h}^{n}\|^{2}$ \cite{li2016analysis} and hypotheses  (H2), (H3) along with Cauchy-Schwarz and  Young's inequality to reach at
			\begin{equation}\label{S1.23-2}
		\mathbb{D}^{\alpha}_{t}\|u_{h}^{n}\|^{2}+\|\nabla u_{h}^{n}\|^{2}\lesssim  \left(\|f^{n}\|^{2}+\|u_{h}^{n}\|^{2}+\sum_{j=1}^{n-1}w_{nj}\|\nabla u_{h}^{j}\|^{2}\right).
			\end{equation}
			By the definition of $\mathbb{D}^{\alpha}_{t}\|u_{h}^{n}\|^{2}$ \eqref{2.9}, the equation \eqref{S1.23-2} reduces to 
			\begin{equation}\label{S1.23}
			\frac{k^{-\alpha}}{\Gamma(2-\alpha)}\sum_{j=1}^{n}a_{n-j}\left(\|u_{h}^{j}\|^{2}-\|u_{h}^{j-1}\|^{2}\right)+\|\nabla u_{h}^{n}\|^{2}\lesssim  \left(\|f^{n}\|^{2}+\|u_{h}^{n}\|^{2}+\sum_{j=1}^{n-1}w_{nj}\|\nabla u_{h}^{j}\|^{2}\right).
			\end{equation}
			Multiply the equation \eqref{S1.23} by discrete convolution $P_{m-n}$ defined in Lemma \ref{L1.2} and take summation from $n=1$ to $m$ to obtain 
			\begin{equation*}
			\begin{aligned}
			&\sum_{n=1}^{m}p_{m-n}\sum_{j=1}^{n}a_{n-j}\left(\|u_{h}^{j}\|^{2}-\|u_{h}^{j-1}\|^{2}\right)+k^{\alpha}\Gamma(2-\alpha)\sum_{n=1}^{m}p_{m-n}\|\nabla u_{h}^{n}\|^{2}\\
			&\lesssim k^{\alpha}\Gamma(2-\alpha)\left(\sum_{n=1}^{m}p_{m-n}\|f^{n}\|^{2}+\sum_{n=1}^{m}p_{m-n}\|u_{h}^{n}\|^{2}+\sum_{n=1}^{m}p_{m-n}\sum_{j=1}^{n-1}w_{nj}\|\nabla u_{h}^{j}\|^{2}\right).
			\end{aligned}
			\end{equation*}
			 Interchanging the summation and property of discrete kernel \eqref{2.81} with $\alpha_{0}=k^{\alpha}\Gamma(2-\alpha)$ yield
			\begin{equation*}
			\begin{aligned}
		(1-\alpha_{0})	\|u_{h}^{m}\|^{2}+k^{\alpha}\sum_{n=1}^{m}p_{m-n}\|\nabla u_{h}^{n}\|^{2}
			&\lesssim \sum_{n=1}^{m-1}\left(k^{\alpha}p_{m-n}\|u_{h}^{n}\|^{2}+k_{1}k^{\alpha}\sum _{j=1}^{n}p_{n-j}\|\nabla u_{h}^{j}\|^{2}\right)\\
			&+k^{\alpha}\sum_{n=1}^{m}p_{m-n}\|f_{h}^{n}\|^{2}+\|u_{h}^{0}\|^{2}.
			\end{aligned}
			\end{equation*}
			 Then for sufficiently small $k^{\alpha} <\frac{1}{\Gamma{(2-\alpha)}}$ one have 
				\begin{equation*}
			\begin{aligned}
			\|u_{h}^{m}\|^{2}+k^{\alpha}\sum_{n=1}^{m}p_{m-n}\|\nabla u_{h}^{n}\|^{2}
			&\lesssim \sum_{n=1}^{m-1}\left(k^{\alpha}p_{m-n}+k_{1}\right)\left(\|u_{h}^{n}\|^{2}+k^{\alpha}\sum _{j=1}^{n}p_{n-j}\|\nabla u_{h}^{j}\|^{2}\right)\\
			&+k^{\alpha}\sum_{n=1}^{m}p_{m-n}\|f^{n}\|^{2}+\|u_{h}^{0}\|^{2}.
			\end{aligned}
			\end{equation*}
			Further,  the  discrete Gr\"{o}nwall's inequality provides
			\begin{equation}\label{11-10-22-3}
			\begin{aligned}
			\|u_{h}^{m}\|^{2}+k^{\alpha}\sum_{n=1}^{m}p_{m-n}\|\nabla u_{h}^{n}\|^{2}
			&\lesssim \left(k^{\alpha}\sum_{n=1}^{m}p_{m-n}\|f^{n}\|^{2}+\|u_{h}^{0}\|^{2}\right)\exp\left(\sum_{n=1}^{m-1}\left(k^{\alpha}p_{m-n}+k_{1}\right)\right).
			\end{aligned}
			\end{equation}
			Finally, using property of discrete kernel   \eqref{2.91} one obtain    
			\begin{equation}\label{11-10-22-1}
			    k^{\alpha}\sum_{n=1}^{m}p_{m-n}\|f^{n}\|^{2}\lesssim \max_{1\leq n \leq N}\|f^{n}\|^{2}\left(k^{\alpha}\sum_{n=1}^{m}p_{m-n}\right)\lesssim\max_{1\leq n \leq N}\|f^{n}\|^{2} k^{\alpha}m^{\alpha}\lesssim T^{\alpha} \max_{1\leq n \leq N}\|f^{n}\|^{2}.
			\end{equation}
			Also, 
	\begin{equation}\label{11-10-22-2}
	    \sum_{n=1}^{m-1}\left(k^{\alpha}p_{m-n}+k_{1}\right)\lesssim k^{\alpha}m^{\alpha}+mk_{1}\lesssim T.
	\end{equation}
	To conclude the result \eqref{S1.24},  put \eqref{11-10-22-1} and \eqref{11-10-22-2} in \eqref{11-10-22-3} as 
	\begin{equation}\label{11-10-22-3-1}
			\begin{aligned}
			\|u_{h}^{m}\|^{2}+k^{\alpha}\sum_{n=1}^{m}p_{m-n}\|\nabla u_{h}^{n}\|^{2}
			&\lesssim \max_{1\leq n \leq N}\|f^{n}\|^{2}+\|u_{h}^{0}\|^{2} \lesssim \max_{1\leq n \leq N}\|f^{n}\|^{2}+\|\nabla u_{0}\|^{2}.
			\end{aligned}
			\end{equation}
		\end{proof}
			\begin{lem}\label{5.3-1}
		Suppose that hypotheses (H1), (H2), and (H3) hold. Then the  solution  $u_{h}^{n}~ (n\geq 1)$ of the scheme \eqref{7-10-22-2}  satisfy the following a priori bound
		\begin{equation}\label{S1.24-1}
		\begin{aligned}
		\max_{1\leq m \leq N}\|\nabla u_{h}^{m}\|^{2}+k^{\alpha}\sum_{n=1}^{N}p_{N-n}\|\Delta_{h}^{M}u_{h}^{n}\|^{2}
		\lesssim \|\nabla u_{0}\|^{2}+ \max_{1\leq n \leq N}\|f^{n}\|^{2}.
		\end{aligned}
		\end{equation}
		where $\Delta_{h}^{M} : X_{h}\rightarrow X_{h}$ is the  discrete Laplacian operator defined in \eqref{8-10-22-5} associated with $M$.
			\end{lem}
			\begin{proof}
			    By  making use of definitions of discrete Laplacian operator $-\Delta_{h}^{M}u_{h}^{1}$ \eqref{8-10-22-5} and $-\Delta_{h}^{b_{2}}u_{h}^{0}$ \eqref{8-10-22-6} the equation for $n=1$ in the scheme \eqref{7-10-22-2} is rewritten as 
			     	\begin{equation}\label{9-10-22-1}
    \begin{aligned}
\left(\mathbb{D}^{\alpha}_{t}u_{h}^{1},v_{h}\right)+(-\Delta_{h}^{M}u_{h}^{1},v_{h})&=\left(f^{1},v_{h}\right)+k(-\Delta_{h}^{b_{2}}u_{h}^{0},v_{h})+k(\nabla \cdot (b_{1}(x,t,s)u_{h}^{0}),v_{h})\\
&+k(b_{0}(x,t,s)u_{h}^{0},v_{h}).
\end{aligned}
\end{equation}
Setting $v_{h}=-\Delta_{M_{h}}u_{h}^{1}$ in \eqref{9-10-22-1} one obtain
	\begin{equation}\label{9-10-22-2}
    \begin{aligned}
\left(\mathbb{D}^{\alpha}_{t}\nabla u_{h}^{1},\nabla u_{h}^{1}\right)+\|\Delta_{h}^{M}u_{h}^{1}\|^{2}&\lesssim \|f^{1}\|^{2}+k^{2}\left(\|\Delta_{h}^{b_{2}}u_{h}^{0}\|^{2}+\|\nabla u_{h}^{0}\|^{2}+\|u_{h}^{0}\|^{2}\right).
\end{aligned}
\end{equation}
Identity $\left(\mathbb{D}^{\alpha}_{t}u_{h}^{n},u_{h}^{n}\right)\geq \frac{1}{2}\mathbb{D}^{\alpha}_{t}\|u_{h}^{n}\|^{2}$ and estimate \eqref{29-10-22-1}  simplify the equation \eqref{9-10-22-2} to  
\begin{equation}\label{9-10-22-3}
    \begin{aligned}
\|\nabla u_{h}^{1}\|^{2}+k^{\alpha}\|\Delta_{h}^{M}u_{h}^{1}\|^{2}&\lesssim k^{\alpha}\|f^{1}\|^{2}+k^{\alpha}k^{2}\|\nabla u_{h}^{0}\|^{2}+\|\nabla u_{h}^{0}\|^{2}.
\end{aligned}
\end{equation}
For sufficiently small $k$, we deduce
		\begin{equation}\label{9-10-22-4}
    \begin{aligned}
\|\nabla u_{h}^{1}\|^{2}+k^{\alpha}\|\Delta_{h}^{M}u_{h}^{1}\|^{2}&\lesssim \|f^{1}\|^{2}+\|\nabla u_{h}^{0}\|^{2}\lesssim \|f^{1}\|^{2}+\|\nabla u_{0}\|^{2}.
\end{aligned}
\end{equation}	     
Consider the scheme \eqref{7-10-22-2} for $n\geq 2$ with definitions of discrete Laplacian operators \eqref{8-10-22-5} and \eqref{8-10-22-6}  
\begin{equation}\label{9-10-22-5}
    \begin{aligned}
    \left(\mathbb{D}^{\alpha}_{t}u_{h}^{n},v_{h}\right)+\left(-\Delta_{h}^{M}u_{h}^{n}, v_{h}\right)&=\left(f^{n},v_{h}\right)+\sum_{j=1}^{n-1}w_{nj}\left(-\Delta_{h}^{b_{2}}u_{h}^{j},v_{h}\right)\\
    &+\sum_{j=1}^{n-1}w_{nj}\left(\nabla \cdot(b_{1}(x,t_{n},t_{j})u_{h}^{j}),v_{h}\right)\\
    &+\sum_{j=1}^{n-1}w_{nj}\left(b_{0}(x,t_{n},t_{j})u_{h}^{j},v_{h}\right).
    \end{aligned}
\end{equation}
Take $v_{h}=-\Delta_{h}^{M}u_{h}^{n}$ in \eqref{9-10-22-5} and apply (H2), (H3) along with Cauchy-Schwarz and  Young's inequality to get 
\begin{equation}\label{9-10-22-5-1}
    \begin{aligned}
    \left(\mathbb{D}^{\alpha}_{t}\nabla u_{h}^{n},\nabla u_{h}^{n}\right)+\|\Delta_{h}^{M}u_{h}^{n}\|^{2}&\lesssim \|f^{n}\|^{2}+\sum_{j=1}^{n-1}w_{nj}\|\Delta_{h}^{b_{2}}u_{h}^{j}\|^{2}\\
    &+\sum_{j=1}^{n-1}w_{nj}\|\nabla u_{h}^{j}\|^{2}+\sum_{j=1}^{n-1}w_{nj}\|u_{h}^{j}\|^{2}.
    \end{aligned}
\end{equation}
By using the estimate  \eqref{29-10-22-1}  and  the identity $\left(\mathbb{D}^{\alpha}_{t}u_{h}^{n},u_{h}^{n}\right)\geq \frac{1}{2}\mathbb{D}^{\alpha}_{t}\|u_{h}^{n}\|^{2}$,  the equation \eqref{9-10-22-5-1} is converted into  to 
			\begin{equation}\label{S1.23-1}
			\mathbb{D}^{\alpha}_{t}\|\nabla u_{h}^{n}\|^{2}+\|\Delta_{h}^{M}u_{h}^{n}\|^{2}\lesssim  \left(\|f^{n}\|^{2}+\sum_{j=1}^{n-1}w_{nj}\|\nabla u_{h}^{j}\|^{2}\right).
			\end{equation}
Further proceed as we prove  estimate \eqref{S1.24} to complete the proof.
			\end{proof}
	\par To show the existence of the fully discrete solution $u_{h}^{n}~(n\geq 1)$ of the problem \eqref{7-10-22-2}, the following variant of Br\"{o}uwer fixed point theorem is used.
	\begin{thm}\cite{kesavan1989topics}\label{T1.3}
		Let $H$ be   finite dimensional Hilbert space. Let $G:H\rightarrow H$ be a continuous map such that $\left(G(w),w\right)>0$  for all $w$ in $H$ with $\|w\|=r,~r>0$ then there exists a $\tilde{w}$ in $H$ such that $G(\tilde{w})=0$ and $\|\tilde{w}\|\leq r.$
	\end{thm}
	\begin{thm}\label{T1.4}
		Suppose that hypotheses (H1), (H2), and (H3) hold. Then there exists a unique solution $u_{h}^{n}~(n\geq 1)$ to the problem \eqref{7-10-22-2}.
			\end{thm}
		\begin{proof}
			\textbf{(Existence)} ~Take $n=1$ in the scheme \eqref{7-10-22-2} and  apply the definition of  $\mathbb{D}^{\alpha}_{t}u_{h}^{1}$ \eqref{2.9} with $\alpha_{0}=k^{\alpha}\Gamma(2-\alpha)$ to obtain 
			\begin{equation}\label{9-10-22-6}
			\begin{aligned}
			\left(u_{h}^{1}-u_{h}^{0},v_{h}\right)&+\alpha_{0}\left(M\left(x,t_{1},\|\nabla u_{h}^{1}\|^{2}\right)\nabla u_{h}^{1},\nabla v_{h}\right)=\alpha_{0}\left(f^{1},v_{h}\right)
			+\alpha_{0}kB\left(t_{1},t_{0},u_{h}^{0},v_{h}\right).
			\end{aligned}
			\end{equation}
			In the view of \eqref{9-10-22-6} we define a map $G:X_{h}\rightarrow X_{h}$ by 
			\begin{equation}\label{S1.19}
			\begin{aligned}
			\left(G\left(u_{h}^{1}\right),v_{h}\right)&=\left(u_{h}^{1},v_{h}\right)-\left(u_{h}^{0},v_{h}\right)+\alpha_{0}\left(M\left(x,t_{1},\|\nabla u_{h}^{1}\|^{2}\right)\nabla u_{h}^{1},\nabla v_{h}\right)\\
			&-\alpha_{0}\left(f^{1},v_{h}\right)-\alpha_{0}kB\left(t_{1},t_{0},u_{h}^{0},v_{h}\right).
			\end{aligned}
			\end{equation}
			Then using (H2), (H3), and Cauchy-Schwarz inequality and Poincar\'{e} inequality with Poincar\'{e} constant $C_{p}$, we have 
			\begin{equation}
			\begin{aligned}
			\left(G\left(u_{h}^{1}\right),u_{h}^{1}\right)&\geq \| u_{h}^{1}\|^{2}-\|u_{h}^{0}\|\|u_{h}^{1}\|-\alpha_{0}\|f^{1}\|\|u_{h}^{1}\|+\alpha_{0}m_{0}\|\nabla u_{h}^{1}\|^{2}-\alpha_{0}kB_{0}\|\nabla u_{h}^{0}\|\|\nabla u_{h}^{1}\|\\
			&\geq \|u_{h}^{1}\|\left(\|u_{h}^{1}\|-\|u_{h}^{0}\|-\alpha_{0}\|f^{1}\|\right)+\alpha_{0}\|\nabla u_{h}^{1}\|\left(m_{0}\|\nabla u_{h}^{1}\|-kB_{0} \|\nabla u_{h}^{0}\|\right)\\
			&\geq \|u_{h}^{1}\|\left(\|u_{h}^{1}\|-\|u_{h}^{0}\|-\alpha_{0}\|f_{h}^{1}\|\right)+\alpha_{0}\|\nabla u_{h}^{1}\|\left(\| u_{h}^{1}\|-\frac{kB_{0}C_{p}}{m_{0}} \|\nabla u_{h}^{0}\|\right).
			\end{aligned}
			\end{equation}
			Thus, for $\|u_{h}^{1}\|> \|u_{h}^{0}\|+\alpha_{0}\|f^{1}\|+\frac{kB_{0}C_{p}}{m_{0}}\|\nabla u_{h}^{0}\|$ one  have $\left(G\left(u_{h}^{1}\right),u_{h}^{1}\right)>0$ and the map $G$ defined by \eqref{S1.19} is continuous as a consequence of  continuity of $M$ and $B$. Hence existence of $u_{h}^{1}$ follows by  Theorem \ref{T1.3} immediately.\\ 
	\noindent 	\textbf{(Uniqueness)} Suppose that $X_{h}^{1}$ and $Y_{h}^{1}$ are solutions of the scheme  \eqref{7-10-22-2} for $n=1$, then $Z_{h}^{1}=X_{h}^{1}-Y_{h}^{1}$ satisfies the following equation for all $v_{h}$ in $X_{h}$
			\begin{equation}\label{S1.22}
			\begin{aligned}
			\left(\mathbb{D}^{\alpha}_{t}Z_{h}^{1},v_{h}\right)&+\left(M\left(x,t_{1},\|\nabla X_{h}^{1}\|^{2}\right)\nabla Z_{h}^{1},\nabla v_{h}\right)\\
			&=\left(\left[M\left(x,t_{1},\|\nabla Y_{h}^{1}\|^{2}\right)-M\left(x,t_{1},\|\nabla X_{h}^{1}\|^{2}\right)\right]\nabla Y_{h}^{1},\nabla v_{h}\right).
			\end{aligned}
			\end{equation}
			Put $v_{h}=Z_{h}^{1}$ in  \eqref{S1.22} and using (H2) we get
			\begin{equation*}
			\frac{1}{2}\mathbb{D}^{\alpha}_{t}\|Z_{h}^{1}\|^{2}+m_{0}\|\nabla Z_{h}^{1}\|^{2}\leq L_{M}\|\nabla Z_{h}^{1}\|\left(\|\nabla X_{h}^{1}\|+\|\nabla Y_{h}^{1}\|\right)\left(\nabla Y_{h}^{1},\nabla Z_{h}^{1}\right)
			\end{equation*}
			 Cauchy-Schwarz inequality and a priori bound \eqref{S1.24-1} yield
			\begin{equation*}
			\|Z_{h}^{1}\|^{2}+k^{\alpha}\left(2m_{0}-4L_{M}K^{2}\right)\|\nabla Z_{h}^{1}\|^{2}\leq 0.
			\end{equation*}
			At last, employ (H2) to obtain $\|Z_{h}^{1}\|=\|\nabla Z_{h}^{1}\|=0$ that  concludes the uniqueness of solution for $n=1$ in the scheme \eqref{7-10-22-2}.
\par  For $n\geq 2$, the numerical scheme \eqref{7-10-22-2} is linear with a positive definite coefficient matrix as a result we get the  existence and uniqueness of the  solution $u_{h}^{n}~(n\geq 2)$ for the problem \eqref{7-10-22-2}.
\end{proof}
\par To derive the convergence rate of  developed numerical scheme \eqref{7-10-22-2}, first we discuss  aproximation properties of L1 scheme \eqref{2.9}, linearization technique \eqref{7-10-22-1}, and quadrature error \eqref{S1.14}. 
\begin{lem}\label{9-10-22-12}\cite{lin2007finite} If $u\in C^{2}([0,T];L^{2}(\Omega))$ then truncation error $\mathbb{Q}^{n}$ defined in \eqref{2.9} satisfies 
\begin{equation}\label{9-10-22-7}
    \|\mathbb{Q}^{n}\|\lesssim k^{2-\alpha}~\text{for}~n~\geq~1.
\end{equation}
\begin{lem}\label{9-10-22-11} For any function $u\in C^{2}[0,T]$, the linearization error $(u^{n}-\bar{u}^{n-1}):=\bar{\mathbb{L}}^{n-1}$ defined in \eqref{7-10-22-1} undergoes 
\begin{equation}\label{9-10-22-8}
    |\bar{\mathbb{L}}^{n-1}|\lesssim k^{2}~\text{for}~n\geq 2.
\end{equation}
\end{lem}
\begin{proof}Apply the Taylor's series expansion of $u^{n}$ around $u^{n-1}$ and $u^{n-2}$  to obtain 
\begin{equation}
    |\bar{\mathbb{L}}^{n-1}|\lesssim k^{2}(u_{tt}(\xi_{1})+u_{tt}(\xi_{2}))~\text{for some}~\xi_{1}~\in~ (t_{n-1},t_{n})~\text{and for some}~\xi_{2}~\in~(t_{n-2},t_{n}).
\end{equation}
As $u \in C^{2}[0,T]$ that implies the result \eqref{9-10-22-8}.
\end{proof}
\end{lem}
\begin{lem}\label{9-10-22-10}\cite{pani1992numerical} If $u\in C^{4}[0,T]$ then quadrature error defined by \eqref{S1.14} has the following error estimate 
\begin{equation}\label{9-10-22-9}
    |q^{n}(u)|\lesssim k^{2}~\text{for}~n~\geq ~1. 
\end{equation}
\end{lem}
We prove the convergence estimate of the proposed  numerical scheme \eqref{7-10-22-2} by assuming that the solution $u$ of the problem \eqref{6-10-22-1} satisfy the regularity assumption used in Lemma \ref{9-10-22-12}  to Lemma \ref{9-10-22-10}, i.e.,  $u \in C^{4}([0,T];H^{2}(\Omega)\cap H^{1}_{0}(\Omega))$.
	\begin{subsection}{Proof of the Theorem \ref{J1}}
	\begin{proof} 	First we prove the error estimate for the case $n=1$. Substitute $u_{h}^{1}=W^{1}-\theta^{1}$ for $n=1$ in  the scheme \eqref{7-10-22-2} and using weak formulation \eqref{6-10-22-3}  along with modified Ritz-Volterra projection operator $W$ at $t_{1}$ to  get 
		\begin{equation}\label{S1.271}
		\begin{aligned}
		\left(\mathbb{D}^{\alpha}_{t}\theta^{1},v_{h}\right)&+\left(M\left(x,t_{1},\|\nabla u_{h}^{1}\|^{2}\right)\nabla \theta^{1},\nabla v_{h}\right)\\
		&=\left(\mathbb{D}^{\alpha}_{t}W^{1}-~^{C}D^{\alpha}_{t_{1}}u,v_{h}\right)-kB\left(t_{1},t_{0},W^{0},v_{h}\right)+\int_{0}^{t_{1}}B(t_{1},s,W(s),v_{h})ds\\
		&+\left(\left[M\left(x,t_{1},\|\nabla u_{h}^{1}\|^{2}\right)-M\left(x,t_{1},\|\nabla u^{1}\|^{2}\right)\right]\nabla W^{1},\nabla v_{h}\right)\\
		&+kB\left(t_{1},t_{0},\theta^{0},v_{h}\right).
		\end{aligned}
		\end{equation}
		Set $v_{h}=\theta^{1}$ in \eqref{S1.271} with  $\theta^{0}=0$ and using (H2), (H3) together with Cauchy-Schwarz inequality and Young's inequality to obtain
	 \begin{equation}\label{S1}
	 \begin{aligned}
	 (1-k^{\alpha}\Gamma(2-\alpha))\|\theta^{1}\|^{2}+&k^{\alpha}\left(m_{0}-4L_{M}K^{2}\right)\|\nabla \theta^{1}\|^{2}\\
	 &\lesssim k^{\alpha}\|\mathbb{Q}^{1}\|^{2}+k^{\alpha}\|~^{C}{D}^{\alpha}_{t_{1}}\rho\|^{2}+k^{\alpha}\|\nabla q^{1}(W)\|^{2}+k^{\alpha}\|\nabla \rho^{1}\|^{2}. 
	 \end{aligned}
	 \end{equation}
For sufficiently small $k^{\alpha}< \frac{1}{\Gamma(2-\alpha)}$, we apply (H2) and the approximation properties \eqref{9-10-22-7}, \eqref{8-10-22-17}, \eqref{11-10-22-51}, and \eqref{9-10-22-9} to conclude 
		\begin{equation}
		\|\theta^{1}\|^{2}+k^{\alpha}\|\nabla \theta^{1}\|^{2}\lesssim \left(k^{2-\alpha}+h\right)^{2}.
		\end{equation}
Now we derive the error estimate for $n\geq 2$, for that take  $u_{h}^{n}=W^{n}-\theta^{n}$ in the scheme  \eqref{7-10-22-2} 
 \begin{equation}\label{10-10-22-7-1-1}
    \begin{aligned}
\left(\mathbb{D}^{\alpha}_{t_{n}}\theta^{n},v_{h}\right)&+\left(M\left(x,t_{n},\|\nabla \bar{u}_{h}^{n-1}\|^{2}\right)\nabla \theta^{n},\nabla v_{h}\right)\\
&=\left(\mathbb{Q}^{n},v_{h}\right)-\left(^{C}D^{\alpha}_{t_{n}}\rho,v_{h}\right)\\
&+\left((M(x,t_{n},\|\nabla \bar{u}_{h}^{n-1}\|^{2})-M(x,t_{n},\|\nabla u^{n}\|^{2}))\nabla W^{n},\nabla v_{h}\right)\\
&+\left(\nabla q^{n}(W),\nabla v_{h}\right)+\sum_{j=1}^{n-1}w_{nj}B(t_{n},t_{j},\theta^{j},v_{h}).
\end{aligned}
\end{equation}
Put $v_{h}=\theta^{n}$ in \eqref{10-10-22-7-1-1} it follows
		\begin{equation}\label{S1.33}
		\begin{aligned}
		\mathbb{D}^{\alpha}_{t_{n}}\|\theta^{n}\|^{2}+\|\nabla \theta^{n}\|^{2}&\lesssim \|\mathbb{Q}^{n}\|^{2}+\|~^{C}D_{t_{n}}^{\alpha}\rho\|^{2}+\| \theta^{n}\|^{2}+\|\nabla \bar{\rho}^{n-1}\|^{2}+\|\nabla \bar{\theta}^{n-1}\|^{2}\\
		&+\|\nabla \bar{\mathbb{L}}^{n-1}\|^{2}+\|\nabla q^{n}(W)\|^{2}+\sum_{j=1}^{n-1}w_{nj}\|\nabla \theta^{j}\|^{2}.
		\end{aligned}
		\end{equation}
		Employ the approximation properties \eqref{9-10-22-7}, \eqref{8-10-22-17}, \eqref{11-10-22-51}, \eqref{9-10-22-8},  and \eqref{9-10-22-9} to deduce
		\begin{equation}\label{S1.37}
		\begin{aligned}
			\mathbb{D}^{\alpha}_{t_{n}}\|\theta^{n}\|^{2}+\|\nabla \theta^{n}\|^{2}&\lesssim \left(k^{2-\alpha}+h^{2}+h+k^{2}+k^{2}\right)^{2}+\|\theta^{n}\|^{2}\\
		&+\left(\|\nabla \theta^{n-1} \|^{2}+\|\nabla \theta^{n-2}\|^{2}+\sum_{j=1}^{n-1}k_{1}\|\nabla \theta^{j}\|^{2}\right).
		\end{aligned}
		\end{equation}
		Now follow the similar arguments as we prove estimate \eqref{S1.24} to conclude the result \eqref{11-10-22-49}.
	\end{proof}
\end{subsection}
\end{section}

\section{Linearized L2-1$_{\sigma}$ Galerkin scheme}\label{11-10-22-57}
In this section, we show that proposed numerical scheme \eqref{7-10-22-6} achieve the second order  convergence in the time direction. 

\begin{lem}\label{6.1}  Under the hypotheses (H1), (H2), and (H3) the solution $u_{h}^{n}~(n\geq 1)$ of the scheme \eqref{7-10-22-6} satisfy the following a priori bound
		\begin{equation}\label{S2.24}
			\begin{aligned}
			\max_{1\leq m \leq N}\|u_{h}^{m}\|^{2}+k^{\alpha}\sum_{n=1}^{N}\tilde{p}_{N-n}^{(N)}\|\nabla u_{h}^{n}\|^{2}
			\lesssim \max_{1\leq n \leq N}\|f^{n-\frac{\alpha}{2}}\|^{2}+\|\nabla u_{0}\|^{2}.
			\end{aligned}
			\end{equation}
	\end{lem}
	\begin{proof} For $n=1$ the scheme  \eqref{7-10-22-6} is 
	\begin{equation}\label{10-10-22-3}
\begin{aligned}
\left(\tilde{\mathbb{D}}^{\alpha}_{t_{1-\sigma}}u_{h}^{1},v_{h}\right)+\left(M\left(x,t_{1-\sigma},\|\nabla \hat{u}_{h}^{1,\sigma}\|^{2}\right)\nabla \hat{u}_{h}^{1,\sigma},\nabla v_{h}\right)&=\left(1-\sigma\right)kB\left(t_{1-\sigma},t_{0},u_{h}^{0},v_{h}\right)\\
&+\left(f^{1-\sigma},v_{h}\right).
\end{aligned}
\end{equation}
Substitute $v_{h}=u_{h}^{1}$ in \eqref{10-10-22-3} to get 
	\begin{equation}\label{10-10-22-2}
\begin{aligned}
\frac{k^{-\alpha}(1-\sigma)^{1-\alpha}}{\Gamma(2-\alpha)}(u_{h}^{1}-u_{h}^{0},u_{h}^{1})&+(1-\sigma)\left(M\left(x,t_{1-\sigma},\|\nabla \hat{u}_{h}^{1,\sigma}\|^{2}\right)\nabla u_{h}^{1},\nabla u_{h}^{1}\right)\\
&=-\sigma\left(M\left(x,t_{1-\sigma},\|\nabla \hat{u}_{h}^{1,\sigma}\|^{2}\right)\nabla u_{h}^{0},\nabla u_{h}^{1}\right)\\
&+\left(1-\sigma\right)kB\left(t_{1-\sigma},t_{0},u_{h}^{0},u_{h}^{1}\right)+\left(f^{1-\sigma},u_{h}^{1}\right).
\end{aligned}
\end{equation}
Simplification of \eqref{10-10-22-2} using (H2) and (H3) with $\alpha_{0}=k^{\alpha}\Gamma(2-\alpha)$ and $\tilde{a}_{0}=(1-\sigma)^{1-\alpha}$  yields
	\begin{equation}\label{10-10-22-35}
\begin{aligned}
\left(1-\frac{\alpha_{0}}{\tilde{a}_{0}}\right)\|u_{h}^{1}\|^{2}+k^{\alpha}\|\nabla  u_{h}^{1}\|^{2}
&\lesssim k^{\alpha}\|\nabla u_{h}^{0}\|^{2}+\|u_{h}^{0}\|^{2}
+k^{2+\alpha}\|\nabla u_{h}^{0}\|^{2}+k^{\alpha}\|f^{1-\sigma}\|^{2}.
\end{aligned}
\end{equation}
Take sufficiently small $k$ to conclude 
\begin{equation}\label{10-10-22-4}
\begin{aligned}
\|u_{h}^{1}\|^{2}+k^{\alpha}\|\nabla  u_{h}^{1}\|^{2}
\lesssim \|\nabla u_{h}^{0}\|^{2} +\|f^{1-\sigma}\|^{2} \lesssim \|\nabla u_{0}\|^{2}+\|f^{1-\sigma}\|^{2}.
\end{aligned}
\end{equation}
For $u_{h}^{n}~( n\geq 2)$ in the scheme \eqref{7-10-22-6} to have 
\begin{equation}\label{10-10-22-5}
\begin{aligned}
\left(\tilde{\mathbb{D}}^{\alpha}_{t_{n-\sigma}}u_{h}^{n},v_{h}\right)+\left(M\left(x,t_{n-\sigma},\|\nabla \bar{u}_{h}^{n-1,\sigma}\|^{2}\right)\nabla \hat{u}_{h}^{n,\sigma},\nabla v_{h}\right)&=\sum_{j=1}^{n-1}\tilde{w}_{nj}B\left(t_{n-\sigma},t_{j},u_{h}^{j},v_{h}\right)\\
&+\left(f^{n-\sigma},v_{h}\right).
\end{aligned}
\end{equation}
		Put $v_{h}=u_{h}^{n}$ in \eqref{10-10-22-5} to obtain
		\begin{equation}\label{10-10-22-6}
\begin{aligned}
\left(\tilde{\mathbb{D}}^{\alpha}_{t_{n-\sigma}}u_{h}^{n},u_{h}^{n}\right)&+(1-\sigma)\left(M\left(x,t_{n-\sigma},\|\nabla \bar{u}_{h}^{n-1,\sigma}\|^{2}\right)\nabla u_{h}^{n},\nabla u_{h}^{n}\right)\\
&=\sum_{j=1}^{n-1}\tilde{w}_{nj}B\left(t_{n-\sigma},t_{j},u_{h}^{j},u_{h}^{n}\right)+\left(f^{n-\sigma},u_{h}^{n}\right)\\
&-\sigma\left(M\left(x,t_{n-\sigma},\|\nabla \bar{u}_{h}^{n-1,\sigma}\|^{2}\right)\nabla u_{h}^{n-1},\nabla u_{h}^{n}\right).
\end{aligned}
\end{equation}
We invoke the identity $\left(\tilde{\mathbb{D}}^{\alpha}_{t_{n-\sigma}}u_{h}^{n},u_{h}^{n}\right)\geq \frac{1}{2}\tilde{\mathbb{D}}^{\alpha}_{t_{n-\sigma}}\|u_{h}^{n}\|^{2}$ and apply (H2), (H3) along with Cauchy-Schwarz and Young's inequality to reach at
		\begin{equation}\label{S2.23}
		\begin{aligned}
		\tilde{\mathbb{D}}^{\alpha}_{t_{n-\sigma}}\|u_{h}^{n}\|^{2}+\|\nabla u_{h}^{n}\|^{2}&\lesssim  \|f^{n-\frac{\alpha}{2}}\|^{2}+\sum_{j=1}^{n-1}\tilde{w}_{nj}\|\nabla u_{h}^{j}\|^{2}+\|\nabla u_{h}^{n-1}\|^{2}+\|u_{h}^{n}\|^{2}.
		\end{aligned}
		\end{equation}
		We follow the similar arguments as we prove estimate \eqref{S1.24} to obtain \eqref{S2.24}.
	\end{proof}
\begin{lem}\label{6.1-1-2}  Under the assumptions  (H1), (H2), and (H3) the solution $u_{h}^{n}~(n\geq 1)$ of the scheme \eqref{7-10-22-6} satisfy the following a priori bound
		\begin{equation}\label{S2.24-2}
			\begin{aligned}
			\max_{1\leq m \leq N}\|\nabla u_{h}^{m}\|^{2}+k^{\alpha}\sum_{n=1}^{N}\tilde{p}_{N-n}^{(N)}\|\Delta_{h}^{M}u_{h}^{n}\|^{2}
			\lesssim \max_{1\leq n \leq N}\|f^{n-\frac{\alpha}{2}}\|^{2}+\|\nabla u_{0}\|^{2}.
			\end{aligned}
			\end{equation}
	\end{lem}
\begin{proof}
    We combine the idea of Lemma \ref{5.3-1} and Lemma \ref{6.1} to prove the result \eqref{S2.24-2}.
\end{proof}
\begin{thm}
Suppose that (H1), (H2), and (H3) hold. Then there exists a unique solution to the problem \eqref{7-10-22-6}.
\end{thm}
\begin{proof}
    For the case $n\geq 2$, the scheme \eqref{7-10-22-6} is linear having positive definite coefficient matrix. Thus existence and uniqueness in this case follow immediately. For the case $n=1$ in the scheme \eqref{7-10-22-6}, the equation is nonlinear so we again use Br\"{o}uwer fixed point Theorem \ref{T1.3}. Consider the case for $n=1$ in the problem \eqref{7-10-22-6} with $\alpha_{0}=k^{\alpha}\Gamma(2-\alpha)$ and $\tilde{a}_{0}=(1-\sigma)^{1-\alpha}$
    \begin{equation}\label{9-10-22-6-1}
			\begin{aligned}
			\left(u_{h}^{1}-u_{h}^{0},v_{h}\right)&+\frac{\alpha_{0}}{\tilde{a}_{0}}\left(M\left(x,t_{1},\|\nabla \hat{u}_{h}^{1,\sigma}\|^{2}\right)\nabla \hat{u}_{h}^{1,\sigma},\nabla v_{h}\right)\\
			&=\frac{\alpha_{0}}{\tilde{a}_{0}}\left(f^{1-\sigma},v_{h}\right)
			+\frac{\alpha_{0}}{\tilde{a}_{0}}kB\left(t_{1-\sigma},t_{0},u_{h}^{0},v_{h}\right).
			\end{aligned}
			\end{equation}
			Multiply the equation \eqref{9-10-22-6-1} by $(1-\sigma)$ to obtain 
			 \begin{equation}\label{9-10-22-6-2}
			\begin{aligned}
			\left(\hat{u}_{h}^{1,\sigma},v_{h}\right)-\left(u_{h}^{0},v_{h}\right)&+(1-\sigma)\frac{\alpha_{0}}{\tilde{a}_{0}}\left(M\left(x,t_{1},\|\nabla \hat{u}_{h}^{1,\sigma}\|^{2}\right)\nabla \hat{u}_{h}^{1,\sigma},\nabla v_{h}\right)\\
			&=(1-\sigma)\frac{\alpha_{0}}{\tilde{a}_{0}}\left(f^{1-\sigma},v_{h}\right)
			+(1-\sigma)\frac{\alpha_{0}}{\tilde{a}_{0}}kB\left(t_{1-\sigma},t_{0},u_{h}^{0},v_{h}\right).
			\end{aligned}
			\end{equation}
			Further, proceeding analogously to the proof of Theorem \ref{T1.4} we conclude the existence of $\hat{u}_{h}^{1,\sigma}$. Hence existence of $u_{h}^{1}$ follows.
\end{proof}
Now, we derive the convergence estimate for the numerical scheme \eqref{7-10-22-6}. This estimate is proved with the help of the following lemmas.
\begin{lem}\label{9-10-22-12-1}\cite{alikhanov2015new} If $u\in C^{3}([0,T];L^{2}(\Omega))$ then truncation error $\tilde{\mathbb{Q}}^{n-\sigma}$ defined in \eqref{1.20} satisfies 
\begin{equation}\label{9-10-22-7-1}
    \|\tilde{\mathbb{Q}}^{n-\sigma}\|\lesssim k^{3-\alpha}~\text{for}~n~\geq~1.
\end{equation}
\begin{lem}\label{9-10-22-11-1} For any function $u\in C^{2}[0,T]$, the linearization error  $(u^{n}-\bar{u}^{n-1,\sigma}):=\bar{\mathbb{L}}^{n-1,\sigma}$ defined in \eqref{7-10-22-4} and  $(u^{n}-\hat{u}^{n,\sigma}):=\hat{\mathbb{L}}^{n,\sigma}$ defined in \eqref{7-10-22-5} undergo
\begin{equation}\label{9-10-22-8-1}
    |\bar{\mathbb{L}}^{n-1,\sigma}|\lesssim k^{2}~\text{for}~n\geq 2.
\end{equation}
and
\begin{equation}\label{9-10-22-8-1-1}
    |\hat{\mathbb{L}}^{n,\sigma}|\lesssim k^{2}~\text{for}~n\geq 1.
\end{equation}
\end{lem}
\begin{proof} This lemma is proved by  an application of Taylor's series expansion as we have proved  Lemma \eqref{9-10-22-11}.
\end{proof}
\end{lem}
\begin{lem}\label{9-10-22-10-1}\cite{pani1992numerical} If $u\in C^{4}[0,T]$ then quadrature error defined by \eqref{S1.14AD} has the following error estimate 
\begin{equation}\label{9-10-22-9-1}
    |q^{n-\sigma}(u)|\lesssim k^{2}~\text{for}~n~\geq ~1. 
\end{equation}
\end{lem}
\begin{subsection}{Proof of the Theorem \ref{J2}}
\begin{proof} Take $u_{h}^{1}=W^{1}-\theta^{1}$ with $\theta^{0}=0$ in the scheme  \eqref{7-10-22-6} for $n=1$, we have the following error equation for $\theta^1$
\begin{equation}\label{10-10-22-7}
    \begin{aligned}
&\left(\tilde{\mathbb{D}}^{\alpha}_{t_{1-\sigma}}\theta^{1},v_{h}\right)+\left(M\left(x,t_{1-\sigma},\|\nabla \hat{u}_{h}^{1,\sigma}\|^{2}\right)\nabla \hat{\theta}^{1,\sigma},\nabla v_{h}\right)\\
&=\left(\tilde{\mathbb{Q}}^{1-\sigma},v_{h}\right)-\left(^{C}D^{\alpha}_{t_{1-\sigma}}\rho,v_{h}\right)+\left(M(x,t_{1-\sigma},\|\nabla u^{1-\sigma}\|^{2})(\nabla \hat{W}^{1,\sigma}-\nabla W^{1-\sigma}),\nabla v_{h}\right)\\
&+\left((M(x,t_{1-\sigma},\|\nabla \hat{u}_{h}^{1,\sigma}\|^{2})-M(x,t_{1-\sigma},\|\nabla u^{1-\sigma}\|^{2}))\nabla \hat{W}^{1,\sigma},\nabla v_{h}\right)+\left(\nabla \tilde{q}^{1-\sigma}(W),\nabla v_{h}\right).
\end{aligned}
\end{equation}
Set $v_{h}=\theta^{1}$ in \eqref{10-10-22-7} with $\alpha_{0}=k^{\alpha}\Gamma(2-\alpha)$ and $\tilde{a}_{0}=(1-\sigma)^{1-\alpha}$ to have
\begin{equation}\label{10-10-22-8}
    \begin{aligned}
\left(1-\frac{\alpha_{0}}{\tilde{a}_{0}}\right)\|\theta^{1}\|^{2}&+k^{\alpha}\left(m_{0}-4L_{M}K^{2}\right)\|\nabla \theta^{1}\|^{2}\\
&\lesssim k^{\alpha}\|\tilde{\mathbb{Q}}^{1-\sigma}\|^{2}+ k^{\alpha}\|~^{C}D^{\alpha}_{t_{1-\sigma}}\rho\|^{2}+k^{\alpha}\|\nabla \hat{\mathbb{L}}^{1,\sigma}\|^{2}+ k^{\alpha}\|\nabla \hat{\rho}^{1,\sigma}\|^{2}\\
&+ k^{\alpha}\|\nabla \tilde{q}^{1-\sigma}(W)\|^{2}.
\end{aligned}
\end{equation}
For small value of  $k$ with hypothesis (H2) and approximation properties  \eqref{9-10-22-7-1}, \eqref{8-10-22-17}, \eqref{11-10-22-51},  \eqref{9-10-22-8-1-1}, and  \eqref{9-10-22-9-1}, we deduce
\begin{equation}\label{10-10-22-9}
    \begin{aligned}
\|\theta^{1}\|^{2}&+k^{\alpha}\|\nabla \theta^{1}\|^{2}\lesssim \left(k^{3-\alpha}+h^{2}+k^{2}+h+k^{2}\right)^{2}\lesssim \left(k^{2}+h\right)^{2}.
\end{aligned}
\end{equation}
For $n\geq 2$, substitute $u_{h}^{n}=W^{n}-\theta^{n}$ in the scheme \eqref{7-10-22-6}, then $\theta^{n}$ satisfies 
\begin{equation}\label{10-10-22-7-1}
    \begin{aligned}
&\left(\tilde{\mathbb{D}}^{\alpha}_{t_{n-\sigma}}\theta^{n},v_{h}\right)+\left(M\left(x,t_{n-\sigma},\|\nabla \bar{u}_{h}^{n-1,\sigma}\|^{2}\right)\nabla \hat{\theta}^{n,\sigma},\nabla v_{h}\right)\\
&=\left(\tilde{\mathbb{Q}}^{n-\sigma},v_{h}\right)-\left(^{C}D^{\alpha}_{t_{n-\sigma}}\rho,v_{h}\right)+\left(M(x,t_{n-\sigma},\|\nabla\bar{u}_{h}^{n-1,\sigma}\|^{2})(\nabla \hat{W}^{n,\sigma}-\nabla W^{n-\sigma}),\nabla v_{h}\right)\\
&+\left((M(x,t_{n-\sigma},\|\nabla \bar{u}_{h}^{n-1,\sigma}\|^{2})-M(x,t_{n-\sigma},\|\nabla u^{n-\sigma}\|^{2}))\nabla W^{n-\sigma},\nabla v_{h}\right)\\
&+\left(\nabla \tilde{q}^{n-\sigma}(W),\nabla v_{h}\right)+\sum_{j=1}^{n-1}\tilde{w}_{nj}B(t_{n-\sigma},t_{j},\theta^{j},v_{h}).
\end{aligned}
\end{equation}
Put $v_{h}=\theta^{n}$ in \eqref{10-10-22-7-1} it follows 
	\begin{equation}\label{2.21}
	\begin{aligned}
	\tilde{\mathbb{D}}^{\alpha}_{t_{n-\frac{\alpha}{2}}}\|\theta^{n}\|^{2}+\| \nabla \theta^{n}\|^{2}
	&\lesssim \|\tilde{\mathbb{Q}}^{n-\sigma}\|^{2}+\|~^{C}D^{\alpha}_{t_{n-\sigma}}\rho\|^{2}+\|\nabla \hat{\mathbb{L}}^{n,\sigma}\|^{2} +\|\nabla \bar{\rho}^{n-1,\sigma}\|^{2}\\
	&+\|\nabla \bar{\mathbb{L}}^{n-1,\sigma}\|^{2}+\|\nabla \tilde{q}^{n-\sigma}(W)\|^{2}\\
	&+\|\nabla \theta^{n-1}\|^{2}+\|\nabla \theta^{n-2}\|^{2}+\sum_{j=1}^{n-1}\tilde{w}_{nj}\|\nabla \theta^{j}\|^{2}.
	\end{aligned}
	\end{equation}
	Further, apply the approximation properties  \eqref{9-10-22-7-1}, \eqref{8-10-22-17}, \eqref{11-10-22-51}, \eqref{9-10-22-8-1}, \eqref{9-10-22-8-1-1},  and \eqref{9-10-22-9-1} to arrive at 
	\begin{equation}\label{S2.37}
	\begin{aligned}
	\tilde{\mathbb{D}}^{\alpha}_{t_{n-\frac{\alpha}{2}}}\|\theta^{n}\|^{2}+\|\nabla \theta^{n}\|^{2}&\lesssim\left(\| \theta^{n}\|^{2}+\|\nabla \theta^{n-1} \|^{2}+\|\nabla \theta^{n-2}\|^{2}+\sum_{j=1}^{n-1}k_{1}\|\nabla \theta^{j}\|^{2}\right)\\
	&+ \left(k^{3-\alpha}+h^{2}+k^{2}+h+k^{2}+k^{2}+h^{2}+h^{2}\right)^{2}.
	\end{aligned}
	\end{equation}
\noindent Now, follow the similar arguments as in Theorem \ref{J1} and Theorem \ref{6.1} to obtain 
	\begin{equation*}
	\|\theta^{m}\|^{2}+k^{\alpha}\sum_{n=1}^{m}\tilde{p}_{m-n}^{(m)}\|\nabla \theta^{n}\|^{2}
	\lesssim \left(h+k^{2}\right)^{2}.
	\end{equation*}
\end{proof}
\end{subsection}

\section{Numerical results}\label{11-10-22-58}
 In this section, we implement the theoretical  results obtained from fully discrete formulations \eqref{7-10-22-2} and \eqref{7-10-22-6}  for the problem \eqref{6-10-22-1}. For the  space discretization linear hat basis functions  say $\{\psi_{1}, \psi_{2}, \dots, \psi_{J}\}$ for $J$ dimensional subspace $X_{h}$ of $H_{0}^{1}(\Omega)$ are used, then numerical solution $u_{h}^{n}~(n\geq 1)$ for the considered \eqref{6-10-22-1}  at any time $t_{n}$ in $[0,T]$ is written as  
\begin{equation}\label{7.1}
u_{h}^{n}=\sum_{i=1}^{J}\alpha_{i}^{n}\psi_{i},
\end{equation}
where $\alpha^{n}=\left(\alpha_{1}^{n}, \alpha_{2}^{n}, \alpha_{3}^{n}, \dots, \alpha_{J}^{n} \right)$ is to be determined.
Further, denote the error estimates that we have proved in Theorems \ref{J1} and  \ref{J2} by 
\begin{equation*}
\begin{aligned}
\text{\bf Error-1}&=\max_{1\leq n \leq N}\|u(t_{n})-u_{h}^{n}\|+\left(k^{\alpha} \sum_{n=1}^{N}p_{N-n}\|\nabla u(t_{n})-\nabla u_{h}^{n}\|^{2}\right)^{1/2},
\end{aligned}
\end{equation*} 
and 
\begin{equation*}
\begin{aligned}
\text{\bf Error-2}&=\max_{1\leq n \leq N}\|u(t_{n})-u_{h}^{n}\|+\left(k^{\alpha} \sum_{n=1}^{N}\tilde{p}_{N-n}^{(N)}\|\nabla u(t_{n})-\nabla u_{h}^{n}\|^{2}\right)^{1/2},
\end{aligned}
\end{equation*} 
respectively.
\begin{exam} We consider the problem \eqref{6-10-22-1} on $\Omega \times [0,T]$  where $\Omega = [0,1]\times[0,1]$ and $T=1$ with following data
\begin{enumerate}
	\item $M\left(x,y,t,\|\nabla u\|^{2}\right)=a(x,y)+b(x,y)\|\nabla u\|^{2}$ with $a(x,y)=x^{2}+y^{2}+1$ and $b(x,y)=xy$. This type of diffusion coefficient $M$ has been studied by medeiros $et.al.$ for the purpose of numerical experiments in  \cite{medeiros2012perturbation}.
	\item Moreover, we take the following memory operator as in \cite{barbeiro2011h1}\\
	$b_{2}(t,s,x,y)=-e^{t-s}I; ~ b_{1}(t,s,x,y)=b_{0}(t,s,x,y)=0.$ 
		\item	Source term  
	 $f(x,t)=f_{1}(x,t)+f_{2}(x,t)+f_{3}(x,t)$ with
	 \begin{equation*}
	     \begin{aligned}
	       f_{1}(x,t)&=\frac{2}{\Gamma(3-\alpha)}t^{2-\alpha}(x-x^{2})(y-y^{2}),\\   
	       f_{2}(x,t)&=2(x+y-x^{2}-y^{2})\left(t^{2}x^{2}+t^{2}y^{2}+\frac{xyt^{6}}{45}-2t-2+2e^{t}\right),\\
	       f_{3}(x,t)&=t^{2}(2x-1)(y-y^{2})\left(2x+\frac{yt^{4}}{45}\right)+t^{2}(2y-1)(x-x^{2})\left(2y+\frac{xt^{4}}{45}\right).
	     \end{aligned}
	 \end{equation*}
	 Corresponding to the above data, the exact solution of the problem \eqref{6-10-22-1} is  given by 
		$u=t^{2}\left(x-x^{2}\right)\left(y-y^{2}\right).$	
\end{enumerate}
\end{exam}
We obtain errors and convergence rates in the space direction as well as in the time direction for different  parameters $h, k,$ and  $\alpha$. The convergence rate is calculated through the following $\log$ vs. $\log$ formula 
\begin{equation*}
   \text{ Convergence rate} =\begin{cases}\frac{\log(E(\tau,h_{1})/E(\tau,h_{2}))}{\log(h_{1}/h_{2})};& \text{In space direction}\\
   \frac{\log(E(\tau_{1},h)/E(\tau_{2},h))}{\log(\tau_{1}/\tau_{2})};& \text{In time direction }
\end{cases}
\end{equation*}
where $E(\tau,h)$ denotes the error at mesh points $\tau$ and $h$.\\
\noindent \textbf{Linearized  L1  Galerkin FEM:} This numerical scheme \eqref{7-10-22-2}  provides a convergence order of $O(h+k^{2-\alpha})$. To observe this order of convergence numerically, we run the MATLAB code at different iterations by  setting $h \simeq k^{2-\alpha}$. Here $h$ is taken as the area of the triangle in the triangulation of domain $\Omega=[0,1]\times [0,1]$. For next iteration, we join the midpoint of each edge and make another triangulation as presented in Figures 1-3. In this way we collect the numerical results upto five iterations to support our theoretical estimate.

\begin{figure}[H]
	\centering
\includegraphics[scale=0.2]{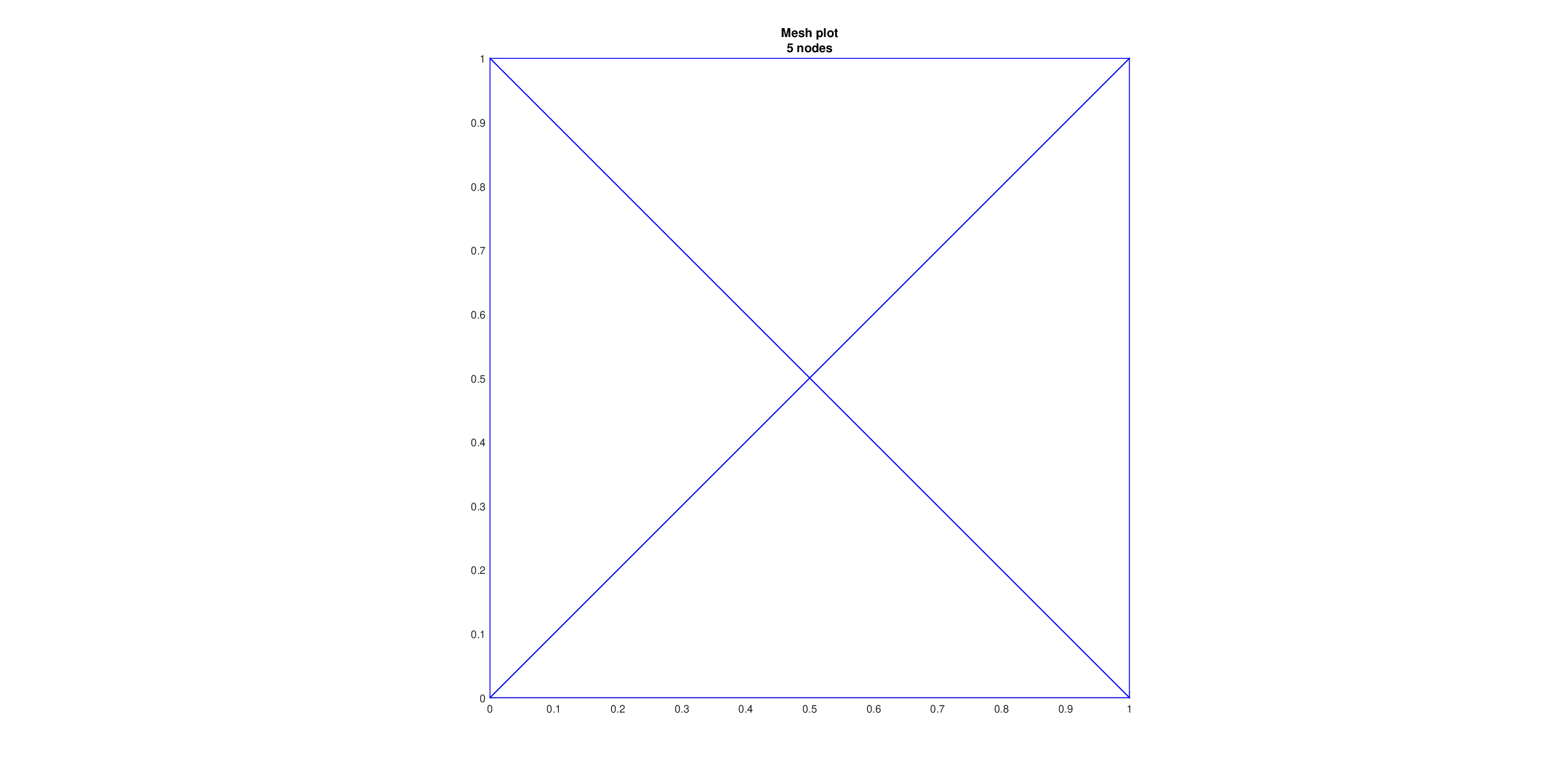}
\caption{Iteration no. 1}
\end{figure}
\begin{figure}[H]
	\centering
\includegraphics[scale=0.2]{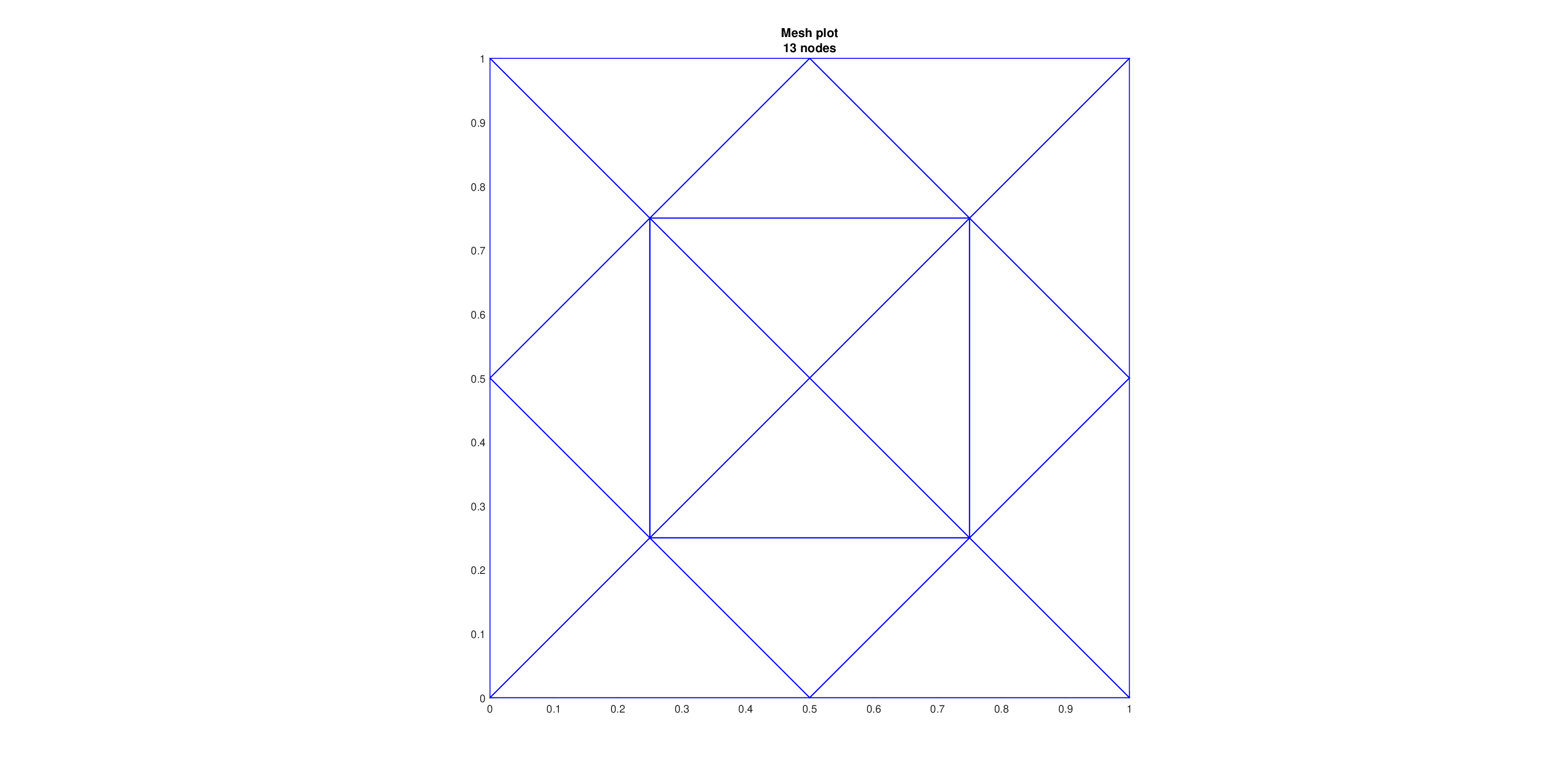}
\caption{Iteration no. 2}
\end{figure}
\begin{figure}[H]
	\centering
\includegraphics[scale=0.2]{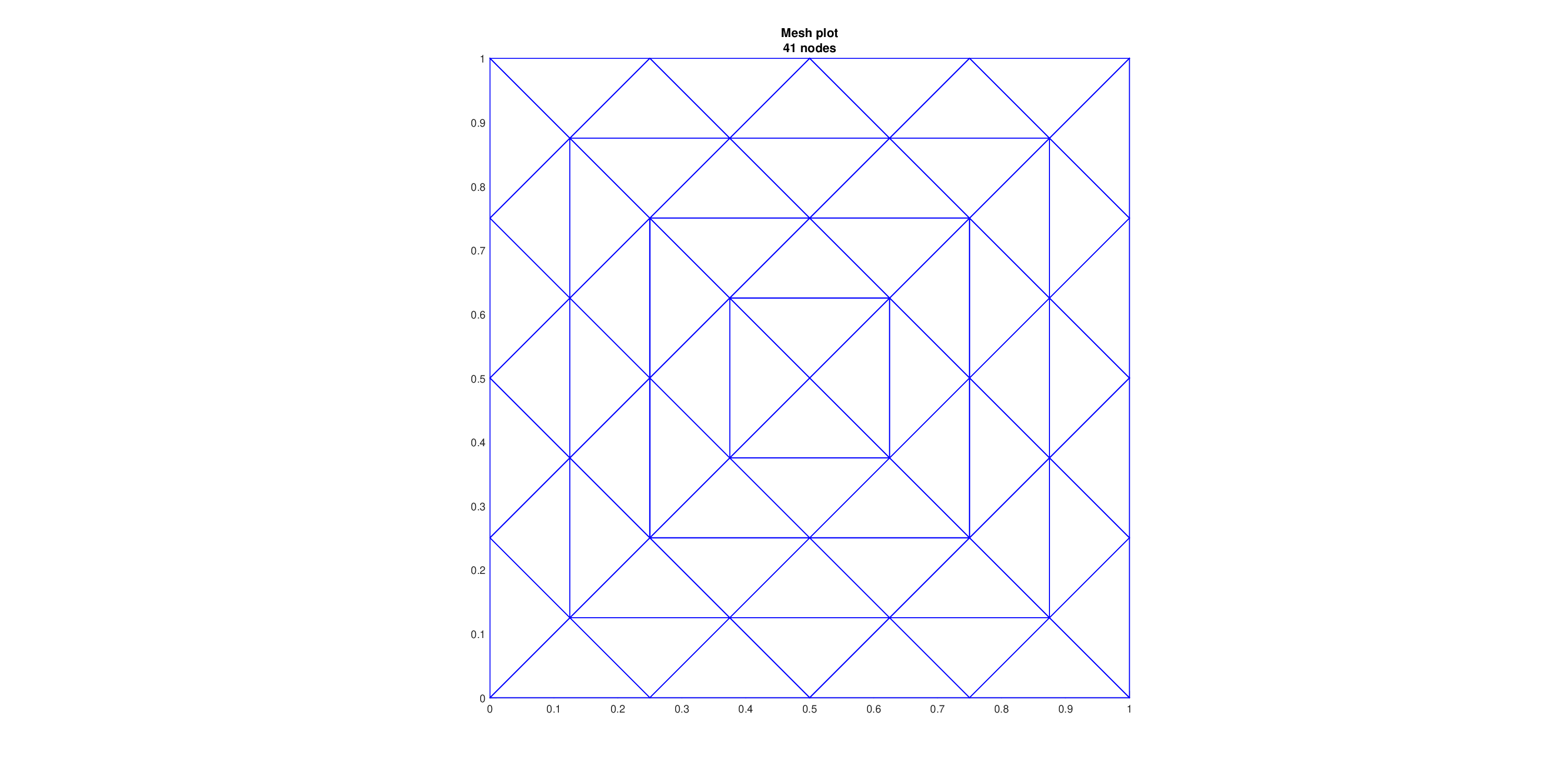}
\caption{Iteration no. 3}
\end{figure}
\noindent From  Tables \ref{table3}, \ref{table31}, and \ref{table361}, we conclude that the convergence rate is linear in space and $(2-\alpha)$ in the time direction. We also observe that as $\alpha \rightarrow 1 $ then convergence rates is  approaching to 1 in the  time direction which coincide with the results  established in \cite{kumar2020finite} for classical diffusion case. 
  \begin{table}[htbp]
	\centering
	\caption{\emph{Error and Convergence Rates in space-time direction for $\alpha=0.25$ }}
	\begin{tabular}{|c|c|c|c|}
		\hline
		\textbf{Iteration No.} & \textbf{Error-1} & \textbf{Rate in Space} & \textbf{Rate in Time} \\
		\hline
		1 & 3.04e-02 & - & -  \\
		\hline
		2 & 1.50e-02& 1.0175 & 1.7807  \\
		\hline
		3 &  7.60e-03 & 0.9824 & 1.7192 \\
		\hline
		4 & 3.82e-03 & 0.9909 & 1.7342\\
		\hline
		5 & 1.89e-03 & 1.0166 & 1.7790\\
		\hline
	\end{tabular}%
	\label{table3}%
\end{table}
\begin{table}[htbp]
	\centering
	\caption{\emph{Error and Convergence Rates in space-time direction for $\alpha=0.5$ }}
	\begin{tabular}{|c|c|c|c|}
		\hline
		\textbf{Iteration No.} & \textbf{Error-1} & \textbf{Rate in Space} & \textbf{Rate in Time} \\
		\hline
		1 & 2.63e-02 & - & -  \\
		\hline
		2 & 1.35e-02& 0.9568 & 1.4352  \\
		\hline
		3 &  6.53e-03 & 1.0521 & 1.5781 \\
		\hline
		4 & 3.20e-03 & 1.0285 & 1.5428\\
		\hline
		5 & 1.58e-03 & 1.0141 & 1.5212\\
		\hline
	\end{tabular}%
	\label{table31}%
\end{table}
\begin{table}[htbp]
	\centering
	\caption{\emph{Error and Convergence Rates in space-time direction for $\alpha=0.75$ }}
	\begin{tabular}{|c|c|c|c|}
		\hline
		\textbf{Iteration No.} & \textbf{Error-1} & \textbf{Rate in Space} & \textbf{Rate in Time} \\
		\hline
		1 & 2.23e-02 & - & -  \\
		\hline
		2 & 1.09e-02& 1.0340 & 1.2925  \\
		\hline
		3 &  5.33e-03 & 1.0345 & 1.2931 \\
		\hline
		4 & 2.65e-03 & 1.0087 & 1.2960\\
		\hline
		5 & 1.29e-03 & 1.0331 & 1.2914\\
		\hline
	\end{tabular}%
	\label{table361}%
\end{table}

\noindent \textbf{Linearized L2-1$_\sigma$  Galerkin  FEM:} This numerical scheme has theoretical convergence order of  $O(h+k^{2})$. Here we set $h\simeq k^{2}$ to conclude the convergence rates in the  space-time directions. Here iteration numbers have the same meaning as for L1 Galerkin FEM \eqref{7-10-22-2}. From the Tables \ref{table32}, \ref{table322}, and \ref{table323}, we deduce  that the convergence rate is linear in space and quadratic in the time direction.
\begin{table}[htbp]
	\centering
	\caption{\emph{Error and Convergence Rates in space-time direction for $\alpha=0.25$  }}
	\begin{tabular}{|c|c|c|c|}
		\hline
		\textbf{Iteration No.} & \textbf{Error-2} & \textbf{Rate in Space} & \textbf{Rate in Time} \\
		\hline
		1 & 2.32e-02 & - & -  \\
		\hline
		2 & 1.10e-02& 1.0679 & 2.1359  \\
		\hline
		3 &  5.38e-03 & 1.0393 & 2.0787 \\
		\hline
		4 & 2.52e-03 & 1.0904 & 2.1809\\
		\hline
		5 & 1.24e-03 & 1.0225 & 2.0451\\
		\hline
	\end{tabular}%
	\label{table32}%
\end{table}
 \begin{table}[htbp]
	\centering
	\caption{\emph{Error and Convergence Rates in space-time direction for $\alpha=0.5$  }}
	\begin{tabular}{|c|c|c|c|}
		\hline
		\textbf{Iteration No.} & \textbf{Error-2} & \textbf{Rate in Space} & \textbf{Rate in Time} \\
		\hline
		1 & 2.42e-02 & - & -  \\
		\hline
		2 & 1.16e-02& 1.0573 & 2.1147  \\
		\hline
		3 &  5.70e-03 & 1.0329 & 2.0658 \\
		\hline
		4 & 2.68e-03 & 1.0873 & 2.1746\\
		\hline
		5 & 1.32e-03 & 1.0207 & 2.0414\\
		\hline
	\end{tabular}%
	\label{table322}%
\end{table}

\begin{table}[htbp]
	\centering
	\caption{\emph{Error and Convergence Rates in space-time direction for $\alpha=0.75$  }}
	\begin{tabular}{|c|c|c|c|}
		\hline
		\textbf{Iteration No.} & \textbf{Error-2} & \textbf{Rate in Space} & \textbf{Rate in Time} \\
		\hline
		1 & 1.97e-02 & - & -  \\
		\hline
		2 & 9.05e-03& 1.1284 & 2.2569  \\
		\hline
		3 &  4.25e-03 & 1.0873 & 2.1746 \\
		\hline
		4 & 1.94e-03 & 1.1335 & 2.2671\\
		\hline
		5 & 9.30e-04 & 1.0614 & 2.1228\\
		\hline
	\end{tabular}%
	\label{table323}%
\end{table}


\section{Conclusions}\label{11-10-22-59}
 In this work, the well-posedness of the weak formulation for time-fractional integro-differential equations of Kirchhoff type for non-homogeneous materials is established. As a consequence of new Ritz-Volterra type projection operator, error estimate of $O(h)$ in energy norm for semi discrete formulation is derived. Further, to obtain the numerical solution  for this class of equations, we have developed and analyzed two different kinds of efficient numerical schemes. First,  we constructed  a linearized  L1 Galerkin FEM  and derived the  convergence rate  of order  ($2$-$\alpha$) in the time direction. Next, to enhance the convergence order in the time direction we proposed a new  linearized L2-1$_{\sigma}$  Galerkin FEM. We have  proved that numerical solutions of this scheme  converge to the exact solution of the problem \eqref{6-10-22-1} with the  accuracy rate of $O(h+k^{2})$. Finally, numerical results revealed that theoretical error estimates are sharp.

\begin{comment}
\documentclass[a4paper, 12pt, twoside, onecolumn ]{article}
\usepackage[utf8]{inputenc}
\usepackage{amsmath, amssymb, amsthm,mathrsfs}
\usepackage{enumerate} 
\usepackage{graphicx}
\usepackage[top=1in, bottom=1.3in, left=1.3in, right=1in]{geometry}
\renewcommand{\baselinestretch}{1.2}
\usepackage{lmodern}
\usepackage[all]{xy}
\usepackage{tikz-cd} 
\usepackage{authblk}
 \usepackage{float}
 \usepackage{hyperref}
\date{}

\newtheorem{defn}{\bf Definition}[section] 
\newtheorem{exam}[defn]{\bf Example}
\newtheorem{exer}[defn]{\bf Exercise}
\newtheorem{prop}[defn]{\bf Proposition}
\newtheorem{lem}[defn]{\bf Lemma}
\newtheorem{thm}[defn]{\bf Theorem} 
\newtheorem{cor}[defn]{\bf Corollary}
\newtheorem{rmk}{\bf Remark}

\numberwithin{equation}{section}
\title{Finite Element Analysis of Time Fractional Integro-differential Equations of Kirchhoff type for Non-homogeneous Materials}
\author[$\star$]{Lalit Kumar}
\author[$\star$]{Sivaji Ganesh Sista}
\affil[$\star$]{Department of Mathematics, Indian Institute of Technology Bombay, Mumbai-400076, India}
\author[$\dagger$]{K. Sreenadh}
\affil[$\dagger$]{Department of Mathematics, Indian Institute of Technology Delhi, New Delhi-110016, India}	
\usepackage[pagewise]{lineno}
\linenumbers
\begin{document}

\maketitle
\begin{abstract}
	\noindent In this paper, we study an initial-boundary value problem of Kirchhoff type involving memory term for non-homogeneous materials $(D_{\alpha})$. The purpose of this research is threefold. First, we prove  the existence and uniqueness of weak solutions to the problem $(D_{\alpha})$ using the  Galerkin method. Second, to  obtain numerical solutions of $(D_{\alpha})$ efficiently, we develop a L1 type \cite{li2016analysis} backward Euler-Galerkin FEM, which is $O(h+k^{2-\alpha})$ accurate, where $\alpha~ (0<\alpha<1)$ is the order of fractional time derivative,  $h$ and $k$ are the  discretization parameters for space and time directions, respectively. Next, to achieve the optimal rate of convergence in time, we propose a fractional  Crank-Nicolson-Galerkin FEM based on L2-1$_{\sigma}$ scheme \cite{alikhanov2015new}. We prove that the  numerical solutions of this scheme converge to the solution of $(D_{\alpha})$ with accuracy $O(h+k^{2})$. We also  derive a priori bounds on numerical  solutions for the proposed schemes. Finally, some numerical experiments are conducted  to validate our theoretical claims.
	\end{abstract}
\textbf{Keywords:}
Nonlocal, Finite element method (FEM),  Fractional time derivatives, Fractional Crank-Nicolson scheme, Integro-differential equations.
\section{Introduction}
Let  $\Omega$ be a convex and bounded subset of $\mathbb{R}^{d}~(d\geq 1)$ with smooth boundary $\partial \Omega$ and $[0,T]$ is a fixed finite time interval. We consider the following integro-differential equation of Kirchhoff type involving fractional time derivative of order $\alpha ~(0<\alpha <1)$ for non-homogeneous materials 
\begin{align*}\tag{$D_{\alpha}$}
^{C}D^{\alpha}_{t}u- \nabla.\left(M\left(x,t,\|\nabla u\|^{2}\right)\nabla u\right)=f(x,t)+\int_{0}^{t}b(x,t,s)u(x,s)ds \quad \text{in}~ \Omega \times (0,T],
\end{align*}
with initial and boundary conditions
\begin{align*}
u(x,0)&=u_{0}(x) \quad \text{in} ~\Omega,\\
u(x,t)&=0 \quad \text{on} ~\partial \Omega \times [0,T],
\end{align*}
where  $u:=u(x,t) :\Omega \times [0,T]\rightarrow \mathbb{R} $ is the unknown function, $M:\Omega\times [0,T]\times \mathbb{R}^{+}\rightarrow \mathbb{R}^{+},$ initial data $ u_{0}$, source term $f$ are known functions and $b(x,t,s)$ is a memory operator to be defined later. $^{C}D_{t}^{\alpha}u$ in the problem  $(D_{\alpha})$ is the fractional time derivative of order $\alpha$ in the Caputo sense, which is defined in \cite{podlubny1998fractional} as
\begin{equation}\label{LS1}
^{C}D_{t}^{\alpha}u:=\frac{1}{\Gamma(1-\alpha)} \int_{0}^{t}\frac{1}{(t-s)^{\alpha}}\frac{\partial u(s)}{\partial s}ds,
\end{equation}
where $\Gamma(\cdot)$ denotes the gamma function. For the case $\alpha=1$, Lalit et al. \cite{kumar2020finite} proposed linearized backward Euler-Galerkin FEM and linearized Crank-Nicolson-Galerkin FEM with accuracy $O(h+k)$ and $O(h+k^{2})$, respectively.\\

\noindent There are various notions of  fractional derivatives other than the Caputo derivative, which include Riemann-Liouville, Gr\"{u}nwald-Letnikov, Weyl, Marchaud, and Riesz fractional derivatives \cite{ podlubny1998fractional, miller1993introduction}. Among these,  Caputo fractional derivative and Riemann-Liouville fractional derivative are the  most commonly found in literature \cite{podlubny1998fractional}. These two fractional derivatives are related to each other by the following relation, see \cite{podlubny1998fractional}
\begin{equation}\label{1.2}
^{C}D_{t}^{\alpha}u(t)=~^{R}D_{t}^{\alpha}\left(u(t)-u(0)\right),
\end{equation} 
where  $^{R}D_{t}^{\alpha}u$ is the Riemann-Liouville fractional derivative defined by 
\begin{equation}
^{R}D_{t}^{\alpha}u(t):=\frac{1}{\Gamma(1-\alpha)}\frac{d}{dt}\int_{0}^{t}\frac{1}{(t-s)^{\alpha}}u(s)ds .
\end{equation}
It is worth noting that the fractional derivatives other than
 Caputo fractional derivative  require initial condition containing the limiting value of fractional derivatives at $t=0$ \cite{podlubny1998fractional}, which has no physical interpretation. An advantage of choosing Caputo fractional derivative in the problem $(D_{\alpha})$ is that it allows the initial and boundary conditions in the same way as those for integer order differential equations. \\

\noindent There are many physical and biological  processes in which the mean-squared displacement of the particle motion grows only sublinearly with time $t$, instead of linear growth. For instance, acoustic wave propagation in viscoelastic materials \cite{mainardi2010fractional}, cancer invasion system \cite{manimaran2019numerical}, anomalous diffusion transport \cite{metzler2000random},  which cannot be  described  accurately by classical  models having integer order derivatives.  Therefore the study of fractional differential  equations has  evolved immensely in  recent years.\\

\noindent  Problems involving fractional time derivatives  have been studied by many researchers, for instance, see   \cite{schneider1989fractional,huang2005time, giga2017well}. 
Analytical solutions of fractional differential equations are expressed in terms of Fox $H$-functions, Green functions, and hypergeometric functions. Such special functions are more complex to compute, which restricts the applications of fractional calculus in applied sciences. This motivates the researchers to develop numerical algorithms for solving fractional differential equations.\\

\noindent There are two predominant discretization techniques in time for fractional differential equations, namely Gr\"{u}nwald-Letnikov approximation \cite{dimitrov2013numerical} and L1 type approximation \cite{alikhanov2015new,lin2007finite}. The second category, namely the L1 type approximation scheme, is based on piecewise interpolation of the integrand in definition \eqref{LS1} of the Caputo fractional derivative. Lin and Xu \cite{lin2007finite} developed the L1 scheme based on piecewise linear interpolation for Caputo fractional derivative and Legendre spectral method in space for the following time fractional PDE in one space dimension
\begin{equation}\label{1.3}
\begin{aligned}
^{C}D^{\alpha}_{t}u-\frac{\partial^{2}u}{\partial x^{2}}&=f(x,t)\quad x \in (0,L),~ t \in (0,T],\\
u(x,0)&=g(x) \quad x \in (0,L),\\
u(0,t)=u(L,t)&=0 \quad 0 \leq t \leq T,
\end{aligned}
\end{equation} 
and achieved the convergence estimates of $O(h^{2}+k^{2-\alpha})$ for the solutions belonging to $C^{2}\left([0,T];H^{2}(\Omega)\cap H^{1}_{0}(\Omega)\right)$. Recently, Alikhanov  \cite{alikhanov2015new} proposed a modification of the L1 type scheme in time direction and difference scheme in space direction for the problem \eqref{1.3}. In his work, the  author proved that the convergence rate is $O(h^{2}+k^{2})$ for the solutions belonging to $C^{3}\left([0,T];C^{4}(\Omega)\right)$.\\

\noindent On a similar note, there has been considerable attention devoted to the nonlocal diffusion problems where  diffusion depends on the entire domain rather than pointwise. Lions \cite{lions1978some} studied the following problem 
\begin{equation*}
    \frac{\partial^{2}u}{\partial t^{2}}-M\left(x,\int_{\Omega}|\nabla u|^{2}dx\right)\Delta u=f(x,t) \quad \text{in}\quad \Omega \times [0,T],
\end{equation*}
which models transversal oscillations of an elastic string or membrane by considering the change in length during vibrations. Also, the nonlocal terms appear in various physical and biological systems. For instance, the temperature in a thin region during friction welding \cite{kavallaris2007behaviour}, Kirchhoff equations with magnetic field \cite{mingqi2018critical}, Ohmic heating with variable thermal conductivity \cite{tzanetis2001nonlocal}, and many more. We cite \cite{mishra2015existence,goel2019kirchhoff,correa2004existence} for some contemporary works related to the  existence, uniqueness, and regularity of the problems involving nonlocal terms. \\

\noindent The models discussed above behave accurately only for a perfectly homogeneous medium, but in real-life situations, a large number of heterogeneities are present, which causes some memory effect or feedback term. These phenomena cannot be described by classical PDEs, which motivates us to study  time fractional PDEs for non-homogeneous materials. We note  that this class of equations has not been analyzed in the literature yet, and this is the first attempt to establish new results for the model $(D_{\alpha})$.\\

\noindent The salient feature of our problem is its doubly nonlocal nature due to the presence of Kirchhoff term and fractional derivatives or  memory term. The memory term incorporates the history of the phenomena under investigation, which is crucial in evolutionary processes. This structure induces additional difficulties in the use of classical methods of PDEs. In the  first part of this article, we show the existence and uniqueness of weak solutions for the problem $(D_{\alpha})$ by applying the Galerkin method. Further, we discretize the domain in space direction by using a conforming FEM \cite{thomee1984galerkin}, keeping the time variable continuous. We define a new Ritz-Volterra type operator as an intermediate function and  derive error estimates of $O(h)$ for semi-discrete  solutions of the problem $(D_{\alpha})$.\\

\noindent Finally, we find numerical solutions of $(D_{\alpha})$ by discretizing the domain in  the time direction also. Qiao et al. \cite{qiao2019adi} studied $(D_{\alpha})$ having weakly singular kernel and without Kirchhoff term and derived $O(h^{2}+k^{2-\alpha})$ convergence rate by using difference schemes in space and L1 type scheme in the time direction. For the case $b=0$ in $(D_{\alpha})$, Manimaran et al. \cite{manimaran2021finite} used L1 type methods for fractional derivative and Newton methods for nonlinearity  to obtain $O(h^{2}+k^{2-\alpha})$ accuracy  for sufficiently smooth solutions. \\

\noindent The novelty of our  work is to obtain the optimal rate of convergence $O(k^{2})$ in time for the numerical solutions of the problem $(D_{\alpha})$. For nonlinear PDEs, linearization techniques have the advantage over Newton's iterations method in terms of computational cost and storage. This is the motivation for us to develop a linearized fully discrete formulation based on the Galerkin FEM in space and the L1 type scheme in the time direction. This scheme  is proved to be accurate of $O(h+k^{2-\alpha})$. Further, to increase the accuracy in time, we propose a new linearized  numerical algorithm based on L2-1$_{\sigma}$ scheme. In this scheme, we  prove that numerical solutions converge to the exact solution of $(D_{\alpha})$ with the optimal  rate of accuracy $O(h+k^{2})$. The presence of Kirchhoff term makes the problem nonlinear which requires high computational efforts to obtain the numerical solutions of ($D_{\alpha}$) at each time step. Also the discretization of  memory term demands large computer storage. To resolve these issues, we  come up with a  linearization techniques for approximation of  Kirchhoff term and memory term is approximated by modified Simpson's rule \cite{pani1992numerical}. Further, some numerical experiments are carried out to confirm the theoretical findings. In addition, we remark that to the best of our knowledge, there is no article that has established convergence of  $O(h+k^{2})$ to the time fractional PDEs involving memory term.\\

\noindent Turning to the layout of this paper, in Section 2, we provide  some preliminaries and state the main contributions of this work. Section 3 includes the proof of the existence and uniqueness of weak solutions for the problem $(D_{\alpha})$. In Section 4, the corresponding semi-discrete formulation is studied. A priori bounds and error estimates for semi-discrete solutions are derived, which are uniform in time. In Section 5, we develop  a new L1 type Galerkin FEM and derive a priori bounds on numerical solutions of $(D_{\alpha})$ and prove that our numerical algorithm  has a convergence rate of $O(h+k^{2-\alpha})$. In Section 6, we establish the optimal rate of convergence $O(h+k^{2})$ in time by developing a new fractional Crank-Nicolson-Galerkin FEM. Section 7 contains numerical experiments that  reveal that theoretical results are sharp. Finally, we conclude this work in Section 8. \\

\noindent In this paper, for any  two quantities, $a$ and $b$, the symbol $a \lesssim b$ means  that there is a  generic positive constant $K$ such that $a\leq K b$. Here $K$  may vary at different  stages but is independent of mesh parameters $h$ and $k$.

\section{Preliminaries and main results}
This section of the article is intended to provide some preliminaries for time fractional PDEs, which will be used in subsequent  sections. Further, we state the main results of this article. \\
\noindent Let $L^{2}(\Omega)$ be the standard Lebesgue space of square integrable functions with norm $\|\cdot\|$, induced by the inner product $(\cdot,\cdot)$, and $H^{1}_{0}(\Omega)$ is a standard Hilbert space. Let us denote 
\begin{equation}\label{2.1}
k(t):=\frac{t^{-\alpha}}{\Gamma(1-\alpha)},
\end{equation}
and $\ast$ denotes the convolution of two  integrable functions $g$ and $h$ on $[0,T]$ as follows
\begin{equation}\label{354}
    (g\ast h)(t)=\int_{0}^{t}g(t-s)h(s)ds\quad \text{for}~~t~~\text{in}~~[0,T].
\end{equation}
\begin{rmk}\label{rmk1}
	Note that, $l(t)$ defined by  $l(t):=\frac{t^{\alpha-1}}{\Gamma(\alpha)}$ satisfies $k\ast l=1$.
\end{rmk}

\noindent Now using \eqref{2.1}, \eqref{354}, and \eqref{1.2}, we rewrite the problem $(D_{\alpha})$ in $\Omega \times (0,T]$ as 
\begin{equation}\tag{$R_{\alpha}$}
\begin{aligned}
\frac{d}{dt}\left[k\ast (u-u_{0})\right](t)- \nabla.\left(M\left(x,t,\|\nabla u\|^{2}\right)\nabla u\right)&=f(x,t)+\int_{0}^{t}b(x,t,s)u(x,s)ds,
\end{aligned}
\end{equation}
with initial and boundary conditions
\begin{equation*}
    \begin{aligned}
    u(x,0)&=u_{0}(x) \quad \text{in} ~\Omega,\\
    u(x,t)&=0 \quad \text{on} ~\partial \Omega \times [0,T].\\
    \end{aligned}
\end{equation*}

\noindent Throughout the paper, we assume the following hypotheses on data:\\
(H1)\quad $u_{0}\in H^{1}_{0}(\Omega)$ and $f \in L^{\infty}\left(0,T,L^{2}(\Omega)\right)$.\\
(H2)\quad $M: \Omega \times (0,T]\times \mathbb{R}^{+}\rightarrow \mathbb{R}^{+}$ is a Lipschitz continuous function such that  there exists a constant $m_{0}$ which satisfies 
\begin{equation*}\label{1.1}
M(x,t,s)\geq m_{0} >0 ~\text{for all }~ (x,t,s) \in \Omega \times (0,T]\times \mathbb{R}^{+} ~\text{and} \left(m_{0}-4L_{M}\hat{K}^{2}\right)>0,
\end{equation*}
where $\hat{K}=\left(\|\nabla u_{0}\|+\|f\|_{L^{\infty}\left(0,T,L^{2}(\Omega)\right)}\right)$ and $L_{M}$ is a Lipschitz constant.\\
(H3)\quad  $b(x,t,s)$ is a second order partial differential operator of the form of 
\begin{equation*}
b(x,t,s)u(x,s)=-\nabla\cdot(b_{2}(x,t,s)\nabla u(x,s))+\nabla \cdot(b_{1}(x,t,s)u(x,s))+b_{0}(x,t,s)u(x,s),
\end{equation*}
with $b_{2}:\Omega\times(0,T]\times (0,T] \rightarrow \mathbb{R}^{d\times d}$ is a matrix with entries $[b_{2}^{ij}(x,t,s)]$, $b_{1}:\Omega \times(0,T]\times (0,T] \rightarrow \mathbb{R}^{d}$ is a vector with entries $[b_{1}^{j}(x,t,s)]$ and  $b_{0}: \Omega\times (0,T]\times (0,T] \rightarrow \mathbb{R}$ is a scalar function. We assume that $b_{2}^{ij}, b_{1}^{j}, b_{0}$  are smooth functions  in all variables for $i,j=1,2,\dots,d$.\\

\noindent The notion of weak formulation for the problem $(R_{\alpha})$ is described as follows:\\
find 
$u$ in $L^{\infty}\left(0,T,L^{2}(\Omega)\right) \cap L^{2}\left(0,T,H^{1}_{0}(\Omega)\right)$ and $\frac{d}{dt}\left[k\ast (u-u_{0})\right](t)$ in $L^{2}\left(0,T,H^{-1}(\Omega)\right) $ such that the following holds for all $v$ in $H^{1}_{0}(\Omega)$ and $a.e. ~t$ in $(0,T]$
\begin{equation}\tag{$W_{\alpha}$}
\begin{aligned}
\left(\frac{d}{dt}\left[k\ast (u-u_{0})\right](t),v\right)&+ \left(M\left(x,t,\|\nabla u\|^{2}\right)\nabla  u,\nabla v\right)\\
&=\left(f,v\right)+\int_{0}^{t}B(t,s,u,v)ds~ \text{in}~ \Omega \times (0,T],
\end{aligned}
\end{equation}
with initial and boundary conditions
\begin{equation*}
\begin{aligned}
u(x,0)&=u_{0}(x) \quad \text{in} ~\Omega,\\
u(x,t)&=0 \quad \text{on} ~\partial \Omega \times [0,T],
\end{aligned}
\end{equation*}
where 
\begin{equation*}
\begin{aligned}
B(t,s,u,v)&=\left( b_{2}(x,t,s)\nabla u(s),\nabla v\right)+\left(\nabla \cdot( b_{1}(x,t,s)u(s)),v\right)+\left(b_{0}(x,t,s)u(s),v\right),\\
\left(b_{2}(x,t,s)\nabla u(s),\nabla v\right)&=\int_{\Omega}b_{2}(x,t,s)\nabla u(s)\cdot \nabla v dx,\\
\left(\nabla \cdot (b_{1}(x,t,s)u(s)),v\right)&=\int_{\Omega}\left(\nabla \cdot (b_{1}(x,t,s)u(s))\right)v dx,\\
\left(b_{0}(x,t,s)u(s),v\right)&
=\int_{\Omega}b_{0}(x,t,s)u(s)v dx.
\end{aligned}
\end{equation*}
The assumptions on the coefficients $b_{0}, b_{1}$, and $b_{2}$ imply that there exists a positive constant $B_{0}$ such that for $(t,s)$ in $(0,T)\times (0,T)$ and $u,v$ in $H^{1}_{0}(\Omega)$ the following inequality holds 
\begin{equation}\label{2.3}
|B(t,s,u,v)|\leq B_{0}\|\nabla u\|\|\nabla v\| 
\end{equation} 
We remark that the following examples of $b(x,t,s)$ satisfy (H3) and \eqref{2.3}.
\begin{exam}
	Let $\Omega$ be a bounded domain in $\mathbb{R}^{2}, x=(x_{1},x_{2})$ in $\Omega$ and $t,s$ in $[0,T]$
	\begin{enumerate}
		\item $b_{2}(x,t,s)=e^{t-s}\begin{bmatrix}
		-1& 0\\
		0 &-1
		\end{bmatrix}$; $b_{1}(x,t,s)=b_{0}(x,t,s)=0.$
		\item  $b_{2}(x,t,s)=e^{t-s}\begin{bmatrix}
		0& -x_{1}x_{2}\\
		-x_{1}x_{2} &0
		\end{bmatrix}$; $b_{1}(x,t,s)=e^{t-s}\left(x_{1}^{2},x_{2}^{2}\right);b_{0}(x,t,s)=0.$
	\end{enumerate}
	These examples were used for numerical experiments in the context of parabolic integro-differential equations in \cite{barbeiro2011h1}.
\end{exam} 

\noindent We use the following version of continuous Gr\"{o}nwall's inequality to derive the bounds on the weak  solutions of $(D_{\alpha})$.
\begin{lem}\label{lem2.5}\cite{almeida2017gronwall}
	Let $\alpha$ be such that $0<\alpha<1$ and  suppose that $u$ and $v$ are two non-negative integrable functions on $[a,b]$, $v$ is nondecreasing and $g$ is continuous function in $[a,b]$. If
	\begin{equation*}
	u(t)\leq v(t)+g(t)\int_{0}^{t}(t-s)^{\alpha-1}u(s)ds \quad \text{for all}~ t ~~\text{in}~~ [a,b],
	\end{equation*}
	then 
	\begin{equation*}
	u(t)\leq v(t)E_{\alpha}\left[g(t)\Gamma(\alpha)t^{\alpha}\right] \quad \text{for all}~ t ~~\text{in}~~ [a,b],
	\end{equation*}
	where $E_{\alpha}(\cdot)$ is the one parameter Mittag-Leffler function \cite{podlubny1998fractional}.
\end{lem}

 \noindent Further, for the semi-discrete formulation of the problem  $(D_{\alpha})$, we discretize the domain in space variable by a conforming FEM \cite{thomee1984galerkin}.
 Let $\mathbb{T}_{h}$ be a shape regular (non overlapping), quasi-uniform triangulation of $\Omega$ and $h$ be the discretization parameter in  the space direction. We define a  finite dimensional subspace $X_{h}$ of $H^{1}_{0}(\Omega)$ as 
 $$X_{h}=\{v_{h}\in C(\bar \Omega)~:~v_{h}|_{\tau}~ \text{is a linear polynomial for all}~ \tau \in \mathbb{T}_{h} ~\text{and} ~v_{h}=0 ~\text{on}~\partial \Omega\}.$$
 
 \noindent Now semi-discrete formulation for the problem $(D_{\alpha})$ is to find $u_{h}$ in $X_{h}$ such that the following holds for all $v_{h}$ in $X_{h}$ and a.e. $t$ in $(0,T]$ 
 \begin{equation}\tag{$S_{\alpha}$}
 \begin{aligned}
 \left(^{C}D^{\alpha}_{t}u_{h},v_{h}\right) &+\left(M\left(x_{h},t,\|\nabla u_{h}\|^{2}\right)\nabla u_{h},\nabla v_{h}\right)\\
 &=(f_{h},v_{h})+\int_{0}^{t}B(t,s,u_{h},v_{h})ds ~~~\text{in}~~\mathbb{T}_{h} \times (0,T],
 \end{aligned}
 \end{equation}
 with initial and boundary conditions
 \begin{equation*}
 \begin{aligned}
 u_{h}(x,0)&=R_{h}u_{0} ~~\text{in} ~~\Omega,\\
 u_{h}&=0 ~~\text{on}~~\partial \Omega\times [0,T],
 \end{aligned}
 \end{equation*}
 where $f_{h}=f(x_{h},t)$ for $x_{h}$ in $\mathbb{T}_{h}$ and  $R_{h}u_{0}$ is the Ritz projection of $u_{0}$ on $X_{h}$ defined by
 \begin{equation}\label{Ad2.3}
     \left(\nabla R_{h}u_{0},\nabla v_{h}\right)=\left(\nabla u_{0},\nabla v_{h}\right)~~\text{ for all}~~ v_{h} ~~\text{in}~~X_{h}.
 \end{equation}

\noindent Now we move to the fully discrete formulation of $(D_{\alpha})$, for that we divide the interval $[0,T]$  into sub intervals of uniform step size $k$ and $t_{n}=nk$ for $n=1,2,3,\dots,N$ with $t_{N}=T$. We approximate Caputo fractional derivative by the following schemes.\\
\noindent  \textbf{L1 Approximation Scheme \cite{li2016analysis}:} In this scheme, Caputo fractional derivative is approximated at the point $t_{n}$ using linear interpolation or backward Euler difference method  as follows
\begin{equation}\label{2.9}
\begin{aligned}
^{C}D^{\alpha}_{t_{n}}u~&=\frac{1}{\Gamma(1-\alpha)}\int_{0}^{t_{n}}\frac{1}{(t_{n}-s)^{\alpha}}\frac{\partial u}{\partial s}ds\\
&=\frac{1}{\Gamma(1-\alpha)}\sum_{j=1}^{n}\frac{u^{j}-u^{j-1}}{k}\int_{t_{j-1}}^{t_{j}}\frac{1}{(t_{n}-s)^{\alpha}}~ds+\mathbb{Q}^{n}\\
&=\frac{k^{-\alpha}}{\Gamma(2-\alpha)}\sum_{j=1}^{n}a_{n-j}\left(u^{j}-u^{j-1}\right)+\mathbb{Q}^{n}
\end{aligned}
\end{equation}
where $a_{i}=(i+1)^{1-\alpha}-i^{1-\alpha},~i\geq0$, $u^{j}=u(x,t_{j})$, and $\mathbb{Q}^{n}$ is the truncation error. \\

\noindent \textbf{L2-1$_{\sigma}$ Approximation Scheme \cite{alikhanov2015new}:} In this scheme, Caputo fractional derivative is approximated at the point $t_{n-\frac{\alpha}{2}}$ using quadratic interpolation on $[0,t_{n-1}]$ and linear interpolation on $[t_{n-1 },t_{n-\frac{\alpha}{2}}]$ as  follows

\begin{equation*}\label{2.11}
^{C}D^{\alpha}_{t_{n-\frac{\alpha}{2}}}=\tilde{\mathbb{D}}^{\alpha}_{t_{n-\frac{\alpha}{2}}}u+\tilde{\mathbb{Q}}^{n-\frac{\alpha}{2}}
\end{equation*}
where 
\begin{equation}\label{1.20} \tilde{\mathbb{D}}^{\alpha}_{t_{n-\frac{\alpha}{2}}}u=\frac{k^{-\alpha}}{\Gamma(2-\alpha)}\sum_{j=1}^{n}\tilde{c}_{n-j}^{(n)}\left(u^{j}-u^{j-1}\right) 
\end{equation}
with weights $\tilde{c}_{n-j}^{(n)}$ satisfying $\tilde{c}_{0}^{(1)}=\tilde{a}_{0}$ for $n=1$ and for $n\geq 2$
\begin{equation}
\tilde{c}_{j}^{(n)}=\begin{cases}
\tilde{a}_{0}+\tilde{b}_{1}, & j=0,\\
\tilde{a}_{j}+\tilde{b}_{j+1}-\tilde{b}_{j}, & 1\leq j \leq n-2,\\
\tilde{a}_{j}-\tilde{b}_{j},& j=n-1,
\end{cases}
\end{equation}
where 
\begin{equation*}
\tilde{a}_{0}=\sigma^{1-\alpha}~\text{and}~
\tilde{a}_{l}=\left(l+\sigma\right)^{1-\alpha}-\left(\sigma\right)^{1-\alpha} \quad l\geq 1,
\end{equation*}
\begin{equation*}
\tilde{b}_{l}=\frac{1}{(2-\alpha)}\left[\left(l+\sigma\right)^{2-\alpha}-\left(l-\frac{\alpha}{2}\right)^{2-\alpha}\right]-\frac{1}{2}\left[\left(l+\sigma\right)^{1-\alpha}+\left(l-\frac{\alpha}{2}\right)^{1-\alpha}\right]~l\geq 1,
\end{equation*}
with $\sigma=\left(1-\frac{\alpha}{2}\right)$  and $\tilde{\mathbb{Q}}^{n-\frac{\alpha}{2}}$ is the  truncation error.\\

\noindent Similar to the continuous case (Remark-\ref{rmk1}), there exists a discrete kernel corresponding to above schemes  which satisfies the following properties.  

\begin{lem}\label{L1.2} \textbf{(For L1 type approximation \cite{li2016analysis})}
	Let $p_{n}$ be a sequence defined by 
	\begin{equation*}
	p_{0}=1,~~	p_{n}=\sum_{j=1}^{n}(a_{j-1}-a_{j})p_{n-j}~~\text{for}~~n\geq 1.
	\end{equation*}
	Then $p_{n}$ satisfies
		 \begin{equation}
0<p_{n}<1,
		\end{equation} 
		\begin{equation}\label{2.81}
		\sum_{j=k}^{n}p_{n-j}a_{j-k}=1,\quad 1\leq k \leq n,
		\end{equation}
		\begin{equation}\label{2.91}
		\Gamma(2-\alpha)\sum_{j=1}^{n}p_{n-j}\leq \frac{n^{\alpha}}{\Gamma(1+\alpha)}.
		\end{equation}	
\end{lem}
\begin{lem}\label{L1.21} \textbf{(For L2-1$_{\sigma}$ type approximation \cite{liao2019discrete})} Define
	\begin{equation*}
	\tilde{p}_{0}^{(n)}=\frac{1}{\tilde{c}_{0}^{(n)}},~~	\tilde{p}_{j}^{(n)}=\frac{1}{\tilde{c}_{0}^{(n-j)}}\sum_{k=0}^{j-1}\left(\tilde{c}_{j-k-1}^{(n-k)}-\tilde{c}_{j-k}^{(n-k)}\right)\tilde{p}_{k}^{(n)}~~\text{for}~~1\leq j\leq n-1.
	\end{equation*}
	Then $\tilde{p}_{j}^{(n)}$ satisfies
		 \begin{equation}
0< \tilde{p}_{n-j}^{(n)} <1,
		\end{equation} 
		\begin{equation}\label{2.812}
		\sum_{j=k}^{n}\tilde{p}_{n-j}^{(n)}\tilde{c}_{j-k}^{(j)}=1,\quad 1\leq k \leq n \leq N,
		\end{equation}
		\begin{equation}\label{2.914}
		\Gamma(2-\alpha)\sum_{j=1}^{n}\tilde{p}_{n-j}^{(n)}\leq \frac{n^{\alpha}}{\Gamma(1+\alpha)}, \quad 1\leq n \leq N.
		\end{equation}	
\end{lem}

\noindent With this background, we state the main contributions of this article.
\begin{thm}\label{t2.5}
	\textbf{(Well-posedness of weak formulation (W$_{\alpha}$))} Under the hypotheses of (H1), (H2), and (H3),  the problem $(W_{\alpha})$ admits a unique  solution  which satisfies the following a priori bounds
	\begin{equation*}
	\|u\|^{2}_{L^{\infty}\left(0,T,L^{2}(\Omega)\right)}+\frac{1}{\Gamma(\alpha)}\int_{0}^{t}(t-s)^{\alpha-1}\|\nabla u(s)\|^{2}ds \lesssim \left(\|\nabla u_{0}\|^{2}+\|f\|^{2}_{L^{\infty}(0,T,L^{2}(\Omega))}\right),
	\end{equation*}
	\begin{equation*}
	\|u\|^{2}_{L^{\infty}(0,T,H^{1}_{0}(\Omega))}+\frac{1}{\Gamma(\alpha)}\int_{0}^{t}(t-s)^{\alpha-1}\|\Delta u(s)\|^{2}ds \lesssim \left(\|\nabla u_{0}\|^{2}+\|f\|^{2}_{L^{\infty}(0,T,L^{2}(\Omega))}\right).
	\end{equation*} 
\end{thm}
\noindent Now we derive  the error estimates of semi-discrete solutions and fully discrete solutions by assuming that the  solution of $(D_{\alpha})$ is in $C^{4}\left([0,T];H^{2}(\Omega)\cap H^{1}_{0}(\Omega)\right)$.
\begin{thm}\label{t2.6}
	\textbf{(Error estimates for semi-discrete  formulation (S$_{\alpha}$))} Suppose that  the hypotheses (H1), (H2), and (H3) hold. Then  we have the following error estimates  for the solution  $u_{h}$ 
of the problem 	$(S_{\alpha})$, which is  uniform in time
\begin{equation*}
	\|u-u_{h}\|_{L^{\infty}\left(0,T,L^{2}(\Omega)\right)}+\left(\frac{1}{\Gamma(\alpha)}\int_{0}^{t}(t-s)^{\alpha-1}\|\nabla u(s)-\nabla u_{h}(s)\|^{2}ds\right)^{1/2} \lesssim  h
	\end{equation*}
	provided that $u(t)$ is in $H^{2}(\Omega)\cap H^{1}_{0}(\Omega)$ for a.e. $t$ in $[0,T]$ and triangulation of the domain is quasi-uniform.
\end{thm}
\noindent Next, we  develop the two different types of numerical algorithms to find the numerical solutions of the problem $(D_{\alpha})$.

\noindent \textbf{ Backward Euler-Galerkin FEM based on L1 scheme:} Find $u_{h}^{n} ~(n=1,2,3,\dots,N)$ in $X_{h}$ with $\bar{u}_{h}^{n-1}=2u_{h}^{n-1}-u_{h}^{n-2}$ such that the following equations hold for all $v_{h}$ in $X_{h}$ \\
For $n\geq 2,$
\begin{equation}\tag{$E_{\alpha}$}
\left(\mathbb{D}^{\alpha}_{t}u_{h}^{n},v_{h}\right)+\left(M\left(x_{h},t_{n},\|\nabla \bar{u}_{h}^{n-1}\|^{2}\right)\nabla u_{h}^{n},\nabla v_{h}\right)=\left(f_{h}^{n},v_{h}\right)+\sum_{j=1}^{n-1}w_{nj}B\left(t_{n},t_{j},u_{h}^{j},v_{h}\right).
\end{equation}
For $n=1$,
\begin{equation*}
    \begin{aligned}
\left(\mathbb{D}^{\alpha}_{t}u_{h}^{1},v_{h}\right)+\left(M\left(x_{h},t_{1},\|\nabla u_{h}^{1}\|^{2}\right)\nabla u_{h}^{1},\nabla v_{h}\right)=\left(f_{h}^{1},v_{h}\right)+kB\left(t_{1},t_{0},u_{h}^{0},v_{h}\right),
\end{aligned}
\end{equation*}
 with initial and boundary conditions
\begin{equation*}
  \begin{aligned}
u_{h}^{0}&=R_{h}u_{0}~~\text{in}~~~X_{h},\\
u_{h}^{n}&=0 ~~\text{on}~~ \partial \Omega ~~~\text{for}~~n=1,2,3,\dots,N.
\end{aligned}
\end{equation*}
where $\mathbb{D}^{\alpha}_{t}u_{h}^{n}$ is the approximation operator defined in \eqref{2.9} as 
\begin{equation}\label{S1.18}
\begin{aligned}
\mathbb{D}^{\alpha}_{t}u_{h}^{n}&=\frac{k^{-\alpha}}{\Gamma(2-\alpha)}\sum_{j=1}^{n}a_{n-j}\left(u_{h}^{j}-u_{h}^{j-1}\right),
\end{aligned}
\end{equation}
 \noindent and $w_{nj}$ are the quadrature weights satisfying the  following modified Simpson's rule \cite{pani1992numerical}. Let $m_{1}=[k^{-1/2}]$, where $[\cdot]$ denotes the greatest integer  function. Set $k_{1}=m_{1}k$ and $\bar{t}_{j}=jk_{1}$. Let $j_{n}$ be the largest even integer such that $\bar{t}_{j_{n}} <t_{n}$ and introduce quadrature points 
 \[
\bar{t}_{j}^{n}=\begin{cases}
 &jk_{1},\quad 0\leq j\leq j_{n},\\
 &\bar{t}_{j}^{n}+(j-j_{n})k,\quad j_{n}\leq j\leq J_{n},
 \end{cases}
 \]
 where $\bar{t}_{J_{n}}^{n}=t_{n-1}$.
 Then quadrature rule for any function $g$ is given by 
\begin{equation}\label{S1.14}
\begin{aligned}
\int_{0}^{t_{n}}g(s)~ds&=\sum_{j=0}^{n-1}w_{nj}g(t_{j})+q^{n}(g)\\
&=\frac{k_{1}}{3}\sum_{j=1}^{j_{n}/2}\left[g(\bar{t}_{2j}^{n})+4g(\bar{t}_{2j-1}^{n})+g(\bar{t}_{2j-2}^{n})\right]\\
&+\frac{k}{2}\sum_{j=j_{n}+1}^{J_{n}}\left[g(\bar{t}_{j}^{n})+g(\bar{t}_{j-1}^{n})\right]+kg(\bar{t}_{J_{n}}^{n})+q^{n}(g)
\end{aligned}
\end{equation}
where $q^{n}(g)$ is the quadrature error associated with the function $g$ at $t_{n}$. 
\begin{thm}\label{J1}
	\textbf{(Convergence Estimates for the Scheme (E$_{\alpha}$))} Under the hypotheses of (H1), (H2), and (H3), the fully discrete solution  $u_{h}^{n}~(1\leq n \leq N)$ of $(E_{\alpha})$ converges to the solution $u$ of $(D_{\alpha})$ with the following rate of accuracy
	\begin{equation*}
	\max_{1\leq n \leq N}\|u(t_{n})-u_{h}^{n}\|+\left(k^{\alpha}\sum_{n=1}^{N}p_{N-n}\|\nabla u(t_{n}) -\nabla u_{h}^{n}\|^{2}\right)^{1/2}\lesssim (h+k^{2-\alpha}).
	\end{equation*}
	
\end{thm}
\noindent  At this point we observe that the  convergence rate is of $O(k^{2-\alpha})$ in temporal direction which is not optimal. Thus to obtain the optimal rate of  convergence, we propose  the following numerical scheme. 

\noindent \textbf{Fractional Crank-Nicolson-Galerkin FEM based on L2-1$_\sigma$ scheme:} Find $u_{h}^{n}~(n=1,2,3\dots,N)$ in $X_{h}$ with $\bar{u}_{h}^{n-1,\alpha}=\left(2-\frac{\alpha}{2}\right)u_{h}^{n-1}-\left(1-\frac{\alpha}{2}\right)u_{h}^{n-2}$ and $\hat{u}_{h}^{n,\alpha}=\left(1-\frac{\alpha}{2}\right)u_{h}^{n}+\left(\frac{\alpha}{2}\right)u_{h}^{n-1}$ such that the following equations hold for all $v_{h}$ in $X_{h}$\\
For $n\geq 2$,
\begin{equation}\tag{$F_{\alpha}$}
\begin{aligned}
\left(\tilde{\mathbb{D}}^{\alpha}_{t_{n-\frac{\alpha}{2}}}u_{h}^{n},v_{h}\right)+\left(M\left(x_{h},t_{n-\frac{\alpha}{2}},\|\nabla \bar{u}_{h}^{n-1,\alpha}\|^{2}\right)\nabla \hat{u}_{h}^{n,\alpha},\nabla v_{h}\right)&=\sum_{j=1}^{n-1}\tilde{w}_{nj}B\left(t_{n-\frac{\alpha}{2}},t_{j},u_{h}^{j},v_{h}\right)\\
&+\left(f_{h}^{n-\frac{\alpha}{2}},v_{h}\right).
\end{aligned}
\end{equation}
For $n=1$,
\begin{equation*}
\begin{aligned}
\left(\tilde{\mathbb{D}}^{\alpha}_{t_{1-\frac{\alpha}{2}}}u_{h}^{1},v_{h}\right)+\left(M\left(x_{h},t_{1-\frac{\alpha}{2}},\|\nabla u_{h}^{1}\|^{2}\right)\nabla u_{h}^{1},\nabla v_{h}\right)&=\left(1-\frac{\alpha}{2}\right)kB\left(t_{1-\frac{\alpha}{2}},t_{0},u_{h}^{0},v_{h}\right)\\
&+\left(f_{h}^{1-\frac{\alpha}{2}},v_{h}\right),
\end{aligned}
\end{equation*}
 with initial and boundary conditions
\begin{equation*}
  \begin{aligned}
u_{h}^{0}&=R_{h}u_{0}~~\text{in}~~~X_{h},\\
u_{h}^{n}&=0 ~~\text{on}~~ \partial \Omega ~~~\text{for}~~n=1,2,3,\dots,N,
\end{aligned}
\end{equation*}
where $\tilde{\mathbb{D}}^{\alpha}_{t_{n-\frac{\alpha}{2}}}u_{h}^{n}$ is the approximation operator defined in \eqref{1.20} and $\tilde{w}_{nj}$ is the quadrature weights satisfying the quadrature formula similar to \eqref{S1.14} with small modification as follows
\begin{equation}\label{S1.14AD}
\begin{aligned}
\int_{0}^{t_{n}}g(s)~ds&=\sum_{j=0}^{n-1}\tilde{w}_{nj}g(t_{j})+\tilde{q}^{n-\frac{\alpha}{2}}(g)\\
&=\frac{k_{1}}{3}\sum_{j=1}^{j_{n}/2}\left[g(\bar{t}_{2j}^{n})+4g(\bar{t}_{2j-1}^{n})+g(\bar{t}_{2j-2}^{n})\right]\\
&+\frac{k}{2}\sum_{j=j_{n}+1}^{J_{n}}\left[g(\bar{t}_{j}^{n})+g(\bar{t}_{j-1}^{n})\right]+\left(1-\frac{\alpha}{2}\right)kg(\bar{t}_{J_{n}}^{n})+\tilde{q}^{n-\frac{\alpha}{2}}(g),
\end{aligned}
\end{equation}
where $\tilde{q}^{n-\frac{\alpha}{2}}(g)$ is the quadrature error associated with the function $g$ at $t_{n-\frac{\alpha}{2}}$. 
\begin{thm}\label{J2}
	\textbf{(Convergence Estimates for the Scheme (F$_{\alpha}$))} Under the assumptions of (H1), (H2), and (H3), the fully discrete solution  $u_{h}^{n}~(1\leq n \leq N)$ of $(F_{\alpha})$ satisfies the following convergence estimates
	\begin{equation*}
	\max_{1\leq n \leq N}\|u(t_{n})-u_{h}^{n}\|+\left(k^{\alpha}\sum_{n=1}^{N}\tilde{p}_{N-n}^{(N)}\|\nabla u(t_{n})-\nabla u_{h}^{n}\|^{2}\right)^{1/2}\lesssim (h+k^{2}).
	\end{equation*}
\end{thm}
\begin{rmk}
	We observe that as $\alpha \rightarrow 1$, the schemes  $(E_{\alpha})$ and $(F_{\alpha})$ reduces to linearized backward Euler-Galerkin FEM and linearized Crank-Nicolson-Galerkin FEM respectively \cite{kumar2020finite}.
\end{rmk} 
\section{Weak solutions of ($D_{\alpha}$)}
In this section, we prove the existence and uniqueness of weak solutions for the problem $(R_{\alpha})$ using the Galerkin method. For this, we study the  variational formulation of $(R_{\alpha})$ and derive a priori bounds on every Galerkin sequence. These a priori bounds establish the convergence of the Galerkin sequence using  compactness lemma for time fractional PDEs \cite{li2018some}.
\noindent The following lemma discusses the approximation of the kernel $k$ defined in \eqref{2.1}.
\begin{lem}\cite{zacher2009weak} \label{L3.1}Let $k$ be the kernel defined in \eqref{2.1}, then there exists a sequence of kernels $k_{n}$ in $W^{1,1}(0,T)$ such that
\begin{enumerate}
    \item $k_{n}=ns_{n}$, with $s_{n}$ being the solution of scalar-valued Volterra equations
    \begin{equation*}
        s_{n}(t)+n(l\ast s_{n})(t)=1~~~t>0,~~~n\in \mathbb{N},
    \end{equation*}
    \item $k_{n}$ converges to $k$ in $L^{1}(0,T)$ as $n$ tends to infinity,
    \item $k_{n}$ is non-negative and non-increasing in $(0,\infty)$.
\end{enumerate}
    
\end{lem}
\begin{subsection}{Proof of the Theorem \ref{t2.5}}
\begin{proof}
	\textbf{(Existence)}~Let $\{\lambda_{i}\}_{i=1}^{m}$ be the eigenvalues and $\{\phi_{i}\}_{i=1}^{m}$ be the corresponding  eigenfunctions  of the Dirichlet problem for the standard Laplacian operator in $H^{1}_{0}(\Omega)$. We consider the finite dimensional subspace $\mathbb{V}_{m}$ of $H^{1}_{0}(\Omega)$ such that $\mathbb{V}_{m}= \text{span} \{\phi_{1}, \phi_{2}, \dots , \phi_{m}\}.$  We assume that 
	\begin{equation}\label{3.1}
	u_{m}(\cdot,t)=\sum_{j=1}^{m}\alpha_{mj}(t)\phi_{j}\quad \text{and }\quad  u_{m}(\cdot,0)=\sum_{j=1}^{m}(u_{0},\phi_{j})\phi_{j},
	\end{equation}
	then $u_{m}(\cdot,0)$ converges to $u_{0}$. Now we  show that for each $m$ there is $u_{m} \in \mathbb{V}_{m}$,  satisfying the following equation for all $v_{m} \in \mathbb{V}_{m}~ \text{and}~ a.e. ~t \in (0,T]$
	\begin{equation}\label{3.2}
	\begin{aligned}
	\left(\frac{d}{dt}\left[k\ast (u_{m}-u_{m}(0))\right](t),v_{m}\right)&+\left(M(x,t,\|\nabla u_{m}\|^2)\nabla u_{m}, \nabla v_{m}\right)\\
	&=(f,v_{m})+ \int_{0}^{t}B(t,s,u_{m},v_{m})ds.
	\end{aligned}
	\end{equation}
	Put the values of $u_{m}$ and $u_{m}(0)$ in \eqref{3.2}, we obtain a system of fractional order differential equations. Then by the theory of fractional order differential equations, the system \eqref{3.2} has a continuous solution $u_{m}(t)$ on some interval $[0,t_{n}), 0<t_{n}< T,$ with vanishing trace of  $k\ast(u_{m}-u_{m}(0))$ at $t=0$ \cite{zacher2009weak}. These local solutions $u_{m}(t)$ are extended to the whole interval by using the following a priori bounds on $u_{m}$. Put $v_{m}=u_{m}(t)$  in \eqref{3.2},  we obtain
	\begin{equation}\label{3.3}
	\begin{aligned}
	\left(\frac{d}{dt}\left[k\ast (u_{m}-u_{m}(0))\right](t),u_{m}\right)&+\left(M(x,t,\|\nabla u_{m}\|^2)\nabla u_{m}, \nabla u_{m}\right)\\
	&=(f,u_{m})+ \int_{0}^{t}B(t,s,u_{m},u_{m})ds.
	\end{aligned}
	\end{equation}
	Let $k_{n}, n\in \mathbb{N}$ be the kernels  defined in Lemma \ref{L3.1}, then equation \eqref{3.3} is rewritten as  
		\begin{equation}\label{3.4Ad}
	\begin{aligned}
	\left(\frac{d}{dt}(k_{n}\ast u_{m})(t),u_{m}\right)&+\left(M(x,t,\|\nabla u_{m}\|^2)\nabla u_{m}, \nabla u_{m}\right)\\
	&=(f,u_{m})+ \int_{0}^{t}B(t,s,u_{m},u_{m})ds+k_{n}(t)(u_{m}(0),u_{m})+h_{mn}(t),
	\end{aligned}
	\end{equation}
	with
	\begin{equation}
	    h_{mn}(t)=	\left(\frac{d}{dt}\left[k_{n}\ast (u_{m}-u_{m}(0))\right](t)-	\frac{d}{dt}\left[k\ast (u_{m}-u_{m}(0))\right](t),u_{m}\right).
	\end{equation}
	Using Lemma 2.1 from \cite{zacher2009weak} and (H2), (H3)  together with Cauchy-Schwarz and  Young's inequality, we get
	\begin{equation}\label{3.4}
	\begin{aligned}
	\frac{d}{dt}\left(k_{n}\ast \|u_{m}\|^{2}\right)(t)+m_{0}\|\nabla u_{m}\|^{2}&\leq \|f\|^{2}+                    \|u_{m}\|^{2}+k_{n}(t)\|u_{m}(0)\|^{2}\\
	&+\frac{B_{0}^{2}T}{m_{0}}\int_{0}^{t}\|\nabla u_{m}(s)\|^{2}ds+2h_{mn}(t).
	\end{aligned}
	\end{equation} 
	Now convolve \eqref{3.4} with the kernel $l$ defined in Remark \ref{rmk1} and use the fact that 
	\begin{equation}
	    l\ast \frac{d}{dt}\left(k_{n}\ast \|u_{m}\|^{2}\right)(t)=\frac{d}{dt}\left(k_{n}\ast l \ast \|u_{m}\|^{2}\right)(t)\rightarrow \frac{d}{dt}\left(k\ast l\ast \|u_{m}\|^{2}\right)(t)=\|u_{m}\|^{2},
	\end{equation}
	with 
	\begin{equation}
	    (l\ast k_{n})(t)\rightarrow (l \ast k)(t)= 1 ~~\text{and}~~(l\ast h_{mn})(t)\rightarrow 0~~\text{as}~~n \rightarrow \infty ~~\text{in}~~L^{1}(0,T),
	\end{equation}
	we get 
	\begin{equation*}
	\begin{aligned}
	\|u_{m}\|^{2}+\left(l\ast\|\nabla u_{m}\|^{2}\right)(t)&\lesssim \int_{0}^{t}(t-s)^{\alpha-1}\left(\|u_{m}(s)\|^{2}+\left(l\ast\|\nabla u_{m}\|^{2}\right)(s)\right)ds\\ &+\left(l\ast\|f\|^{2}\right)(t)+\|u_{0}\|^{2}.
	\end{aligned}
	\end{equation*} 
	Using Lemma \ref{lem2.5} and Poincar\'{e} inequality, we conclude that
	\begin{equation}\label{3.5}
	\begin{aligned}
	\|u_{m}\|^{2}_{L^{\infty}\left(0,T,L^{2}(\Omega)\right)}+\left(l\ast\|\nabla u_{m}\|^{2}\right)(t)&\lesssim \left(\| \nabla u_{0}\|^{2}+\|f\|^{2}_{L^{\infty}\left(0,T,L^{2}(\Omega)\right)}\right).
	\end{aligned}
	\end{equation} 
	Further, we introduce discrete Laplacian $\Delta_{m}: \mathbb{V}_{m}\rightarrow \mathbb{V}_{m}$ as follows
	\begin{equation}
	\left(-\Delta_{m}u_{m},v_{m}\right):=\left(\nabla u_{m},\nabla v_{m}\right)\quad \text{for all}~ u_{m},v_{m} ~\text{in}~ \mathbb{V}_{m}.
	\end{equation}
	Now take $v_{m}=-\Delta_{m}u_{m}$ in \eqref{3.2} and follow the same arguments as we prove estimate \eqref{3.5} to achieve 
	\begin{equation}\label{3.9}
	\|u_{m}\|^{2}_{L^{\infty}\left(0,T,H^{1}_{0}(\Omega)\right)}+\left(l\ast\|\Delta_{m}u_{m}\|^{2}\right)(t) \lesssim\left(\|\nabla u_{0}\|^{2}+\|f\|^{2}_{L^{\infty}\left(0,T,L^{2}(\Omega)\right)}\right)=\hat{K}.
	\end{equation}
	Moreover, for any $v \in H^{1}_{0}(\Omega)$, we write $v=v^{1}+v^{2}$ where $v^{1}\in$ span$\{\phi_{j}\}_{j=1}^{m}$ and $v^{2}$ satisfies $(\nabla v^{2},\nabla \phi_{j})=0$ for $j=1,2,3,\dots,m$, as $\{\phi_{j}\}_{j=1}^{m}$ are orthogonal in $H^{1}_{0}(\Omega)$. So from equation \eqref{3.2} and (H2) we have 
	\begin{equation*}
	\begin{aligned}
	\left(\frac{d}{dt}\left[k\ast (u_{m}-u_{m}(0))\right](t),v^{1}\right)+m_{0}(\nabla u_{m}, \nabla v^{1})&\leq (f,v^{1})+\int_{0}^{t}B(t,s,u_{m},v^{1})ds.
	\end{aligned}
	\end{equation*}
	Using estimates \eqref{3.5} and \eqref{3.9},  we deduce that 
	\begin{equation}\label{3.11}
	\left\|\frac{d}{dt}\left[k\ast(u_{m}-u_{m}(0))\right](t)\right\|_{L^{2}\left(0,T,H^{-1}(\Omega)\right)}\leq K.
	\end{equation}
	Thus estimates \eqref{3.5} and \eqref{3.11} provide a subsequence of $(u_{m})$ again denoted by $(u_{m})$ such that $u_{m}\rightharpoonup u$ in $L^{2}\left(0,T,H^{1}_{0}(\Omega)\right)$ and  $\frac{d}{dt}\left[k\ast (u_{m}-u_{m}(0))\right](t) \rightharpoonup \frac{d}{dt}\left[k\ast (u-u_{0})\right](t)$ in $L^{2}\left(0,T,H^{-1}(\Omega)\right)$. In the  light of estimates \eqref{3.9} and \eqref{3.11}, we apply  compactness lemma \cite{li2018some} to conclude $u_{m} \rightarrow u$ in $L^{2}\left(0,T,H^{1}_{0}(\Omega)\right)$. Now  using the fact that $M(x,t,\|\nabla u_{m}\|^{2})$ and $B(t,s,u_{m},v_{m})$ are continuous and an application of Lebesgue dominated convergence theorem, we pass the limit inside \eqref{3.2} 
	which establishes the existence of weak solutions of the problem $(D_{\alpha}).$ \\
	\noindent \textbf{(Initial Condition)} Now weak solution $u$ satisfies the following equation for all $v$ in $H^{1}_{0}(\Omega)$
\begin{equation}\label{N2}
\begin{aligned}
\left(\frac{d}{dt}\left[k\ast (u-u_{0})\right](t),v\right)+ \left(M\left(x,t,\|\nabla u\|^{2}\right)\nabla  u,\nabla v\right)&=\left(f,v\right)+\int_{0}^{t}B(t,s,u,v)ds.\\
\end{aligned}
\end{equation}
Let $\phi$ in $C^{1}\left([0,T];H^{1}_{0}(\Omega)\right)$ with $\phi(T)=0$, then multiply \eqref{N2} with $\phi$ and integrate by parts, we get 
\begin{equation}\label{N24}
\begin{aligned}
-\int_{0}^{T}\left((k\ast (u-u_{0}))(t),v\right)\phi'(t)dt&+ \int_{0}^{T}\left(M\left(x,t,\|\nabla u\|^{2}\right)\nabla  u,\nabla v\right)\phi(t)dt\\
&=\int_{0}^{T}\left(f,v\right)\phi(t)dt+\int_{0}^{T}\int_{0}^{t}B(t,s,u,v)\phi(t)dsdt\\
&+((k\ast(u-u_{0}))(0),\phi(0)).
\end{aligned}
\end{equation}
Since $C^{1}\left([0,T];H^{1}_{0}(\Omega)\right)$ is dense in $L^{2}\left(0,T,H^{1}_{0}(\Omega)\right)$, thus using \eqref{3.2} and $(k\ast(u_{m}-u_{m}(0)))$ has vanishing trace at $t=0$, we have 
\begin{equation}\label{N248}
\begin{aligned}
-\int_{0}^{T}\left((k\ast (u_{m}-u_{m}(0)))(t),v\right)&\phi'(t)dt+ \int_{0}^{T}\left(M\left(x,t,\|\nabla u_{m}\|^{2}\right)\nabla  u_{m},\nabla v\right)\phi(t)dt\\
&=\int_{0}^{T}\left(f,v\right)\phi(t)dt+\int_{0}^{T}\int_{0}^{t}B(t,s,u_{m},v)\phi(t)dsdt.
\end{aligned}
\end{equation}
Let $m$ tend  to infinity in \eqref{N248} and comparing with  \eqref{N24}, we get 
$((k\ast(u-u_{0}))(0),\phi(0))=0$. Since $\phi(0)$ is arbitrary, so we have $(k\ast(u-u_{0}))(0)=0$ which implies $u=u_{0}$ at $t=0$ \cite{zacher2009weak}.\\

\noindent \textbf{(Uniqueness)}	Suppose that $u_{1},u_{2}$ are solutions of equation $(W_{\alpha})$, then $z=u_{1}-u_{2}$ satisfies the following equation for all $ v \in H_{0}^{1}(\Omega)$ and $a.e. ~ t \in (0,T]$ 
	\begin{equation}\label{3.13}
	\begin{aligned}
	\left(\frac{d}{dt}(k\ast z)(t),v\right)+\left(M\big(x,t,\|\nabla u_{1}\|^2\big)\nabla u_{1}, \nabla v\right)&=\left(M\big(x,t,\|\nabla u_{2}\|^2\big)\nabla u_{2}, \nabla v\right)\\
	&+ \int_{0}^{t}B(t,s,z,v)ds.
	\end{aligned}
	\end{equation}
	Put $v=z(t)$ in above equation and using (H2), (H3) along with Cauchy Schwarz and Young's inequality,  we obtain
	\begin{equation}
	\begin{aligned}
	2\left(\frac{d}{dt}(k\ast z)(t),z\right)+(m_{0}-4L_{M}\hat{K}^{2})\|\nabla z\|^{2} \lesssim  
	 \int_{0}^{t}\|\nabla z(s)\|^{2}ds 
	\end{aligned}
	\end{equation}
	Now following the same lines as in the proof of  estimate \eqref{3.5} and using (H2)
	we conclude  $\|z\|_{L^{2}\left(0,T,H^{1}_{0}(\Omega)\right)}   =\|z\|_{L^{\infty}\left(0,T,L^{2}(\Omega)\right)}=0$. Thus uniqueness follows. 
\end{proof}
\end{subsection}
\section{Error estimates for semi-discrete solution}
In this section, we derive error estimates for semi-discrete solutions of the problem $(S_{\alpha})$ by discretizing the domain in space variable while keeping the time variable continuous.\\
 \noindent The existence, uniqueness, and a priori bounds  on semi-discrete solutions is established on the similar lines as for the weak solutions of $(D_{\alpha})$, therefore omitted here. For the  error estimates of  semi-discrete solutions, we define a new Ritz-Volterra type projection operator $W:[0,T]\rightarrow X_{h}$ by 
\begin{equation}\label{2.5}
\left(M\left(x,t,\|\nabla u\|^{2}\right)\nabla (u-W),\nabla v_{h}\right)=\int_{0}^{t}B(t,s,u-W,v_{h})~ds\quad \text{for all}~ v_{h}~\text{in}~ X_{h}.
\end{equation} 
Ritz-Volterra type projection  operator $W$  is well defined by the positivity of Kirchhoff term $M$.  
 \noindent The estimates on $u-W$, $W$,~ $u-R_{h}u$ are given by   the following theorems.
\begin{thm}\label{thm2.7} \cite{rannacher1982some}
	Under the assumption of uniform triangulation of $\Omega$, there exists a positive constant $C$ independent of $h$ such that the following hold for any function $u$ in $H^{2}(\Omega) \cap H_{0}^{1}(\Omega)$
	\begin{equation*}
	\|u-R_{h}u\|+h\|u-R_{h}u\|_{1} \leq Ch^{2}\|u\|_{2},
	\end{equation*}
	\begin{equation*}
	\|R_{h}u\|_{1}\leq C\|u\|_{1},
	\end{equation*}
	where $\|p\|_{1}=\|p\|+\|\nabla p\|.$
\end{thm}
\begin{thm} \label{thm2.8}\cite{cannon1988non}
	There exists a positive constant $C$ independent of $h$ such that  following estimate hold
	\begin{equation*}
	\begin{aligned}
	\|\rho(t)\|+h\|\rho(t)\|_{1} &\leq Ch^{2}\left(\|u(t)\|_{2}+\int_{0}^{t}\|u(s)\|_{2}ds\right),\\
	\|\rho_{t}(t)\|+h\|\rho_{t}(t)\|_{1} &\leq Ch^{2}\left(\|u(t)\|_{2}+\|u_{t}\|_{2}+\int_{0}^{t}\|u_{t}(s)\|_{2}ds\right),
	\end{aligned}
	\end{equation*}
	where $\rho=u-W$ and $\|p\|_{1}=\|p\|+\|\nabla p\|.$
\end{thm} 
\begin{lem}\cite{kumar2020finite}\label{lem2.8}
	Let $W$ be the Ritz-Volterra projection defined by \eqref{2.5}, then $\|\nabla W\|$ is bounded for every $t$ in $[0,T]$, i.e. 
	\begin{equation*}
	\|\nabla W\|\lesssim\|\nabla u\|.
	\end{equation*}
\end{lem}

\begin{subsection}{Proof of the Theorem \ref{t2.6}}
\begin{proof}
	Put $u_{h}=W-\theta$ in $(S_{\alpha})$, we  have
	\begin{equation*}
	\left(^{C}D^{\alpha}_{t}(W-\theta),v_{h}\right) +\left(M\left(x_{h},t,\|\nabla u_{h}\|^{2}\right)\nabla (W-\theta),\nabla v_{h}\right)=(f_{h},v_{h})+\int_{0}^{t}B(t,s,W-\theta,v_{h})ds.
	\end{equation*}
	Since $f$ in $L^{2}(\Omega)$ thus $(f,v_{h})=(f_{h},v_{h})$ for all $v_{h}$ in $X_{h}$, then by using (H2), $(W_{\alpha})$, and the definition \eqref{2.5} of Ritz-Volterra projection $W$ we get 
	\begin{equation}\label{4.4}
	\begin{aligned}
	\left(^{C}D^{\alpha}_{t}\theta,v_{h}\right) +m_{0}(\nabla \theta,\nabla v_{h})&\leq L_{M}\left(\|\nabla u\|+\|\nabla u_{h}\|\right)\left(\|\nabla u-\nabla u_{h}\|\right)(\nabla W,\nabla v_{h})\\
	&-\left(^{C}D^{\alpha}_{t}\rho,v_{h}\right)+\int_{0}^{t}B(t,s,\theta,v_{h})ds.	
	\end{aligned}
	\end{equation}
	Set $v_{h}=\theta$ in \eqref{4.4} and using (H3), estimate \eqref{3.9} together with Cauchy-Schwarz and Young's inequality, we obtain
	\begin{equation}
	\begin{aligned}
\left(^{C}D^{\alpha}_{t}\theta,v_{h}\right)+(m_{0}-4L_{M}\hat{K}^{2})\|\nabla \theta\|^{2} &\lesssim \|\nabla \rho\|^{2}+\|~^{C}D^{\alpha}_{t}\rho\|^{2}\\
&+\|\theta\|^{2}+\int_{0}^{t}\|\nabla \theta(s)\|^{2}ds. 
\end{aligned}
	\end{equation}
	Use (H2) and  apply the same arguments as we prove estimate \eqref{3.5}, to obtain 
	\begin{equation*}
\|\theta\|^{2}_{L^{\infty}\left(0,T,L^{2}(\Omega)\right)}+\left(l\ast\|\nabla \theta\|^{2}\right)(t) \lesssim \left[l\ast \left(\|\nabla \rho\|^{2}+\|~^{C}D^{\alpha}_{t}\rho\|^{2}\right)\right](t)+\|\theta(0)\|^{2}_{1}.
	\end{equation*}
	Theorems \ref{thm2.7} and \ref{thm2.8} with triangle inequality concludes 
	\begin{equation*}
	\|u-u_{h}\|^{2}_{L^{\infty}\left(0,T,L^{2}(\Omega)\right)}+\left(l\ast\|\nabla (u-u_{h})\|^{2}\right)(t) \lesssim h^{2}.
	\end{equation*}
\end{proof}
\end{subsection}
\begin{section}{Backward Euler-Galerkin FEM  based on L1 type scheme}
	In this section, we prove the well-posedness and derive convergence rate for the proposed  fully discrete formulation $(E_{\alpha})$.\\
	\noindent For $n\geq 2$, the linear system  is positive definite due to the  non-degeneracy of diffusion coefficient  $M$  which gives a unique solution of $(E_\alpha)$ and the existence of $u_{h}^{1}$ is established by the following variant of Br\"{o}uwer fixed point theorem.
	\begin{thm}\cite{thomee1984galerkin}\label{T1.3}
		Let $H$ be   finite dimensional Hilbert space. Let $G:H\rightarrow H$ be a continuous map such that $\left(G(w),w\right)>0$  for all $w$ in $H$ with $\|w\|=r,~r>0$ then there exists a $\tilde{w}$ in $H$ such that $G(\tilde{w})=0$ and $\|\tilde{w}\|\leq r.$
	\end{thm}
	\begin{thm}\label{T1.4}
		Suppose that (H1), (H2), and (H3) holds, then there exists a unique solution $u_{h}^{1}$ to the problem  $(E_{\alpha})$  which satisfies the following a priori bound 
		\begin{equation}\label{S1.21}
		\|u_{h}^{1}\|^{2}+k^{\alpha}\|\nabla u_{h}^{1}\|^{2}\lesssim \left( \|f_{h}^{1}\|^{2}+\|\nabla u_{h}^{0}\|^{2}\right)\approx \tilde{K}_{1},
		\end{equation}
		provided $k< k_{0}$ for some positive constant $k_{0}$.
			\end{thm}
		\begin{proof}
			\textbf{(Existence)} ~Let $\alpha_{0}=k^{\alpha}\Gamma(2-\alpha)$ and put the definition of  $\mathbb{D}^{\alpha}_{t}u_{h}^{1}$ in $(E_{\alpha})$, we get 
			\begin{equation*}
			\begin{aligned}
			\left(u_{h}^{1}-u_{h}^{0},v_{h}\right)&+\alpha_{0}\left(M\left(x_{h},t_{1},\|\nabla u_{h}^{1}\|^{2}\right)\nabla u_{h}^{1},\nabla v_{h}\right)=\alpha_{0}\left(f_{h}^{1},v_{h}\right)
			+\alpha_{0}kB\left(t_{1},t_{0},u_{h}^{0},v_{h}\right).
			\end{aligned}
			\end{equation*}
			Define a map $G:X_{h}\rightarrow X_{h}$ by 
			\begin{equation}\label{S1.19}
			\begin{aligned}
			\left(G\left(u_{h}^{1}\right),v_{h}\right)&=\left(u_{h}^{1},v_{h}\right)-\left(u_{h}^{0},v_{h}\right)+\alpha_{0}\left(M\left(x_{h},t_{1},\|\nabla u_{h}^{1}\|^{2}\right)\nabla u_{h}^{1},\nabla v_{h}\right)\\
			&-\alpha_{0}\left(f_{h}^{1},v_{h}\right)-\alpha_{0}kB\left(t_{1},t_{0},u_{h}^{0},v_{h}\right).
			\end{aligned}
			\end{equation}
			Then using (H2), (H3), and Cauchy-Schwarz inequality, we have 
			\begin{equation*}
			\begin{aligned}
			\left(G\left(u_{h}^{1}\right),u_{h}^{1}\right)\geq \left(\|u_{h}^{1}\|-\|u_{h}^{0}\|-\alpha_{0}\|f_{h}^{1}\|-\alpha_{0}k\|u_{h}^{0}\|\right)\|u_{h}^{1}\|.
			\end{aligned}
			\end{equation*}
			Thus for $\|u_{h}^{1}\|> \|u_{h}^{0}\|+\alpha_{0}\|f_{h}^{1}\|+k\alpha_{0}\|u_{h}^{0}\|$, we have $\left(G\left(u_{h}^{1}\right),u_{h}^{1}\right)>0$ and the map $G$ defined by \eqref{S1.19} is continuous by continuity of $M$. Hence existence follows by  Theorem \ref{T1.3}.\\
		\noindent \textbf{(A priori bound)}
			Put $v_{h}=u_{h}^{1}$ for $n=1$ in $(E_{\alpha})$ and using (H2), (H3) with Cauchy-Schwarz and Young's  inequality, we get 
			\begin{equation*}
			\begin{aligned}
			(1-k^{\alpha}\Gamma(2-\alpha))\|u_{h}^{1}\|^{2}+k^{\alpha}\|\nabla u_{h}^{1}\|^{2}&\lesssim k^{\alpha}
			\left(\|f_{h}^{1}\|^{2}+k^{2}\|\nabla u_{h}^{0}\|^{2}\right)+\|u_{h}^{0}\|^{2},
			\end{aligned}
			\end{equation*}
			then for sufficiently small $k$ such that $k^{\alpha}< \frac{1}{\Gamma(2-\alpha)}$, we conclude \eqref{S1.21}.

	\noindent 	\textbf{(Uniqueness)} Suppose that $X_{h}^{1}$ and $Y_{h}^{1}$ are the solutions of $(E_{\alpha})$ for $n=1$, then $Z_{h}^{1}=X_{h}^{1}-Y_{h}^{1}$ satisfies the following equation for all $v_{h}$ in $X_{h}$
			\begin{equation}\label{S1.22}
			\begin{aligned}
			\left(\mathbb{D}^{\alpha}_{t}Z_{h}^{1},v_{h}\right)&+\left(M\left(x_{h},t_{1},\|\nabla X_{h}^{1}\|^{2}\right)\nabla Z_{h}^{1},\nabla v_{h}\right)\\
			&=\left(\left[M\left(x_{h},t_{1},\|\nabla Y_{h}^{1}\|^{2}\right)-M\left(x_{h},t_{1},\|\nabla X_{h}^{1}\|^{2}\right)\right]\nabla Y_{h}^{1},\nabla v_{h}\right).
			\end{aligned}
			\end{equation}
			Put $v_{h}=Z_{h}^{1}$ in equation \eqref{S1.22} and using (H2) we get
			\begin{equation*}
			\frac{1}{2}\mathbb{D}^{\alpha}_{t}\|Z_{h}^{1}\|^{2}+m_{0}\|\nabla Z_{h}^{1}\|^{2}\leq L_{M}\|\nabla Z_{h}^{1}\|\left(\|\nabla X_{h}^{1}\|+\|\nabla Y_{h}^{1}\|\right)\left(\nabla Y_{h}^{1},\nabla Z_{h}^{1}\right)
			\end{equation*}
			Now Cauchy-Schwarz inequality and \eqref{S1.21} yields
			\begin{equation*}
			\mathbb{D}^{\alpha}_{t}\|Z_{h}^{1}\|^{2}+\left(2m_{0}-4L_{M}\tilde{K}_{1}^{2}\right)\|\nabla Z_{h}^{1}\|^{2}\leq 0.
			\end{equation*}
			Then (H2) concludes the proof.
		\end{proof}

	\begin{thm}\label{5.3}
		Under the assumption of (H1), (H2), and (H3) the solution  $u_{h}^{n}~ (n\geq 2)$ of $(E_{\alpha})$ satisfy the following a priori bound
		\begin{equation}\label{S1.24}
		\begin{aligned}
		\max_{2\leq m \leq N}\|u_{h}^{m}\|^{2}+k^{\alpha}\sum_{n=2}^{N}p_{N-n}\|\nabla u_{h}^{n}\|^{2}
		\lesssim \left(k^{\alpha}\sum_{n=2}^{N}p_{N-n}\|f_{h}^{n}\|^{2}+\|\nabla u_{h}^{0}\|^{2}\right)\approx \tilde{K_{2}}.
		\end{aligned}
		\end{equation}
			\end{thm}
		\begin{proof}
			Put $v_{h}=u_{h}^{n}$ for $n\geq 2$ in $(E_{\alpha})$ and using (H2), (H3) along with Cauchy-Schwarz and  Young's Inequality we get 
			\begin{equation}\label{S1.23}
			\frac{k^{-\alpha}}{\Gamma(2-\alpha)}\sum_{j=1}^{n}a_{n-j}\left(\|u_{h}^{j}\|^{2}-\|u_{h}^{j-1}\|^{2}\right)+\|\nabla u_{h}^{n}\|^{2}\lesssim  \left(\|f_{h}^{n}\|^{2}+\|u_{h}^{n}\|^{2}+\sum_{j=1}^{n-1}w_{nj}\|\nabla u_{h}^{j}\|^{2}\right).
			\end{equation}
			Now multiply the equation \eqref{S1.23} by discrete convolution $P_{m-n}$ defined in Lemma \ref{L1.2} and take summation from $n=1$ to $m$ to obtain 
			\begin{equation*}
			\begin{aligned}
			&\sum_{n=1}^{m}p_{m-n}\sum_{j=1}^{n}a_{n-j}\left(\|u_{h}^{j}\|^{2}-\|u_{h}^{j-1}\|^{2}\right)+k^{\alpha}\Gamma(2-\alpha)\sum_{n=1}^{m}p_{m-n}\|\nabla u_{h}^{n}\|^{2}\\
			&\lesssim k^{\alpha}\Gamma(2-\alpha)\left(\sum_{n=1}^{m}p_{m-n}\|f_{h}^{n}\|^{2}+\sum_{n=1}^{m}p_{m-n}\|u_{h}^{n}\|^{2}+\sum_{n=1}^{m}p_{m-n}\sum_{j=1}^{n-1}w_{nj}\|\nabla u_{h}^{j}\|^{2}\right).
			\end{aligned}
			\end{equation*}
			 Interchanging the summation and \eqref{2.81} yields
			\begin{equation*}
			\begin{aligned}
		(1-\alpha_{0})	\|u_{h}^{m}\|^{2}+k^{\alpha}\sum_{n=1}^{m}p_{m-n}\|\nabla u_{h}^{n}\|^{2}
			&\lesssim \sum_{n=1}^{m-1}\left(k^{\alpha}p_{m-n}\|u_{h}^{n}\|^{2}+k_{1}k^{\alpha}\sum _{j=1}^{n}p_{n-j}\|\nabla u_{h}^{j}\|^{2}\right)\\
			&+k^{\alpha}\sum_{n=1}^{m}p_{m-n}\|f_{h}^{n}\|^{2}+\|u_{h}^{0}\|^{2}.
			\end{aligned}
			\end{equation*}
			Now for sufficiently small $k^{\alpha} <\frac{1}{\Gamma{(2-\alpha)}}$ we have 
				\begin{equation*}
			\begin{aligned}
			\|u_{h}^{m}\|^{2}+k^{\alpha}\sum_{n=1}^{m}p_{m-n}\|\nabla u_{h}^{n}\|^{2}
			&\lesssim \sum_{n=1}^{m-1}\left(k^{\alpha}p_{m-n}+k_{1}\right)\left(\|u_{h}^{n}\|^{2}+k^{\alpha}\sum _{j=1}^{n}p_{n-j}\|\nabla u_{h}^{j}\|^{2}\right)\\
			&+k^{\alpha}\sum_{n=1}^{m}p_{m-n}\|f_{h}^{n}\|^{2}+\|u_{h}^{0}\|^{2}.
			\end{aligned}
			\end{equation*}
			Now  Discrete Gr\"{o}nwall's inequality and \eqref{2.91} concludes the proof.
		\end{proof}
		
	\begin{subsection}{Proof of the Theorem \ref{J1}}
	\begin{proof} 	(Error Estimate for $e^{1}=u(t_{1})-u_{h}^{1}$)
		 Substitute $u_{h}^{1}=W^{1}-\theta^{1}$ for $n=1$ in $(E_{\alpha})$ and using weak formulation $(W_{\alpha})$  along with Ritz-Volterra projection $W$ at $t_{1}$ we get 
		\begin{equation}\label{S1.271}
		\begin{aligned}
		\left(\mathbb{D}^{\alpha}_{t}\theta^{1},v_{h}\right)&+\left(M\left(x_{h},t_{1},\|\nabla u_{h}^{1}\|^{2}\right)\nabla \theta^{1},\nabla v_{h}\right)-kB\left(t_{1},t_{0},\theta^{0},v_{h}\right)\\
		&=\left(\mathbb{D}^{\alpha}_{t}W^{1}-~^{C}D^{\alpha}_{t_{1}}u,v_{h}\right)-kB\left(t_{1},t_{0},W^{0},v_{h}\right)+\int_{0}^{t_{1}}B(t_{1},s,W,v_{h})ds\\
		&+\left(\left[M\left(x_{h},t_{1},\|\nabla u_{h}^{1}\|^{2}\right)-M\left(x_{h},t_{1},\|\nabla u^{1}\|^{2}\right)\right]\nabla W^{1},\nabla v_{h}\right).
		\end{aligned}
		\end{equation}
		Set $v_{h}=\theta^{1}$ in \eqref{S1.271} and using (H2), (H3) with Cauchy-Schwarz inequality and Young's inequality, we get
	 \begin{equation}\label{S1}
	 \begin{aligned}
	 \left(\mathbb{D}^{\alpha}_{t}\theta^{1},\theta^{1}\right)+\|\nabla \theta^{1}\|^{2}&\lesssim \|\mathbb{Q}^{1}\|^{2}+\|~^{C}{D}^{\alpha}_{t_{1}}\rho\|^{2}+\|q^{1}(W)\|^{2}+
	 \|\nabla \rho^{1}\|^{2}.
	 \end{aligned}
	 \end{equation}
		Where $\mathbb{Q}^{1}=\mathbb{D}^{\alpha}_{t}W^{1}-~^{C}D^{\alpha}_{t_{1}}W$ is the truncation error defined in \eqref{2.9} which satisfies $\|\mathbb{Q}^{1}\|\lesssim k^{2-\alpha}$ \cite{lin2007finite} and  $q^{1}(W)$ is  quadrature error for rectangle rule on a single interval of length $k$ which  satisfies $\|q^{1}(W)\|\lesssim  k^{2}$. Thus using  Theorem \ref{thm2.8}, we conclude
		\begin{equation}
		\|\theta^{1}\|^{2}+k^{\alpha}\|\nabla \theta^{1}\|^{2}\lesssim \left(k^{2}+h\right)^{2}.
		\end{equation}
 \noindent 	(Error Estimate for $e^{n}=u(t_{n})-u_{h}^{n},~~n\geq 2$) Put $u_{h}^{n}=W^{n}-\theta^{n}$ in $(E_{\alpha})$ and follow the same lines as for $n=1$, we arrive at 
		\begin{equation}\label{S1.33}
		\begin{aligned}
		\sum_{j=1}^{n}a_{n-j}\left[\|\theta^{j}\|^{2}-\|\theta^{j-1}\|^{2}\right]+k^{\alpha}\|\nabla \theta^{n}\|^{2}&\lesssim k^{\alpha}\left(\|\mathbb{Q}^{n}\|^{2}+\|~^{C}D_{t_{n}}^{\alpha}\rho\|^{2}+\sum_{j=1}^{n-1}w_{nj}\|\nabla \theta^{j}\|^{2}\right)\\
		&+k^{\alpha}\left(\|\nabla \bar{u}_{h}^{n-1}-\nabla u^{n}\|^{2}	+\|q^{n}(W)\|^{2}\right),
		\end{aligned}
		\end{equation}
		\noindent 	where $q^{n}(W)$ be the quadrature error operator on $H^{1}_{0}(\Omega)$, defined by 
		\begin{equation}\label{S1.32}
		q^{n}(W)(v)=\int_{0}^{t_{n}}B(t_{n},s,W,v)-\sum_{j=1}^{n-1}w_{nj}B\left(t_{n},t_{j},W^{j},v\right)ds,
		\end{equation}
		which satisfies $\|q^{n}(W)\|\lesssim k^{2}$ \cite{pani1992numerical} and  $\mathbb{Q}^{n}=\mathbb{D}^{\alpha}_{t}W^{n}-~^{C}D^{\alpha}_{t_{n}}W$ that satisfies $\|\mathbb{Q}^{n}\|\lesssim k^{2-\alpha}$\cite{lin2007finite}. Thus using Theorem \ref{thm2.8}, we obtain 
		\begin{equation}\label{S1.37}
		\begin{aligned}
		\sum_{j=1}^{n}a_{n-j}\left[\|\theta^{j}\|^{2}-\|\theta^{j-1}\|^{2}\right]+k^{\alpha}\|\nabla \theta^{n}\|^{2}&\lesssim k^{\alpha}\left(k^{2-\alpha}+h^{2}+h+k^{2}+k^{2}\right)^{2}+k^{\alpha}\|\theta^{n}\|^{2}\\
		&+k^{\alpha}\left(\sum_{j=1}^{n-1}k_{1}\|\nabla \theta^{j}\|^{2}+\|\nabla \theta^{n-1} \|^{2}+\|\nabla \theta^{n-2}\|^{2}\right).
		\end{aligned}
		\end{equation}
		Now follow the same arguments as we prove estimate \eqref{S1.24} to conclude the result.
	\end{proof}
\end{subsection}
\end{section}
\section{Fractional Crank-Nicolson-Galerkin FEM based on L2-1$_{\sigma}$ scheme}
In this section, we show that our proposed scheme $(F_{\alpha})$ achieve the optimal rate of convergence in time. The existence and uniqueness of solutions for fully discrete formulation $(F_{\alpha})$ is proved along the similar lines as in the proof of Theorem \ref{T1.4}, therefore omitted here.

\begin{thm}\label{6.1} (\textbf{A priori bound}) Under the assumptions of (H1), (H2), and (H3) the solutions $u_{h}^{n}$ of $(F_{\alpha})$ satisfy the following a priori bounds 
		\begin{equation}\label{S2.24}
			\begin{aligned}
			\max_{1\leq m \leq N}\|u_{h}^{m}\|^{2}+k^{\alpha}\sum_{n=1}^{N}\tilde{p}_{N-n}^{(N)}\|\nabla u_{h}^{n}\|^{2}
			\lesssim \left(k^{\alpha}\sum_{n=1}^{N}\tilde{p}_{N-n}^{(N)}\|f_{h}^{n-\frac{\alpha}{2}}\|^{2}+\|\nabla u_{h}^{0}\|^{2}\right)
			\end{aligned}
			\end{equation}
	\end{thm}
	
	\begin{proof} 
		The a priori bounds for $u_{h}^{1}$  is similar as in Theorem \ref{T1.4}. To obtain a priori bounds for $u_{h}^{n}~( n\geq 2)$,
		put $v_{h}=u_{h}^{n}$ for $n\geq 2$ in $(F_{\alpha})$ and using (H2), (H3) along with the following identity 
		\begin{equation}
		\|\nabla u_{h}^{n}+\nabla u_{h}^{n-1}\|^{2}=\|\nabla u_{h}^{n}\|^{2}+\|\nabla u_{h}^{n-1}\|^{2}+2\left(\nabla u_{h}^{n},\nabla u_{h}^{n-1}\right),
		\end{equation}
		 we arrive at 
		\begin{equation}\label{S2.23}
		\begin{aligned}
		\frac{k^{-\alpha}}{\Gamma(2-\alpha)}\sum_{j=1}^{n}\tilde{c}_{n-j}^{(n)}\left(\|u_{h}^{j}\|^{2}-\|u_{h}^{j-1}\|^{2}\right)+\|\nabla u_{h}^{n}\|^{2}&\lesssim  \|f_{h}^{n-\frac{\alpha}{2}}\|^{2}+\sum_{j=1}^{n-1}\tilde{w}_{nj}\|\nabla u_{h}^{j}\|^{2}\\
		&+\|\nabla u_{h}^{n-1}\|^{2}
		\end{aligned}
		\end{equation}
		Now follow the same arguments as we prove Theorem \ref{5.3} to obtain \eqref{S2.24}.
	\end{proof}

\begin{subsection}{Proof of the Theorem \ref{J2}}
\begin{proof}
	Substitute $u_{h}^{n}=W^{n}-\theta^{n}$ in $(F_{\alpha})$ and the evaluation of weak formulation $(W_{\alpha})$ and Ritz-Volterra projection $W$ at $t_{n-\frac{\alpha}{2}}$ and then put $v_{h}=\theta^{n}$, we get 
	\begin{equation}\label{2.21}
	\begin{aligned}
	\tilde{\mathbb{D}}^{\alpha}_{t_{n-\frac{\alpha}{2}}}\|\theta^{n}\|^{2}+\|\nabla \theta^{n}\|^{2}&\lesssim \|\tilde{\mathbb{D}}^{\alpha}_{t_{n-\frac{\alpha}{2}}}W^{n}-~^{C}D^{\alpha}_{t_{n-\frac{\alpha}{2}}}u\|^{2}+\sum_{j=1}^{n-1}\tilde{w}_{nj}\|\nabla \theta^{j}\|^{2}\\
	&+\|\nabla \bar{u}_{h}^{n-1,\alpha}-\nabla u^{n-\frac{\alpha}{2}}\|^{2}+\|\nabla \hat{W}^{n,\alpha}-\nabla W^{n-\frac{\alpha}{2}}\|^{2}\\
	&	+\|\tilde{q}^{n-\frac{\alpha}{2}}(W)\|^{2}+\|\nabla \theta^{n-1}\|^{2}
	\end{aligned}
	\end{equation}
	where  $\tilde{q}^{n-\frac{\alpha}{2}}(W)$ be the quadrature error operator on $H^{1}_{0}(\Omega)$, defined by 
	\begin{equation}\label{S2.32}
	\tilde{q}^{n-\frac{\alpha}{2}}(W)(v)=\int_{0}^{t_{n-\frac{\alpha}{2}}}B(t_{n-\frac{\alpha}{2}},s,W,v)-\sum_{j=1}^{n-1}\tilde{w}_{nj}B\left(t_{n-\frac{\alpha}{2}},t_{j},W^{j},v\right)ds,
	\end{equation}
	that satisfies $\|	\tilde{q}^{n-\frac{\alpha}{2}}(W)\|\lesssim k^{2}$ \cite{pani1992numerical}
	and Let $\tilde{\mathbb{Q}}^{n-\frac{\alpha}{2}}=\tilde{\mathbb{D}}^{\alpha}_{t_{n-\frac{\alpha}{2}}}W^{n}-~^{C}D^{\alpha}_{t_{n-\frac{\alpha}{2}}}W$ is the truncation error which satisfies $\|\tilde{\mathbb{Q}}^{n-\frac{\alpha}{2}}\|\lesssim k^{3-\alpha}$ \cite{alikhanov2015new}. Thus we get 
	\begin{equation}\label{S2.37}
	\begin{aligned}
	\sum_{j=1}^{n}\tilde{c}_{n-j}^{(n)}\left[\|\theta^{j}\|^{2}-\|\theta^{j-1}\|^{2}\right]+k^{\alpha}\|\nabla \theta^{n}\|^{2}&\lesssim k^{\alpha}\left(h^{2}+k^{3-\alpha}+h+k^{2}+k^{2}\right)^{2}\\
	&+k^{\alpha}\left(\sum_{j=1}^{n-1}k_{1}\|\nabla \theta^{j}\|^{2}+\|\nabla \theta^{n-1} \|^{2}+\|\nabla \theta^{n-2}\|^{2}\right).
	\end{aligned}
	\end{equation}
\noindent Now follow the same arguments as we prove Theorem \ref{J1} and Theorem \ref{6.1} to obtain 
	\begin{equation*}
	\|\theta^{m}\|^{2}+k^{\alpha}\sum_{n=1}^{m}\tilde{p}_{m-n}^{(m)}\|\nabla \theta^{n}\|^{2}
	\lesssim \left(h+k^{\min\{3-\alpha,2\}}\right)^{2}
	\end{equation*}
	Use the similar approach for the case $n=1$ and then by  triangle inequality and $\min\{3-\alpha,2\}=2$,  we obtain the required result.
\end{proof}
\end{subsection}
\section{Numerical experiments}
 In this section, we implement the theoretical  results obtained from fully discrete formulation for the model $(D_\alpha)$. For space discretization linear hat basis functions  say $\{\psi_{1}, \psi_{2}, \dots, \psi_{J}\}$ for $J$ dimensional subspace $X_{h}$ of $H_{0}^{1}(\Omega)$ is used, then numerical solution $u_{h}^{n}$ for $(D_{\alpha})$  at any time $t_{n}$ in $[0,T]$ is written as  
\begin{equation}\label{7.1}
u_{h}^{n}=\sum_{i=1}^{J}\alpha_{i}^{n}\psi_{i},
\end{equation}
where $\alpha^{n}=\left(\alpha_{1}^{n}, \alpha_{2}^{n}, \alpha_{3}^{n}, \dots, \alpha_{J}^{n} \right)$ is to be determined.
Let us denote the error estimates that we have proved in Theorems \ref{J1} and  \ref{J2} by 
\begin{equation*}
\begin{aligned}
\text{\bf Error-1}&=\max_{1\leq n \leq N}\|u(t_{n})-u_{h}^{n}\|+\left(k^{\alpha} \sum_{n=1}^{N}p_{N-n}\|\nabla u(t_{n})-\nabla u_{h}^{n}\|^{2}\right)^{1/2},
\end{aligned}
\end{equation*} 
and 
\begin{equation*}
\begin{aligned}
\text{\bf Error-2}&=\max_{1\leq n \leq N}\|u(t_{n})-u_{h}^{n}\|+\left(k^{\alpha} \sum_{n=1}^{N}\tilde{p}_{N-n}^{(N)}\|\nabla u(t_{n})-\nabla u_{h}^{n}\|^{2}\right)^{1/2},
\end{aligned}
\end{equation*} 
respectively.
\begin{exam} We consider the problem $(D_{\alpha})$ on $\Omega \times [0,T]$  where $\Omega = [0,1]\times[0,1]$ and $T=1$ with the following data
\begin{enumerate}
	\item $M\left(x,y,t,\|\nabla u\|^{2}\right)=a(x,y)+b(x,y)\|\nabla u\|^{2}$ with $a(x,y)=x^{2}+y^{2}+1$ and $b(x,y)=xy$. This type of diffusion coefficient $M$ has been studied by medeiros et al. for numerical purpose in \cite{medeiros2012perturbation}.
	\item $b_{2}(t,s,x,y)=e^{t-s}$
	$\begin{bmatrix}
		-1 & 0\\
		0 &-1
		\end{bmatrix}; ~ b_{1}(t,s,x,y)=b_{0}(t,s,x,y)=0.$
		\item	Source term  
	 $f(x,t)=f_{1}(x,t)+f_{2}(x,t)+f_{3}(x,t)$ with
	 \begin{equation*}
	     \begin{aligned}
	       f_{1}(x,t)&=\frac{2}{\Gamma(3-\alpha)}t^{2-\alpha}(x-x^{2})(y-y^{2}),\\   
	       f_{2}(x,t)&=2(x+y-x^{2}-y^{2})\left(t^{2}x^{2}+t^{2}y^{2}+\frac{xyt^{6}}{45}-2t-2+2e^{t}\right),\\
	       f_{3}(x,t)&=t^{2}(2x-1)(y-y^{2})\left(2x+\frac{yt^{4}}{45}\right)+t^{2}(2y-1)(x-x^{2})\left(2y+\frac{xt^{4}}{45}\right).
	     \end{aligned}
	 \end{equation*}
	 Then the exact solution of $(D_{\alpha})$ is  
		$u=t^{2}\left(x-x^{2}\right)\left(y-y^{2}\right).$	
\end{enumerate}
\end{exam}
 \noindent This example can be used to model the diffusion of a substance in a domain $\Omega$, see \cite{barbeiro2011h1}. \\

\noindent  We obtain error and convergence rate in space direction as well as in time direction for different values of mesh parameters $h$ and $k$ and different values of $\alpha$. The convergence rate is calculated through the following $\log$ vs. $\log$ formula 
\begin{equation*}
   \text{ Convergence rate} =\begin{cases}\frac{\log(E(\tau,h_{1})/E(\tau,h_{2}))}{\log(h_{1}/h_{2})};& \text{In space direction}\\
   \frac{\log(E(\tau_{1},h)/E(\tau_{2},h))}{\log(\tau_{1}/\tau_{2})};& \text{In time direction }
\end{cases}
\end{equation*}
where $E(\tau,h)$ is the error in respective norms  at different mesh points.

\noindent \textbf{Backward Euler-Galerkin FEM scheme:} This numerical scheme $(E_{\alpha})$  provides  the convergence rate as $O(h+k^{2-\alpha}),~$see Theorem \ref{J1}. To observe this order of convergence numerically, we run the MATLAB code at different iterations by  setting $h \simeq k^{2-\alpha}$. Here $h$ is taken as the area of the triangle in the triangulation of domain $\Omega=[0,1]\times [0,1]$. For next iteration, we join the midpoint of each edge and make another triangulation as presented in Figures 1-3. In this way we collect the numerical results upto five iterations to support our theoretical estimates established in Theorem \ref{J1}.

\begin{figure}[H]
	\centering
\includegraphics[scale=0.2]{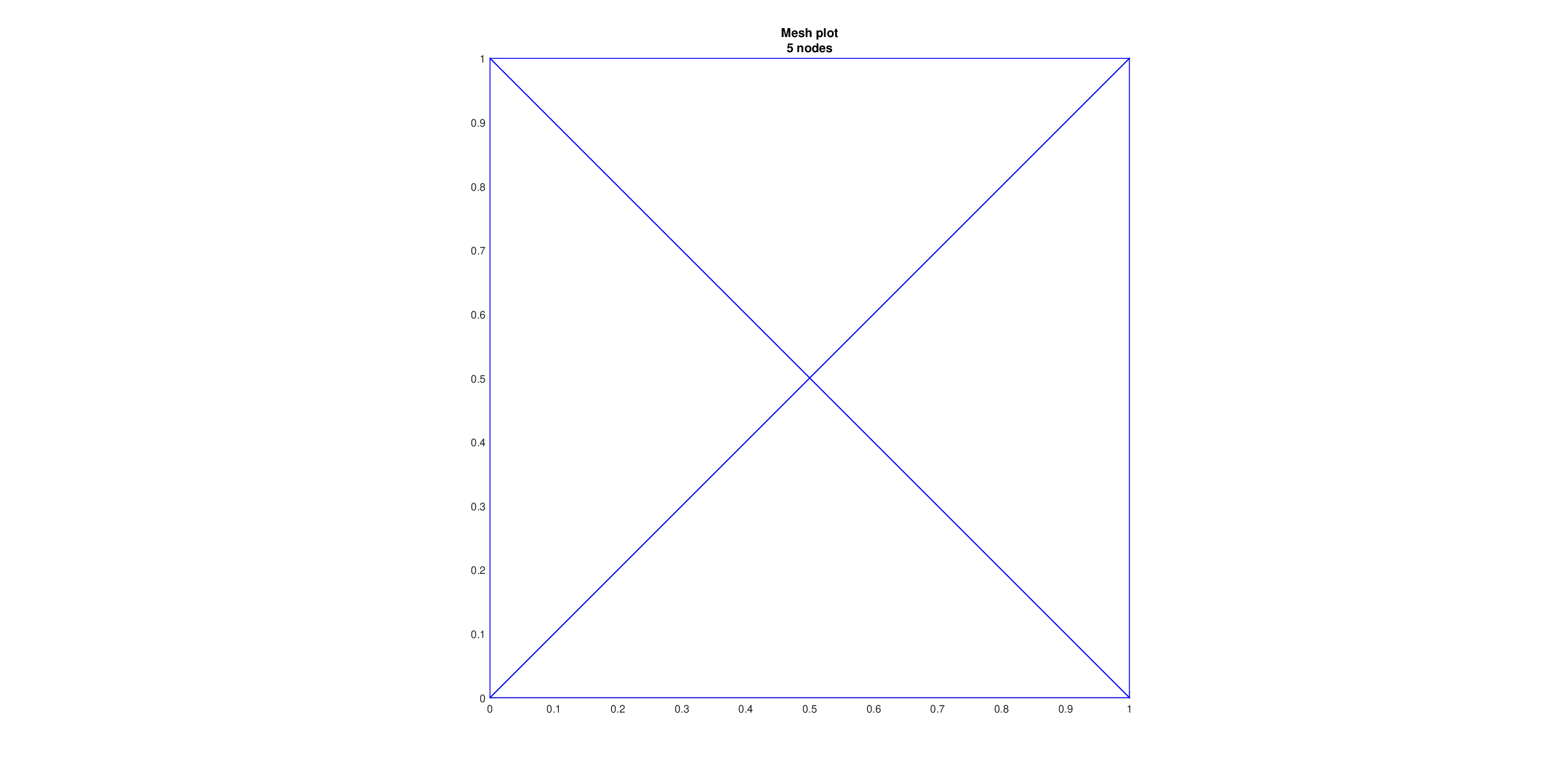}
\caption{Iteration no 1}
\end{figure}
\begin{figure}[H]
	\centering
\includegraphics[scale=0.2]{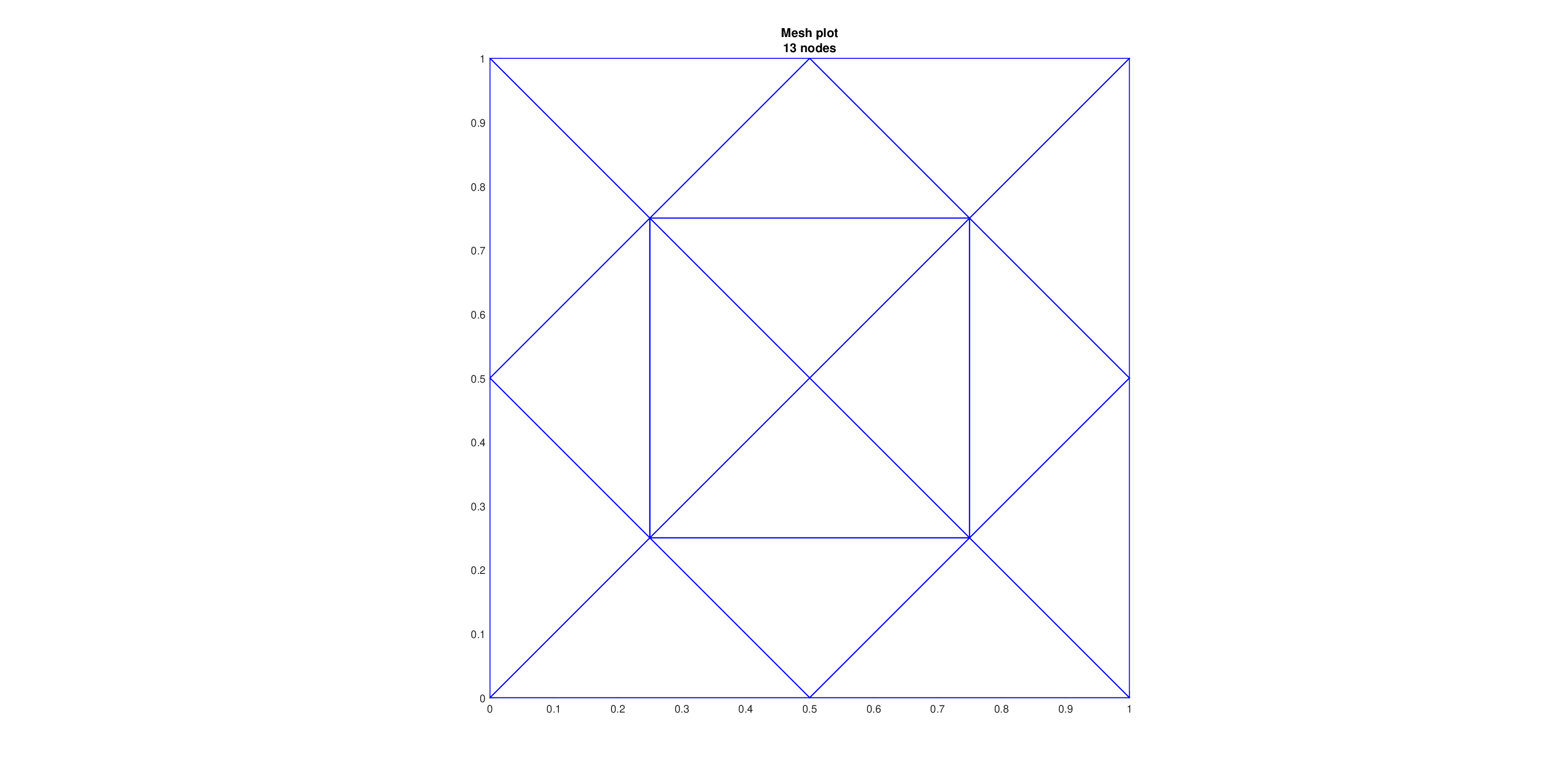}
\caption{Iteration no 2}
\end{figure}
\begin{figure}[H]
	\centering
\includegraphics[scale=0.2]{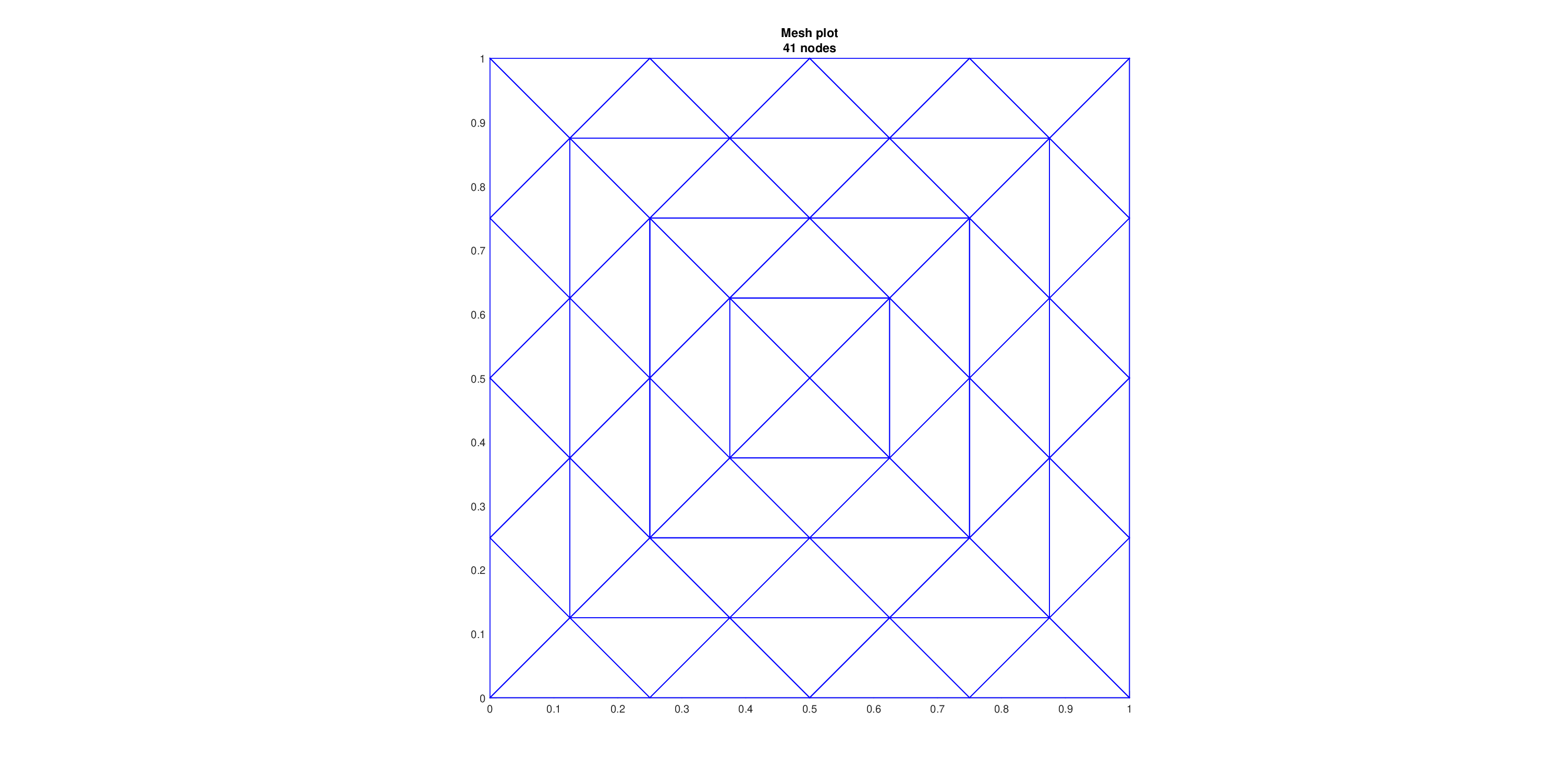}
\caption{Iteration no 3}
\end{figure}

\noindent From  Tables \ref{table3}, \ref{table31}, and \ref{table361}, we conclude that the convergence rate is linear in space and $(2-\alpha)$ in the time direction, as we have proved in Theorem \ref{J1}. We also observe that as $\alpha$ tending  to 1 then convergence rate is  approaching  1 in the  time direction which is same as we have  established in \cite{kumar2020finite} for classical diffusion case. 
  \begin{table}[htbp]
	\centering
	\caption{\emph{Error and Convergence Rates in space-time direction for $\alpha=0.25$ }}
	\begin{tabular}{|c|c|c|c|}
		\hline
		\textbf{Iteration No.} & \textbf{Error-1} & \textbf{Rate in Space} & \textbf{Rate in Time} \\
		\hline
		1 & 3.04e-02 & - & -  \\
		\hline
		2 & 1.50e-02& 1.0175 & 1.7807  \\
		\hline
		3 &  7.60e-03 & 0.9824 & 1.7192 \\
		\hline
		4 & 3.82e-03 & 0.9909 & 1.7342\\
		\hline
		5 & 1.89e-03 & 1.0166 & 1.7790\\
		\hline
	\end{tabular}%
	\label{table3}%
\end{table}
\begin{table}[htbp]
	\centering
	\caption{\emph{Error and Convergence Rates in space-time direction for $\alpha=0.5$ }}
	\begin{tabular}{|c|c|c|c|}
		\hline
		\textbf{Iteration No.} & \textbf{Error-1} & \textbf{Rate in Space} & \textbf{Rate in Time} \\
		\hline
		1 & 2.63e-02 & - & -  \\
		\hline
		2 & 1.35e-02& 0.9568 & 1.4352  \\
		\hline
		3 &  6.53e-03 & 1.0521 & 1.5781 \\
		\hline
		4 & 3.20e-03 & 1.0285 & 1.5428\\
		\hline
		5 & 1.58e-03 & 1.0141 & 1.5212\\
		\hline
	\end{tabular}%
	\label{table31}%
\end{table}
\begin{table}[htbp]
	\centering
	\caption{\emph{Error and Convergence Rates in space-time direction for $\alpha=0.75$ }}
	\begin{tabular}{|c|c|c|c|}
		\hline
		\textbf{Iteration No.} & \textbf{Error-1} & \textbf{Rate in Space} & \textbf{Rate in Time} \\
		\hline
		1 & 2.23e-02 & - & -  \\
		\hline
		2 & 1.09e-02& 1.0340 & 1.2925  \\
		\hline
		3 &  5.33e-03 & 1.0345 & 1.2931 \\
		\hline
		4 & 2.65e-03 & 1.0087 & 1.2960\\
		\hline
		5 & 1.29e-03 & 1.0331 & 1.2914\\
		\hline
	\end{tabular}%
	\label{table361}%
\end{table}

\noindent \textbf{Fractional Crank-Nicolson-Galerkin  FEM scheme:} This numerical scheme converges to the solution with accuracy $O(h+k^{2})$. Here we set $h\simeq k^{2}$ to conclude the convergence rates in space-time direction. Here iteration numbers have the same meaning as in the previous case $(E_{\alpha})$. From the Tables \ref{table32}, \ref{table322}, and \ref{table323}, we deduce  that the convergence rate is linear in space and quadratic in the time direction, as we have predicted  in Theorem \ref{J2}.\\
\begin{table}[htbp]
	\centering
	\caption{\emph{Error and Convergence Rates in space-time direction for $\alpha=0.25$  }}
	\begin{tabular}{|c|c|c|c|}
		\hline
		\textbf{Iteration No.} & \textbf{Error-2} & \textbf{Rate in Space} & \textbf{Rate in Time} \\
		\hline
		1 & 2.32e-02 & - & -  \\
		\hline
		2 & 1.10e-02& 1.0679 & 2.1359  \\
		\hline
		3 &  5.38e-03 & 1.0393 & 2.0787 \\
		\hline
		4 & 2.52e-03 & 1.0904 & 2.1809\\
		\hline
		5 & 1.24e-03 & 1.0225 & 2.0451\\
		\hline
	\end{tabular}%
	\label{table32}%
\end{table}
 \begin{table}[htbp]
	\centering
	\caption{\emph{Error and Convergence Rates in space-time direction for $\alpha=0.5$  }}
	\begin{tabular}{|c|c|c|c|}
		\hline
		\textbf{Iteration No.} & \textbf{Error-2} & \textbf{Rate in Space} & \textbf{Rate in Time} \\
		\hline
		1 & 2.42e-02 & - & -  \\
		\hline
		2 & 1.16e-02& 1.0573 & 2.1147  \\
		\hline
		3 &  5.70e-03 & 1.0329 & 2.0658 \\
		\hline
		4 & 2.68e-03 & 1.0873 & 2.1746\\
		\hline
		5 & 1.32e-03 & 1.0207 & 2.0414\\
		\hline
	\end{tabular}%
	\label{table322}%
\end{table}

\begin{table}[htbp]
	\centering
	\caption{\emph{Error and Convergence Rates in space-time direction for $\alpha=0.75$  }}
	\begin{tabular}{|c|c|c|c|}
		\hline
		\textbf{Iteration No.} & \textbf{Error-2} & \textbf{Rate in Space} & \textbf{Rate in Time} \\
		\hline
		1 & 1.97e-02 & - & -  \\
		\hline
		2 & 9.05e-03& 1.1284 & 2.2569  \\
		\hline
		3 &  4.25e-03 & 1.0873 & 2.1746 \\
		\hline
		4 & 1.94e-03 & 1.1335 & 2.2671\\
		\hline
		5 & 9.30e-04 & 1.0614 & 2.1228\\
		\hline
	\end{tabular}%
	\label{table323}%
\end{table}

 \noindent Now, we plot the graph of an approximate solution as well as an exact solution in figure 1 using backward Euler-Galerkin FEM based on the L1 type scheme.
\begin{figure}[H]
	\centering
\includegraphics[scale=0.6]{123}
\caption{Approximate solution(L.H.S) and Exact solution(R.H.S) at $T=1$  and $\alpha =0.5$.}
\end{figure}

\section{Conclusions}
 In this work, we have established the well-posedness of weak formulation for time fractional integro-differential equations of Kirchhoff type for non-homogeneous materials. We have derived error estimates for semi-discrete formulation by discretizing the domain in space direction using a conforming FEM and obtained convergence rate $O(h)$ in the energy norm. Further, to obtain the numerical solution  for this class of equations, we have developed and analyzed two different kinds of efficient numerical schemes. First, we derived that convergence rate is ($2$-$\alpha$) in the time direction, which is better than first order but not optimal. Next, to obtain  the optimal rate of convergence, we proposed a new Fractional Crank-Nicolson-Galerkin FEM based on the  L2-1$_{\sigma}$ scheme and proved that numerical solutions converge with the $O(h+k^{2})$ rate of accuracy. Finally, numerical experiments reveal that theoretical error estimates are sharp.

\end{document}